\documentclass[reqno]{amsart}
\usepackage[left = 3cm, right=3cm]{geometry}
\usepackage{amssymb, amsmath}
\usepackage{natbib}
\usepackage{lscape,comment}
\usepackage{color}
\usepackage{dsfont}
\usepackage{mathrsfs}
\usepackage{bm}
\usepackage{graphicx}
\usepackage{wrapfig}
\usepackage{caption}
\usepackage{tikz}
\usepackage{pgfplots}
\usetikzlibrary{arrows,decorations.markings,arrows.meta}
\usetikzlibrary{calc}

\usepackage[pdftex,plainpages=false,colorlinks,hyperindex,bookmarksopen,linkcolor=red,citecolor=blue,urlcolor=blue]{hyperref}
\DeclareMathAlphabet{\mathpzc}{OT1}{pzc}{m}{it}

\bibpunct{[}{]}{;}{n}{,}{,}

\newtheorem{thm}{Theorem}[section]

\newtheorem{rmk}[thm]{Remark}
\newtheorem{prop}[thm]{Proposition}
\newtheorem{lem}[thm]{Lemma}

\newtheorem{coro}[thm]{Corollary}

\numberwithin{equation}{section}

\def \ll { \left\lbrace }
\def \rr { \right\rbrace }
\def\centerarc[#1](#2)(#3:#4:#5)% Syntax: [draw options] (center) (initial angle:final angle:radius)
{ \draw[#1] ($(#2)+({#5*cos(#3)},{#5*sin(#3)})$) arc (#3:#4:#5); }

\DeclareMathOperator{\dP}{\mathds{P}}

\DeclareMathOperator{\N}{\mathbb{N}}
\DeclareMathOperator{\R}{\mathbb{R}}

\DeclareMathOperator{\cR}{\mathcal{R}}

\DeclareMathOperator{\cB}{\mathcal{B}}

\DeclareMathOperator{\D}{\mathbb{D}}

\newcommand{\pd}[2]{\frac{\partial #1}{\partial #2}}

\newcommand{\dersup}[3]{\frac{d^{#3} #1}{d #2^{#3}}}

\allowdisplaybreaks

\newcommand{\lb}{\left (}
\newcommand{\rb}{\right )}
\newcommand{\lbb}{\left [}
\newcommand{\rbb}{\right ]}
\newcommand{\labs}{\left |}
\newcommand{\rabs}{\right |}
\newcommand{\lbrb}[1]{\lb #1 \rb}
\newcommand{\lbbrbb}[1]{\lbb#1\rbb}

\newcommand{\lbbrb}[1]{\lbb#1\rb}
\newcommand{\lbrbb}[1]{\lb#1\rbb}

\newcommand{\lbcurly}{\left\{}
\newcommand{\rbcurly}{\right\}}

\newcommand{\abs}[1]{\labs#1\rabs}

\newcommand{\curly}[1]{\lbcurly#1\rbcurly}

\newcommand{\bo}[1]{\mathrm{O}\lbrb{#1}}
\newcommand{\so}[1]{\mathrm{o}\lbrb{#1}}

\newcommand{\Pbb}[1]{\Pb\lb #1\rb}
\newcommand{\Ebb}[1]{\Eb\lbb #1\rbb}

\newcommand{\LL}{L\'{e}vy }

\newcommand{\LLPs}{L\'{e}vy processes }

\newcommand{\Rez}{\Re\lbrb{z}}
\newcommand{\limi}[1]{\lim\limits_{#1\to \infty}}
\newcommand{\limsupi}[1]{\varlimsup\limits_{#1\to \infty}}
\newcommand{\liminfi}[1]{\varliminf\limits_{#1\to \infty}}

\newcommand{\limsupo}[1]{\varlimsup\limits_{#1\to 0}}
\newcommand{\liminfo}[1]{\varliminf\limits_{#1\to 0}}
\newcommand{\Cb}{\mathbb{C}}
\newcommand{\C}{\mathbb{C}}
\newcommand{\Eb}{\mathbb{E}}

\newcommand{\Rb}{\mathbb{R}}

\newcommand{\Pb}{\mathbb{P}}

\renewcommand{\P}{\mathbb{P}}

\newcommand{\Ic}{\mathcal{I}}

\newcommand{\ind}[1]{\mathbb{I}_{\{#1\}}}

\newcommand{\IntOI}{\int_{0}^{\infty}}
\newcommand{\IntII}{\int_{-\infty}^{\infty}}

\renewcommand{\limsupo}[1]{\varlimsup\limits_{#1\to 0+}}
\renewcommand{\liminfo}[1]{\varliminf\limits_{#1\to 0+}}

\newcommand{\minusone}{(-1)}
\definecolor{cadmiumgreen}{rgb}{0.0, 0.42, 0.24}

\newcommand{\ab}{a+ib}

\renewcommand{\Im}{\mathtt{Im}}
\renewcommand{\Re}{\mathtt{Re}}

\newcommand{\red}[1]{{\color{red}#1}}

\newcommand{\fP}[1]{f_{\Phi}\lbrb{#1}}
\renewcommand{\Phi}{\phi}
\renewcommand{\nu}{\mu}

\begin{document}
	
	\title[]{Regularity and asymptotics of densities of inverse subordinators}
	\author[]{Giacomo Ascione$^1$}
	\author[]{Mladen Savov$^{2,3}$}
	\author[]{Bruno Toaldo$^4$}
	\address[]{1: Scuola Superiore Meridionale, Largo S. Marcellino 10 - 80138, Napoli (Italy)}
	\email[]{g.ascione@ssmeridionale.it}
	\address[]{2: Faculty of Mathematics and Informatics, Sofia University ``St Kliment Ohridski'', boul. James Boucher 5, Sofia - 1164, Bulgaria}
	\email[]{msavov@fmi.uni-sofia.bg}
	\address[]{3: Institute of Mathematics and Informatics, Bulgarian Academy of Sciences, str. Akad. Georgi Bonchev, block 8, Sofia - 1113, Bulgaria}
	\email[]{mladensavov@math.bas.bg}
	\address[]{4: Dipartimento di Matematica ``Giuseppe Peano'' - Universit\`{a} degli Studi di Torino, Via Carlo Alberto 10 - 10123, Torino (Italy)}
	\email[]{bruno.toaldo@unito.it}
	\date{\today}

	\begin{abstract}
		In this article densities (and their derivatives) of subordinators and inverse subordinators are considered. Under minor restrictions, {generally} milder than the existing in the literature, employing a useful modification of the saddle point method, we obtain the large asymptotic behaviour of these densities (and their derivatives) for a specific region of space and time and quantify how the ratio between time and space affects the explicit speed of convergence. The asymptotics is governed by an exponential term depending on the Laplace exponent of the subordinator and the region represents the  behaviour of the subordinator when it is atypically small (the inverse one is larger than usual). As a result a route to the derivation of novel general or particular fine estimates for densities with explicit constants in the speed of convergence in the region of the lower envelope/the law of iterated logarithm is available. Furthermore, under mild conditions, we provide  a power series representation for densities (and their derivatives) of subordinators and inverse subordinators. This representation is explicit and based on the derivatives of the convolution of the tails of the corresponding L\'evy measure, whose smoothness is also investigated. In this context the methods adopted are based on Laplace inversion and strongly rely on the theory of Bernstein functions extended to the cut complex plane. As a result, smoothness properties of densities (and their derivatives) and their behaviour near zero immediately follow.
	\end{abstract}
	
	\maketitle

	\tableofcontents
	
	%GA: I think usually Introductions already contain motivations, so I do not think it is necessary to write it in the title.
\section{Introduction}\label{sec:intro}
Subordinators are L\'evy processes, i.e. c\'adl\'ag stochastic processes with stationary and independent increments, whose paths are almost surely non-decreasing. For this reason, they constitute a special class of \LL processes of fundamental interest in probability theory. In this paper, wherever possible, we consider potentially killed subordinators, that is to say that for any proper subordinator $\sigma:={\lbrb{\sigma(x)}_{x\geq 0}}$ we allow for killing at an exponentially distributed random variable $\mathbf{e}_q, q\geq0$ ($\mathbf{e}_0=\infty$) that is independent of $\sigma$,  by setting {$\sigma(x)=\infty$, provided $x \ge  \mathbf{e}_q$,} and keeping the original process until $\mathbf{e}_q$. With each potentially killed subordinator one defines the (right-)inverse process through the passage times $L:=\lbrb{L(t)}_{t \geq  0}$. In this work, our main aim is to offer a detailed study of the densities of $L$. In more detail, under mild and natural conditions, we prove that such densities are smooth in both variables and we provide for them and their derivatives a series representation and precise non-classical Tauberian asymptotics (with explicit speed of convergence). The latter is obtained by means of Laplace inversion, as, indeed, the Laplace transform of the density of $L(t)$ with respect to the variable $t$ admits a simple form in terms of the Laplace exponent of the subordinator $\sigma$. Furthermore, we also extend the results to densities (and their derivatives) of the subordinators themselves. Our results expand uniformly the existing knowledge on these quantities and allow for specific applications. We shall discuss the literature, offer motivation for our study  and outline our methodology below.

We first mention that, in recent years, the interest in subordinators and their hitting-times has grown fast and has reached an increasingly larger audience also outside the probability community, involving many areas of mathematics. This is due, in particular, to the connection with the so-called anomalous processes (or anomalous diffusion) and semi-Markov processes. Indeed, if we take a Markov process $X=\lb X(t) \rb_{ t \geq 0 }$ and consider the time-changed process $Y=\lb Y(t)\rb_{  t \ge 0}$, with $Y(t):=X(L(t))$, for any $t \ge 0$, then the process $Y$ has intervals of constancy, induced by the time-change, that are distributed according to the jumps of $\sigma$ (which are not necessarily exponentials). Under suitable assumptions, these time-changed processes are prototypes of semi-Markov processes (e.g., \cite{cinlarsemi, kaspi, meerstra, savtoa}). The importance of these processes arise in several applications in very different fields, among others: they are scaling limit of continuous time random walks (e.g., \cite{baemstra, meercoupled, meertri, meerstra}), they are useful to model anomalous diffusion and fractional kinetics (e.g., \cite{moving, beghin, gianni2, hairer, koko, kochukondra, FCbook, Metzler, gianni, cimp, savtoa, silva, umarov}), they appear in economics and mathematical finance (e.g., \citep{jacquier, scalas, scalastoaldo, lorenzo1}) and recently also in neuronal modelling (e.g., \cite{annals2020,LIF}).
Furthermore, the relation between these processes and solutions of some time-nonlocal equations is now well-established (see \cite{fracCauchy,zqc} for the most modern recent theory, \cite{FCbook} for a review of applications and \cite{gorevess} for analysis of fractional-type equations).
It is clear that in this context the one-dimensional distributions of the random variables $L(t)$ play a central role. For instance, if we assume that $X$ admits as a state space $\R^d$, for some $d \ge 1$, and we denote by $p(x,y,t)$ the transition densities of $X$, then one can show that
\begin{align*}
\cB(\R^d) \ni E  \mapsto \dP \lb Y(t) \in E \mid Y(0)=x \rb \, = \, \int_{E} \int_0^\infty p(x,y,s)f(s;t)ds  \, dy,
\end{align*}
%Since it is an illustrative example, maybe we can keep it really basic by assuming that the space is \R^d itself and removing some "demonstrative details". 
where, for any $t>0$, $f(s;t)$ is the density of $L(t)$ and $\cB(\R^d)$ is the Borel $\sigma$-algebra of $\R^d$. In practice, the time-changed process admits a density that can be written in terms of the one-dimensional distributions of $L(t)$. Hence, in order to determine some features of the one-dimensional distribution of $Y$, it may be necessary to rely also on some specific properties of the ones of $L(t)$ itself. Therefore, a more detailed study of the main features of the densities of inverse subordinators is needed. This point of view inspired several other works on this topic (e.g., \cite{fausto, kovacmeer, kumar, meerstra2, taqqu}). {The asymptotic behaviour of the density of subordinators and inverse subordinators already revealed to be a strong tool, for example, to study solutions of time-nonlocal equations. Two-sided estimates for densities of subordinators and inverse subordinators and their derivatives have been obtained, by different means, for instance in \cite{CKKW18,CKKW20} under restrictive scaling condition, where they have been employed to provide a two-sided bound for the fundamental solution of a time-nonlocal Poisson equation.} Similarly, small-time asymptotics of the derivative of the density of inverse subordinators has been employed in \cite{moving} to derive the regularity of the fundamental solution of a time-nonlocal heat equation with a moving boundary. Furthermore, asymptotics of the distributions of subordinators and inverse subordinators have been used in \cite{KP22,KP23} to deduce the spectral heat content of some time-changed processes, that play a role in the stochastic representation of the solution of time-nonlocal partial differential equations.

Two of the key objectives in the study of the one-dimensional distributions of random processes are to understand the asymptotic behaviour in time and space of their tails and to find representations for their densities and those of related quantities, e.g. such as passage times. This information is usually in the form of  explicit asymptotic terms at zero and/or at infinity, series expansions, integral representations, etc. In this work we provide results for densities (and their derivatives) of inverse subordinators and, via interchangeability, of subordinators. These include the derivation of precise, universal and explicit form of the large asymptotic behaviour with speed of convergence and general series expansions. The results go well beyond the current state-of-art which we briefly review below.

% Laplace in place of LK since we actually use this notation.
On the large asymptotics the first papers, see  \cite{FriPruitt71,JainPr_87}, offered results for lower tails of subordinators. They have been subsequently refined for densities  with the most up-to-date results contained in \cite{DonRiv,GrLTr21}. General upper and lower bounds with conditions in the spirit of \cite{DonRiv} are derived in \cite{GrLTr21}. Particular cases of estimates on densities for a class of subordinators can be found in the recent work \cite{ChoKim21}. Asymptotic results that are similar to ours (without speed of convergence) are contained in \cite{PatVai22} which deals with densities (and their derivatives) of spectrally negative \LLPs with necessarily positive Brownian component. {Overall, at asymptotic level, our results in Theorem \ref{thm:mainS} (i) relax the assumptions of the general \cite[Theorem 3.3]{GrLTr21} and \cite[Theorem 3.2 (iii)]{DonRiv}. In the setting of Theorem \ref{thm:mainS} (iii) our general condition is of different nature than those in \cite[Theorem 3.3]{GrLTr21} and \cite[Theorem 3.2 (iii)]{DonRiv}, whereas in the case of Theorem \ref{thm:mainS} (ii)  our condition is slightly more restrictive but easier to verify than \cite[Theorem 3.2 (iii)]{DonRiv}. However, in all cases we provide uniform results with explicit speed of convergence that deal not only with densities but with all their derivatives too. Besides, all results for densities of subordinators have their counterparts for densities of inverse subordinators. Detailed discussion is given in Section \ref{subsubsec:EX}. To derive our large asymptotic results we use the saddle point method, as applied in \cite{MinSav23+}, with a modification which allows us to measure the speed of convergence. We must emphasize that the very precise general asymptotic results with speed of convergence coming from an application of the classical saddle point method, see \cite{Olver68,Olver70}, are not directly applicable in our setting for two reasons: first the integrals we study do not have separation of variables as $z$ and $t$ in \cite[(1.0.1)]{Olver70} and second the contour we integrate on varies, whereas it is fixed in \cite[Condition (ii)]{Olver70}. Also, the steepest descent as described in \cite{Olver70} requires very precise knowledge on locations of real values of differences of Bernstein functions. This is usually unavailable. Continuing the discussion of our results note that the obtained speed of convergence depends on the} ratio of time and space, i.e. $t/x$, in $\Pbb{L(t)\in dx; \sigma(L(t))>t}=f(x,t)dx$, as it ranges between the drift of $\sigma$ and $\Ebb{\sigma(1)}$, and the closer $t/x$ is to the drift the faster the speed of convergence is. As a result, under conditions generally milder than the existing in the literature, we present in Theorems \ref{thm:mainL}, \ref{thm:main1}, \ref{thm:main2}, \ref{thm:mainS} explicit expressions for the aforementioned densities (and their derivatives) which are dominated by rather explicit exponential terms stemming from the Laplace exponent of $\sigma$. It is important to highlight once again that these representations are valid for  $t/x<\Ebb{\sigma(1)}$ and therefore capture non-typical slow growth of $\sigma$ or equivalently fast growth of $L$. For example, when the subordinator has a finite second moment, our results capture the region below that of the central limit theorem, see Section \ref{sec:ex} for more details. If $\Ebb{\sigma(1)}=\infty$ then our estimates capture the region of the lower envelope of subordinators and therefore one can obtain more precise local estimates for densities including explicit constants in the speed of convergence and hence furnish estimates for the probabilities that lead up to the law of the iterated logarithm as in \cite{bertoins,FriPruitt71}, see Section \ref{sec:ex}. One is also able to study particular classes of subordinators for which our main results can be further specialized. Also, given the nature of the saddle point method, we can derive under milder assumptions concrete bounds for fixed times as in \cite{ChoKim21,GrLTr21}. These are directions for further investigations. 

On representations of densities with series expansions the literature is mainly concerned with {subordinators and, in general, L\'evy processes}. Series representations for stable laws can be found, e.g., in \cite[Chapter XVII.6]{fellerbook}. In the more specific case of the stable subordinator, a series representation for its density has been provided in \cite{PG10}. It is clear that such a representation can be extended to the density of the inverse stable subordinator by means of the relation that links the two quantities, as for instance highlighted in \cite[Corollary 3.1]{meerschlim}. A similar result has been proved for the inverse tempered stable subordinator in \cite{kumar}, while a further integral representation for the density of the inverse gamma subordinator has been provided in \cite{kumar2} and further improved in \cite{fausto}.  To the best of our knowledge, these very last contributions are the only ones, that provide series representation for the density of (very specific) inverse subordinators. We remark that in both cases the results have been widely used, especially in applications regarding anomalous diffusion (see, e.g., \cite{metzlerbarkai} and references therein) and in the context of governing equations of time-changed processes (see, e.g., \cite{umarov} and references therein). Differently from the case of the inverse subordinators, series expansion and small-time asymptotics of general L\'evy processes have been widely studied. The more recent paper \cite{FH09} provides small-time polynomial expansions of the distribution of a general L\'evy process (later extended to some stochastic volatility models \cite{FH12,FO16}) under some technical assumptions concerning the regularity of its L\'evy measure and density. In \cite{bur}, the authors obtained explicit representations for some subordinators of the Thorin class. Despite in general such explicit representations do not provide power-series expansions, some of them can be still rewritten in this way, as we do in Subsection \ref{subsec:exseries}. In \cite{KK13} the authors provide some small-time bounds and ``bell-like'' estimates for the density of L\'evy processes under a technical condition on the characteristic function. Following the proof, one can observe that the authors employ a preliminary estimate in terms of a compound kernel estimate, that is indeed expressed in terms of a series of convolution powers.  
%We recall again that, except for some specific cases, we are not aware of any extension of such results to inverse subordinators or first passage times of L\'evy processes.

Here, under mild assumptions on the Laplace exponent of the subordinator, we obtain an explicit power series representation for the density (and its derivatives) of the inverse subordinator, in Theorem \ref{thm:seriespi}, and of the subordinator itself, in Theorem \ref{thm:seriessub}. We highlight that our power series representation holds, in particular, when the Laplace exponent of $\sigma$ is a complete Bernstein function, thus covering also the cases discussed in \cite{kumar,PG10}, for which comparison is carried on in Subsection \ref{subsec:exseries}. The assumption on the completeness of the Laplace exponent is sufficient, but not necessary, as the power-series representation can be applied on a wider class of subordinators, as discussed in Section \ref{discussionassumptions}.  From the point of view of polynomial expansion, we use our series representation to obtain a result for inverse subordinators which is similar to that in \cite{FH09} mentioned above: precisely we provide the asymptotics for small $x$, uniformly for $t$ in compact sets, for the density of inverse subordinators and its derivatives (under quite general and easy-to-check conditions).
%For our new results we use Laplace inversion method and the theory of Bernstein functions extended to the complex plane. Under very mild conditions, we obtain a representation as a power series for the density of subordinators and inverse subordinators (and their derivatives)  (see Theorems \ref{thm:seriespi}, \ref{thm:seriessub} and Corollary \ref{cor:seriescreep}) and, as an immediate consequence, we obtain the behaviour at zero of the densities of inverse subordinators and their derivatives (see Theorem \ref{behavatzero}). In particular, our method works whenever the Laplace exponent of the subordinator can be analytically extended to the complex plane with negative real part and non-zero imaginary part (see Section \ref{discussionassumptions} for a thorough discussion on the generality of the used assumptions).
%Later, in Section \ref{sec:comparisonseries} below we offer a comparison of our results with the ones in \cite{kumar, PG10}.

As already mentioned, both the asymptotic behaviour and the power-series representation are obtained via Laplace inversion. On the one hand, the assumptions used to derive the asymptotic behaviour first lead to the absolute convergence of the integral involved in the Laplace inversion (see \cite[Theorem $4.1.21$]{abhn}). On the other hand, the assumptions adopted to obtain the power-series representation allow us to consider a suitable keyhole-type contour, on which the inversion integral becomes absolutely convergent. Furthermore, as a consequence, we get the smoothness of the involved densities starting from any of the two sets of assumptions. Finally, let us remark that the proofs of the main results presented in Section \ref{sec:mainResults} are given separately in Sections \ref{sec:proofs} and \ref{subsec:power}, while {Section \ref{sec:ex} offers examples and further discussion for the two sets of results.}
	\section{Preliminaries}\label{sec:prelim}

We assume everywhere that we work on the standard (for subordinators) probability space, namely the space of c\'adl\'ag functions on $\lbbrbb{0,\infty}$ endowed with the Skorohod topology, the sigma-algebra generated by it and a suitable  probability measure on the latter. Also, because our paper focuses on the inverse subordinators we use $t$ for its time variable and $x$ for its space variable. This choice of notation imposes that for the subordinator $t$ is the space variable and $x$ is the time variable, which is unusual, but still in line with the main focus of the paper.

Let {$\sigma=\lbrb{\sigma(x)}_{x \ge 0}$ be a potentially killed one-dimensional subordinator, that is to say that there exists a conservative subordinator $\sigma^{\vartriangle}:=\lbrb{\sigma^{\vartriangle}(x)}_{x \ge 0}$ and an independent exponential random variable $\mathbf{e}_q$ with parameter $q \ge 0$ (where, if $q=0$, we set $\mathbf{e}_0=\infty$) such that $\sigma(x)=\sigma^{\vartriangle}(x)$ for any $x<\mathbf{e}_q$, while $\sigma(x)=\infty$ for any $x \ge \mathbf{e}_q$.}
  Each $\sigma$ is uniquely determined (in law) by a Bernstein function in the following manner
\begin{equation}\label{def:Phi}
	\begin{split}
		-\log\Ebb{e^{-z\sigma(1)}}=\phi(z)&=q+\mathfrak{b}z+\IntOI \lbrb{1-e^{-zy}}\mu_\phi(dy)\\
		&=q+\mathfrak{b}z+z\IntOI e^{-zy}\bar{\mu}_\phi(y)dy,\,\,\Re(z)\geq 0,
	\end{split}
\end{equation}
where $q,\mathfrak{b}\geq 0$, $\mu_\phi$ is a Radon measure on $\lbrb{0,\infty}$ satisfying  $\IntOI \min\curly{y,1}\mu_\phi(dy)<\infty$ (called the L\'evy measure of $\sigma$) and $\bar{\mu}_\phi(t)={\mu_{\phi}\lbrb{\lbrb{t,\infty}}}=\int_{t}^{\infty}\mu(dy),t\geq 0,$ is the tail of $\mu_\phi$.
%sigma-finite measure on $\lbbrb{0,\infty}$ with $\mu_\phi(\curly{0})=0$  
In fact the right-hand side of \eqref{def:Phi} serves as an equivalent definition of Bernstein functions and thus they are in  bijection with the potentially killed subordinators. Due to their importance in other areas of mathematics, Bernstein functions have been studied in detail. An exposition on the current knowledge of their properties can be found in \cite{librobern} and some additional information is scattered in references such as \cite{BivBernGam,bernsteingamma,Laguerre}. Classical references for \LLPs and subordinators, in particular, are the books \cite{bertoinb,bertoins,kyprianou}. 

With each $\sigma$ one defines the (right-)inverse-subordinator  $L:=\lbrb{L(t)}_{t \ge 0}$ via the passage times
\begin{equation}\label{def:invS}
	L(t):={\inf\curly{x>0:\sigma(x)>t}}.
\end{equation}
Note that $\Pbb{L(t)=\infty}=0$ for every $t \in [0, \infty)$, even if $q>0$, since for $t>\sigma(\mathbf{e}_q-)$ the process $L(t)$ remains stuck in the position $\mathbf{e}_q$.
Also, note that the paths of $L$ are almost surely continuous if and only if $\mathfrak{b}\neq 0$ or $\bar{\mu}_\phi(0)=\infty$, that is $\sigma$ is not a pure-jump compound Poisson process. When $q=\mathfrak{b}=0$ and $\bar{\mu}_\phi(0)=\infty$, it is well-known, see \cite[Theorem 3.1]{meertri}, that, for any $t>0$, $L(t)$ admits density on $(0,\infty)$  and the requirement $q=0$ is immediately seen to be unnecessary. Indeed, it is not hard to check, by a conditioning argument using  the independence between $\mathbf{e}_q$ and $\sigma^{\vartriangle}$ that
\begin{align}
\P \lb L(t) \leq x \rb \, = \, 1-e^{-qx} + e^{-qx} \P \lb L^\Delta(t) \leq x  \rb=1-e^{-qx}\Pbb{\sigma^\vartriangle(x)\leq t},
\label{distr}
\end{align}
where $L^\vartriangle$ stands for the inverse of $\sigma^\vartriangle$. Furthermore, if $q=0$ then  $\sigma^\vartriangle=\sigma$ and $L^\vartriangle=L$, and the measure $\Pbb{L^\vartriangle(t)\in dx, \ \sigma^\vartriangle(L^\vartriangle(t))>t}$ has a density on $(0,t/\mathfrak{b})$, where we set $t/0=\infty$, (see \cite[Lemma 11]{DonRiv2} and \cite[prior to Theorem 2.3]{DonRiv}), given by the relation
\begin{equation}\label{def:ftriangle}
	\begin{split}
		&\Pbb{L^\vartriangle(t)\in dx, \sigma^\vartriangle(L^\vartriangle(t))>t}=\int_{0}^t \bar{\nu}_{\Phi}(t-y)\Pbb{\sigma^\vartriangle(x)\in dy}dx=:f_{\Phi}^\vartriangle (x,t)dx.
	\end{split}
\end{equation}
However, since $\Pbb{\sigma^\vartriangle(L^\vartriangle(t))>t}<1$ if and only if $\mathfrak{b}>0$, see \cite[Propositions 1.7 and 1.9]{bertoins}, we observe that if $\mathfrak{b}=0$ then, for $t>0$, $f_\phi^\vartriangle(\cdot,t)$ is the density of $L^\vartriangle(t)$. Otherwise, if $\mathfrak{b}>0$, $f_\phi^\vartriangle(\cdot,t)$  stands for the density of $L^\vartriangle(t)$ on the event that $\sigma^\vartriangle$ jumps across $t$. On the event that $\sigma^\vartriangle$ creeps up across $t$, i.e. on $\{\sigma^\vartriangle(L^\vartriangle(t))=t\}$, $L^\vartriangle(t)$ does not necessarily admit a density. However, it has been shown in the proof of \cite[Lemma 3.3]{DonRiv}, that if $\sigma^\vartriangle(x)$ admits a density $g_\phi^\vartriangle(x,\cdot)$ on $(\mathfrak{b}x,\infty)$, then, on the event $\{\sigma^\vartriangle(L^\vartriangle(t))=t\}$, $L^\vartriangle (t)$ admits a density in $(0,t/\mathfrak{b})$ given by
%
%
%We denote the creep part of the density as
\begin{align}
f_\phi^{c,\vartriangle}(x,t)dx:= \P \lb L^\vartriangle (t) \in dx, \sigma^\vartriangle(L^\vartriangle(t))=t \rb=\mathfrak{b}g_\phi^\vartriangle(x,t)dx.
\label{def:fctriangle}
\end{align}
%and using the notation
%\begin{align}
%g_\phi^\vartriangle(x,t) dt \, := \, \P \l \sigma^\vartriangle(x) \in dt \rb,
%\label{def:gphitriangle}
%\end{align}
%whenever such a density exists, we know from \cite{DonRiv} that 
%$f^{c,\vartriangle}_{\Phi}(x,t)=\mathfrak{b}g_\phi^\vartriangle(x,t)$.
Plugging \eqref{def:ftriangle} and \eqref{def:fctriangle} into \eqref{distr}, we get on {$[0,t/\mathfrak{b})$}
\begin{align}
\P \lb L(t) \leq x \rb \, = \, 1-e^{-qx} + e^{-qx} \int_0^x \left[ \int_0^t \bar{\nu}_\phi(t-y)g_\phi^\vartriangle(s,y)dy + \mathfrak{b}g_\phi^\vartriangle (s,t) \right] \, ds.
\label{distrexpl}
\end{align}
Furthermore, if $\sigma^\vartriangle(x)$ admits a density $g_\phi^\vartriangle(x,\cdot)$ on $(\mathfrak{b}x,\infty)$, so does $\sigma(x)$ with  density given by
\begin{align}
	g_\phi(x,t)dt \, := \, \P \lb \sigma(x) \in dt \rb \, = \, e^{-qx} g_\phi^\vartriangle (x,t)dt.
	\label{def:gphi}
\end{align}
Differentiating \eqref{distrexpl} in $x$, we get that $L(t)$ admits a density on $(0,t/\mathfrak{b})$, that is given by
\begin{align}
	\P \lb L(t) \in dx \rb \, = \, &q \P \lb \sigma(x) \leq t \rb dx &=:f_\phi^{\rm k}(x,t)dx \label{def:fk} \\
	&+ \left(\int_0^t \bar{\nu}_\phi(t-y)g_\Phi(x,y)dy\right) dx &=:f_\phi(x,t)dx \label{def:f}  \\
	&+\mathfrak{b} g_\phi(x,t)dx &=:f_\phi^{\rm c}(x,t)dx,\label{def:fc}
\end{align}
where we have used \eqref{distr} to identify the double integral in \eqref{distrexpl}.
The functions $f_\phi$, $f_\phi^{\rm k}$ and $f_\phi^{\rm c}$ are defined on the set
\begin{equation}\label{def:Reg}
	\begin{split}
		&\mathbb{D}:=\curly{\lbrb{t,x}\in\Rb^2:\,t>0, 0<x<\frac{t}{\mathfrak{b}}}
	\end{split}
\end{equation}
and then extended to $0$ outside $\mathbb{D}$.

Our main results focus on the quantities in \eqref{def:gphi} and \eqref{def:f}, but the extension of such properties to \eqref{def:fk} and \eqref{def:fc} will be clear.
%Our main results concern the quantities in \eqref{def:gphi} and \eqref{def:f}. They will be applied to \eqref{def:fk} and \eqref{def:fc}.

Despite the fact that the formulations of $f_\phi^{\rm k}$, $f_\phi$ and $f_\phi^{\rm c}$ in \eqref{def:fk}, \eqref{def:f} and \eqref{def:fc} are quite implicit, their Laplace transforms in the variable $t$ can be expressed in terms of $\Phi$ in a simple way. Indeed, for any $z \in \C$ with $\Re(z)>0$, we have
\begin{equation}\label{eq:LT0}
		\int_{0}^\infty f^{\rm k}_\Phi(x,t)e^{-z t}dt=\frac{q}{z}e^{-x\Phi(z)}, \qquad \qquad \int_{0}^\infty f^{\rm c}_\Phi(x,t)e^{-z t}dt=\mathfrak{b}e^{-x\Phi(z)}
\end{equation}
and
\begin{equation}\label{eq:LT1}
	\int_{0}^\infty f_\Phi(x,t)e^{-z t}dt=\frac{\Phi^\dagger\lbrb{z}}{z}e^{-x\Phi(z)},
\end{equation}	
where
\begin{equation}\label{def:Pdag}
	\begin{split}
		&\Phi^\dagger(z):=\Phi(z)-q-\mathfrak{b}z.
	\end{split}
\end{equation}
%is the unkilled and driftless version of \eqref{def:Phi}.
%	 While the Lapace tranforms of $f_\phi^{\rm k}(x,\cdot)$ and $f_\phi^{\rm c}(x,\cdot)$ By the convolution theorem of Laplace transform (see \cite[Proposition 1.6.4]{abhn}) it is clear that the Laplace transform in time of $f_{\Phi}(x,t)$ is given by 
%\begin{equation}\label{eq:LT1}
%	\int_{0}^\infty f_\Phi(x,t)e^{-z t}dt=\frac{\Phi^\dagger\lbrb{z}}{z}e^{-x\Phi(z)},\,\,\Re\lbrb{z}\geq 0,
%\end{equation}
%where 
Expression \eqref{eq:LT1} is the starting point of our study. 

Let us now fix some notation. Here and hereafter, we use $\Cb$ for the complex plane and we write $z=a+ib=\Re(z)+i\Im(z)$. For any $a \in \R$  we set $\mathbb{H}_a:=\{z \in \mathbb{C}: \ \Re z>a\}$. For any  $\alpha, \beta \in (-\pi, \pi]$ with $\beta<\alpha$ we put
\begin{equation*}
	\C(\alpha,\beta):=\{z \in \C: \ \beta<{\rm Arg}(z)<\alpha, |z|>0\}
\end{equation*}
and $\C(\alpha):= \{z \in \C: |{\rm Arg}(z)|<\alpha, |z|>0\}$ when $\alpha>0$.

Throughout the paper, we use $C$ to denote any constant whose value is inessential. If needed we underline the dependence  of $C$ on some parameters $p_1,p_2,\dots$ by using  $C(p_1,p_2,\dots)$.

We use $\so{\cdot},\bo{\cdot}$ in the standard fashion with, e.g. $\so{g(x)}$, as $x \to a$, denoting a generic function $f$ such that $\lim_{x \to a}|f(x)|/|g(x)|=0$, while $\bo{g(x)}$, as $x \to a$, meaning a generic function $f$ with $\limsup_{x \to a}|f(x)|/|g(x)|<\infty$. The same notation is also reserved for functions on regions of $\Cb$. Furthermore, we say that $f \asymp g$, as $x \to a$, if $f=\bo{g(x)}$ and $g=\bo{f(x)}$, as $x \to a$.
For any two functions $f,g:[0,+\infty) \to \R$, we denote the convolution product of $f$ and $g$ as
\begin{equation*}
	(f \ast g)(t)=\int_0^t f(s)g(t-s)ds, \ t \ge 0.
\end{equation*}
Furthermore, we denote the convolution powers as
\begin{equation*}
	f^{\ast 0}(t)=\delta_0 \qquad f^{\ast 1}(t)=f(t) \qquad f^{\ast n}(t)=(f \ast f^{\ast (n-1)})(t), \ n \ge 2.
\end{equation*}

The next lemma collects well-known properties of Bernstein functions used throughout this paper.
\begin{lem}\label{lem:Bern}
	Let $\phi$ be a non-zero Bernstein function. Then
	\begin{enumerate}
		\item\label{it:phi} $\phi$ is non-decreasing on $\lbbrb{0,\infty}$ with $\phi(\infty)=\infty\iff \bar{\mu}_\phi(0)=\infty\text{ or }\mathfrak{b}>0$;
		\item \label{it:sign} for any $z \in \overline{\mathbb{H}_0}$ we have that $\Re \Phi(z) \geq 0$ and \begin{equation}\label{eq:Re>}
		  \Re(\phi(z))\geq \phi(\Rez);
		 \end{equation}
		\item\label{it:phi'} $\phi'$ is completely monotone on $(0,\infty)$, i.e., for all $n\geq 1$,  $\minusone^{n-1}\phi^{(n)}(x)\geq 0$ on $(0,\infty)$, and $\limi{x}\phi'(x)=\mathfrak{b}$; as a result $\phi''<0$ and $\phi'''>0$ on $(0,\infty)$ and $\phi'(0^+)<\infty\iff \IntOI y\mu_\phi(dy)<\infty$;
		\item\label{it:ineq} for all $x>0$ it holds that $x\phi'(x)\leq \phi(x)$ and $-x^2\phi''(x)\leq 2\phi(x)$;
		\item\label{it:real} for any $z\in \mathbb{H}_0$ and any $n\geq 1$ it holds that $\abs{\phi^{(n)}(z)}\leq \abs{\phi^{(n)}\lbrb{\Re(z)}}$;
		\item\label{it:asymp} for any $\phi$ it holds that $\abs{\phi(z)}=\mathfrak{b}|z|\lbrb{1+\so{1}}$, as $z \to \infty$ uniformly on $\overline{\mathbb{H}_0}$, and if $\phi(0)=q>0,\phi(\mathfrak{u}_\phi)=0$, where $\mathfrak{u}_\phi$ is the first zero of $\phi$ on $(-\infty,0)$ and $\phi$ extends analytically (continuously at the boundary) at least to $\overline{\mathbb{H}_{\mathfrak{u}_\phi}}$ then  $\abs{\phi(z)}=\mathfrak{b}|z|\lbrb{1+\so{1}}$, as $z \to \infty$ uniformly on the same region $\overline{\mathbb{H}_{\mathfrak{u}_\phi}}$;
		\item\label{it:compos} if $\phi_1,\phi_2$ are two Bernstein functions, then $z \in [0,+\infty) \mapsto \phi_1(\phi_2(z))) \in \R$ is a Bernstein function;
		\item\label{it:limit} if $(\phi_n)_{n \ge 0}$ is a sequence of Bernstein functions and $\phi:[0,+\infty) \to \R$ is such that $\lim_{n \to \infty}\phi_n(z)=\phi(z)$ for any $z>0$, then $\phi$ is a Bernstein function;
		\item\label{it:imIq} for any $a>0, b\in\Rb$ it holds that
		 \begin{equation}\label{eq:imIq}
		 \qquad\frac{\abs{\Im\lbrb{\Phi\lbrb{a\lbrb{1+ib}}}}}{\Phi(a)}\leq \abs{b}a\frac{\Phi'(a)}{\Phi(a)}\leq |b|,\quad \frac{\Re\lbrb{\Phi\lbrb{a\lbrb{1+ib}}}-\Phi(a)}{\Phi(a)}\leq \frac{b^2a^2}{2}\frac{-\Phi''(a)}{\Phi(a)}\leq b^2.
		 \end{equation}

	\end{enumerate}
\end{lem}
\begin{rmk}\label{rem:Bern}
	Items \eqref{it:phi}, \eqref{it:phi'}, \eqref{it:compos} and \eqref{it:limit} are standard and can be found in \cite{librobern}. Item \eqref{it:sign} is \cite[Proposition 3.1, Item (9)]{bernsteingamma}. Item \eqref{it:ineq} can be located in \cite[(3.3) of Proposition 3.1]{bernsteingamma}, whereas Item \eqref{it:real} is contained in \cite[(3.11) of
	Proposition 3.3]{BivBernGam}. Item \eqref{it:asymp} is taken from \cite[Proposition 3.1, Item (4)]{bernsteingamma}. Item \eqref{it:imIq} is the content of  \cite[(4.25), (4.26)]{Laguerre} combined with $a\Phi'(a)\leq \Phi(a)$ and $a\abs{\Phi''(a)}\leq 2\Phi(a)$ from Item \eqref{it:ineq}.
\end{rmk}
We will always make use of the following property.
\begin{prop}\label{prop:D2}
	For any Bernstein function $\Phi$ and for any $a>0,b\in\Rb$, it holds true that
	\begin{equation}\label{eq:DeltaR}
		\begin{split}
			&\abs{\frac{\Phi\lbrb{a\lbrb{1+ib}}}{\Phi(a)}}\leq 3\max\curly{1,b^2}.
		\end{split}
	\end{equation}
\end{prop}
\begin{proof}
	Using $|1+z|\leq 1+\abs{\Re(z)}+\abs{\Im(z)}$ and \eqref{eq:Re>} we get
	\begin{equation*}
		\begin{split}
			\abs{\frac{\Phi\lbrb{a\lbrb{1+ib}}}{\Phi(a)}}&=	\abs{1+\frac{\Re\lbrb{\Phi\lbrb{a\lbrb{1+ib}}}-\Phi(a)}{\Phi(a)}+i\frac{\Im\lbrb{\Phi\lbrb{a\lbrb{1+ib}}}}{\Phi(a)}}\\
			&\leq 
			1+\frac{\Re\lbrb{\Phi\lbrb{a\lbrb{1+ib}}}-\Phi(a)}{\Phi(a)}+ \frac{\abs{\Im\lbrb{\Phi\lbrb{a\lbrb{1+ib}}}}}{\Phi(a)}.
		\end{split}
	\end{equation*}
	This together with Item \eqref{it:imIq} of Lemma \ref{lem:Bern} give \eqref{eq:DeltaR}.
\end{proof}

In this paper and its examples we use the class of complete Bernstein functions. Recall that a Bernstein function $\phi$ is said to be complete if the L\'evy measure $\mu_\Phi$ admits a completely monotone density. The next lemma collects some well-known facts on complete Bernstein functions.
\begin{lem}\label{lem:CBern}
	The following properties hold true:
	\begin{enumerate}
		\item \label{it:analy1} A non-negative function $\phi:(0,+\infty) \to [0,+\infty)$ is a complete Bernstein function if and only if it admits an analytic continuation on $\C \setminus (-\infty,0]$ (that we still denote $\phi$) such that $\Im(z)\Im(\Phi(z))\ge 0$ for any $z \in \C \setminus (-\infty,0]$ and such that $\lim_{(0,+\infty)\ni z \to 0}\phi(z)$ exists;
		\item \label{it:analy2} A non-negative function $\phi:(0,+\infty) \to [0,+\infty)$ is a complete Bernstein function if and only if it admits an analytic continuation on $\C(0,\pi)$ (that we still denote $\phi$) such that $\Im(\Phi(z))\ge 0$ for any $z \in \C(0,\pi)$ and such that $\lim_{(0,+\infty)\ni z \to 0}\phi(z)$ exists;
		\item\label{it:composC} if $\phi_1,\phi_2$ are two complete Bernstein functions, then $z \in [0,+\infty) \mapsto \phi_1(\phi_2(z))) \in \R$ is a complete Bernstein function;
		\item\label{it:limitC} if $(\phi_n)_{n \ge 0}$ is a sequence of complete Bernstein functions and $\phi:[0,+\infty) \to \R$ is such that $\lim_{n \to +\infty}\phi_n(z)=\phi(z)$ for any $z>0$, then $\phi$ is a complete Bernstein function;
		\item\label{it:anglim} if $\phi$ is a complete Bernstein function, then for any $\alpha \in (0,\pi)$ one has $\lim_{\C(\alpha) \ni z \to \infty}\frac{\Phi(z)}{z}=\mathfrak{b}$.
	\end{enumerate}
\end{lem}
\begin{rmk}\label{rem:BernC}
	All the items of the previous lemma can be found in \cite[Chapter $6$ and $7$]{librobern}.
\end{rmk}
We give here a further characterization of complete Bernstein functions, which is a direct consequence of the previous lemma.
\begin{prop}\label{prop:powchar}
	Let $\phi$ be a Bernstein function and denote for any $\alpha \in (0,1)$, $\phi_\alpha(z):=\phi(z^\alpha)$ for $z \ge 0$. Then the following two properties are equivalent:
	\begin{itemize}
		\item[(i)] $\phi$ is a complete Bernstein function;
		\item[(ii)] There exists a sequence $(\alpha_n)_{n \ge 0}$ in $(0,1)$ with $\alpha_n \to 1$ such that $\phi_{\alpha_n}$ is a complete Bernstein function for any $n \in \N$.
	\end{itemize}
\end{prop}
\begin{proof}
	Clearly, (i) implies (ii) by Item \eqref{it:composC} of Lemma \eqref{lem:CBern}. To show that (ii) implies (i), observe that $\phi(z)=\lim_{n \to +\infty}\phi_{\alpha_n}(z)$ and then we obtain the desired result by Item \eqref{it:limitC} of Lemma \ref{lem:CBern}.
\end{proof}
	\section{Main results concerning densities and their derivatives}\label{sec:mainResults}
In this section we present results concerning $\fP{x,t}$, see \eqref{def:f}, which in the case of $\mathfrak{b}=q=0$ is the (entire) density of $L(t)$. The results are in two directions - series expansion and behaviour at zero, and large asymptotics together with speed of convergence for $\fP{x,t}$ and its derivatives. We complement the latter results by providing information about $f_\phi^k(x,t)$ and $f_\phi^c(x,t)$ (see respectively \eqref{def:fk} and \eqref{def:fc}). We also furnish results for the density of $\sigma(x)$ and its derivatives. 
 
 \subsection{Behaviour at infinity of densities and their derivatives of the inverse subordinator and of the subordinator itself}\label{subsec:L}
 
 In this part we offer results concerning the large asymptotic behaviour of densities (and their derivatives) of inverse subordinators and subordinators themselves. Since, via the saddle point method, we prove the results for all derivatives simultaneously, we impose a condition which may not be optimal but is general enough and easy to implement. For this purpose, we set
 \begin{equation}\label{def:D}
 	\begin{split}
 		&\Delta(x):=\int_0^{\frac1x}y^2\nu_\Phi(dy),\,x>0,
 	\end{split}
 \end{equation}
 for the truncated second moment of the unkilled and driftless version of $\sigma$. Then, the main condition in our work is 
 \begin{equation}\label{def:condiA}
 	\begin{split}
 		&\liminfi{x}\frac{x^2\Delta(x)}{\ln(x)}=L\in\lbrbb{0,\infty}.
 	\end{split}\tag{$\mathbb{A}_1$}
 \end{equation}
The mild requirement
\begin{equation}\label{def:condiB}
	\begin{split}
		&\limsupi{x}\frac{x\Phi'''(x)}{-\Phi''(x)}=K<\infty
	\end{split}\tag{$\mathbb{A}_2$}
\end{equation}
plays a role only in some results. We discuss these conditions in Section \ref{subsubsec:EX}, where we present alternative formulation of \eqref{def:condiB} and demonstrate that they are generally milder than those in the literature. The proofs of all the results in this section are given in Section \ref{sec:proofs}.
First of all, let us underline that condition $\eqref{def:condiA}$ is related to the regularity of the function $f_\Phi$ defined in \eqref{def:f} with domain $\mathbb{D}$, see \eqref{def:Reg}, which is the content of our first result.
\begin{thm}\label{thm:regularityfphi1}
	Let $\Phi$ be the Laplace exponent of some potentially killed subordinator and assume that condition $\eqref{def:condiA}$ holds. Then, for any $n \ge 0$, there exists $x_0(n,L) \ge 0$ such that, for any $k,l \ge 0$ with $k+l \le n$, and for any $x>x_0(n,L)$ and $(x,t) \in \mathbb{D}$, $\displaystyle \frac{\partial^k}{\partial x^k}\frac{\partial^l}{\partial t^l}f_\phi(x,t)$ is well-defined and for any $a>0$
	\begin{equation}
		\frac{\partial^l}{\partial t^l} \frac{\partial^k}{\partial x^k} f_\Phi(x,t) \, = \, (-1)^k \int_{-\infty}^{+\infty} \frac{\Phi^\dagger (a+ib) (\Phi(a+ib))^k }{(a+ib)^{1-l}} e^{t(a+ib)-x\Phi(a+ib)} db,
		\label{intreprder1}
	\end{equation}
	where the integral is absolutely convergent. If $L=\infty$ in \eqref{def:condiA}, then for any $n \geq 0$, $x_0(n,\infty)=0$.
%	
%	 above is infinitely differentiable in both variables on $\mathbb{D}$ and, if $L<\infty$, then, for any fixed $k,l\geq 0$, $\frac{\partial^k}{\partial x^k} \frac{\partial^l}{\partial t^l}f_\phi(t,x)$  exists for points in $\mathbb{D}$, for which $x$ is large enough.
\end{thm}
Theorem \ref{thm:regularityfphi1} 
%is proved in Section \ref{sec:proofs} and it 
is needed for the study of the asymptotic behaviour of $f_\Phi$. The next theorem is the first result in this direction and considers the behaviour of $f_\Phi(x,t)$ when $t/x\downarrow \mathfrak{b},$ as $x\to\infty$. 
%It is proved in Section \ref{sec:proofs}.
\begin{thm}\label{thm:mainL}
	Let $\Phi$ be the Laplace exponent of some potentially killed subordinator and assume that  conditions $\eqref{def:condiA}$ and $\eqref{def:condiB}$ hold true. Let also $t(x)$ be such that $t(x)/x\in\lbrb{\mathfrak{b},\Phi'(0^+)}$ and  $\limi{x}t(x)/x=\mathfrak{b}$. Consider 
	\begin{equation}\label{eq:a*}
		a_*:=a_*(x)=(\phi')^{-1}\lbrb{\frac{t(x)}{x}}\in\lbrb{0,\infty},
		\end{equation}
		that is well-defined since $\phi'$ is decreasing. Define also the set
		\begin{equation}\label{def:region}
			\mathbb{D}^\prime=\{(t,x): \ x\mathfrak{b} < t \le t(x) < x\phi'(0+)\}
		\end{equation}
		and, for $(t,x) \in \mathbb{D}^\prime$, let $c:=c(t,x)=(\phi')^{-1}(t/x)\geq a_*$. Then, for any $k\geq0,l\geq0$, as $x\to\infty$,
	\begin{equation}\label{asymp}
		\begin{split}
			\sup_{x\mathfrak{b}<t\le t(x)}\abs{(-1)^k\sqrt{2\pi}\frac{c^{1-l}\sqrt{-\Phi''(c)x}}{\Phi^\dagger(c)\Phi^{k}(c)}e^{-ct+x\Phi(c)}\frac{\partial^k\partial ^l }{\partial x^k\partial t^l}\fP{x,t}-1}&=\bo{\sup_{c\geq a_*(x)}\frac{\sqrt{\ln\lbrb{c\sqrt{-\Phi''(c)x}}} }{c\sqrt{-\Phi''(c)x}}}\\
			&=\bo{\sqrt{\frac{\ln(x)}{x\ln (a_*(x))}}}.
		\end{split}
	\end{equation}
	%\begin{equation}\label{asymp}
	%	\begin{split}
	%		 \sup_{c\geq a_*}\abs{(-1)^k\sqrt{2\pi}\frac{c^{1-l}\sqrt{-\Phi''(c)x}}{\Phi^\dagger(c)\Phi^{k}(c)}e^{-ct+x\Phi(c)}\frac{\partial^k\partial ^l }{\partial x^k\partial t^l}\fP{x,t}-1}&=\bo{\sup_{c\geq a_*}\frac{\sqrt{\ln\lbrb{c\sqrt{-\Phi''(c)x}}} }{c\sqrt{-\Phi''(c)x}}}\\
	%		&=\bo{\sqrt{\frac{\ln(x)}{x\ln a_*}}}.
	%	\end{split}
	%\end{equation}
	Furtherore, $1/a_*=\so{t(x)/x-\mathfrak{b}}$, as $x \to \infty$.
	% all $x$ large enough. 
\end{thm}
\begin{rmk}\label{rem:mainL} In Section \ref{subsubsec:EX} we show that \eqref{def:condiA} and \eqref{def:condiB} are generally milder than those in \cite{DonRiv}. $\eqref{def:condiB}$ holds for some Compound Poisson processes, e.g., when $\nu_\Phi(dy)=e^{-y}dy$ and cannot be deduced from  \eqref{def:condiA} which implies $\Phi(\infty)=\infty$. \eqref{def:condiA} often holds with $L=\infty$ as in the example in Section \ref{subsubsec:EX}. 
\end{rmk}
\begin{rmk}\label{rem:mainL1}
For this theorem and all others concerning the large asymptotic behaviour one can laboriously track  the constants in the speed of convergence, see \eqref{asymp}, and thereby obtain strict upper bounds on the densities and its derivatives in line with  \cite{ChoKim21,GrLTr21}. For lower bounds the saddle point method is expected to work too. For more information see the discussion in Section \ref{subsubsec:EX}.
\end{rmk}
Next, for the sake of clarity, we formulate a corollary which deals with the most usual case, i.e. when $\mathfrak{b}=q=0$ and $L=\infty$ in \eqref{def:condiA}.
\begin{coro}\label{cor:mainL}
	Let $\phi$ be a Laplace exponent of a subordinator with $\mathfrak{b}=q=0$,  let \eqref{def:condiB} hold and \eqref{def:condiA} be valid with $L=\infty$. Then $f_\Phi(x,t)$ is infinitely differentiable on $\Rb^+\times \Rb^+$. Furthermore, fix $t_\ast>0$ and consider
	\begin{equation}\label{eq:a*1}
		a_\ast:=a_\ast(x)=(\Phi')^{-1}\left(\frac{t_\ast}{x}\right),
	\end{equation}
	that is well-defined since $\Phi'$ is decreasing. Define the set 
	 \begin{equation*}
	 	\mathbb{D}^\prime=\{(t,x):\, 0 < t \le t_\ast < x\phi'(0+)\}
	 \end{equation*}
	 and, for $(t,x) \in \mathbb{D}^\prime$, let $c:=c(t,x)=(\phi')^{-1}(t/x)\geq a_*$. Then, $1/a_*=o(1/x)$ as $x \to \infty$, and,  for any $k\geq0,l\geq0$, as $x\to\infty$,
	 \begin{equation}\label{asymp1}
	 	\begin{split}
	 		\sup_{0< t \le t_\ast}\abs{(-1)^k\sqrt{2\pi}\frac{c^{1-l}\sqrt{-\Phi''(c)x}}{\Phi^{k+1}(c)}e^{-ct+x\Phi(c)}\frac{\partial^k\partial ^l }{\partial x^k\partial t^l}\fP{x,t}-1}=\bo{\sqrt{\frac{1}{x}}}.
	 	\end{split}
	 	\end{equation}
%	  and, \mladen{for any fixed $t>0$}, one has, for any $k,l\geq 0$, 
%	\begin{equation}\label{asymp1}
%		\begin{split}
%			& \frac{\partial^k\partial ^l }{\partial x^k\partial t^l}\fP{x,t}=\frac{(-1)^k}{\sqrt{2\pi}}e^{a_*t-x\Phi(a_*)}\frac{\Phi^{k+1}(a_*)}{a^{1-l}_*\sqrt{-\Phi''(a_*)x}}\lbrb{1+\bo{\frac{\sqrt{\ln\lbrb{a_*\sqrt{-\Phi''(a_*)x}}} }{a_*\sqrt{-\Phi''(a_*)x}}}}, \ \mbox{ as }x \to \infty,
%		\end{split}
%	\end{equation}
%	where $a_*=a_*(x)$ is the unique solution of
%	\begin{equation}\label{eq:a*1}
%		\Phi'(a_*)=\frac{\giacomo{t}}{x}\in\lbrb{0,\Phi'(0^+)},
%	\end{equation}
%	which in addition satisfies $\limi{x}a_*x^{-1}=\infty$.
\end{coro}
Next, we consider the case when $t/x$ does not converge to $\mathfrak{b}$ or $\phi'(0^+)$. 
\begin{thm}\label{thm:main1}
	Let $\Phi$ be the Laplace exponent of some potentially killed subordinator and assume that \eqref{def:condiA} holds. For all $[t_1,t_2]\subset \lbrb{\mathfrak{b},\Phi'(0^+)}$ define correspondingly
	%\begin{equation*} 
	%	\mathbb{D}'=\curly{(t,x) \in \mathbb{D}: \ xt_1\leq t  \leq  xt_2}, \ \mbox{ and } \  c:=c(t,x)=(\phi')^{-1}\left(t/x\right)\in \lbbrbb{(\phi')^{-1}(t_2),(\phi')^{-1}(t_1)}, \ (x,t) \in \mathbb{D}^\prime.
	%\end{equation*} 
	\begin{equation*} 
		\mathbb{D}'=\curly{(t,x) \in \mathbb{D}: \ xt_1\leq t  \leq  xt_2}, \quad \mbox{ and } \quad   c:=c(t,x)=(\phi')^{-1}\left(t/x\right), \ \mbox{ for } (x,t) \in \mathbb{D}^\prime.
	\end{equation*}
	Then, for any $k\geq0,l\geq0$, we have, as $x\to\infty$
	\begin{equation}\label{def:fin}
		\begin{split}
			\sup_{xt_1 \le t \le xt_2}\abs{(-1)^k\sqrt{2\pi}\frac{c^{1-l}\sqrt{-\Phi''(c)x}}{\Phi^\dagger(c)\Phi^{k}(c)}e^{-ct+x\Phi(c)}\frac{\partial^k\partial ^l }{\partial x^k\partial t^l}\fP{x,t}-1}=\bo{\sqrt{\frac{\ln(x)}{x}}}.
		\end{split}
	\end{equation}
\end{thm}
%Theorem \ref{thm:main1} is proved in Section \ref{sec:proofs}.
\begin{rmk}\label{rem:main1}
	 For $f_\Phi$ only, such a result is contained in \cite[Theorem 3.2]{DonRiv}. It will be discussed in detail in Section \ref{subsubsec:EX}. 
\end{rmk}

Finally, we consider case when $t/x$  converges from below to $\phi'(0^+)$. Then, we need a local condition which is a modification of $\eqref{def:condiB}$, namely
\begin{equation}\label{def:condiB'}
	\begin{split}
		&\limsupo{x}\frac{x\Phi'''(x)}{-\Phi''(x)}=K<\infty.
	\end{split}\tag{$\mathbb{A}'_2$}
\end{equation}
In this case we can prove the following result.
%We have the result, whose additional assumption \eqref{eq:addCondi}, as argued in Section \ref{subsubsec:EX}, is mild.
\begin{thm}\label{thm:main2}
	Let $\Phi$ be the Laplace exponent of some potentially killed subordinator and assume that  conditions $\eqref{def:condiA}$ and $\eqref{def:condiB'}$ hold true. If $t=t(x)$ is such that $t(x)/x\in\lbrb{\mathfrak{b},\Phi'(0^+)}$ and  $\limi{x}t(x)/x=\Phi'(0^+)$. Assume further that $a_*:=a_*(x)=(\Phi')^{-1}(t(x)/x)$ satisfies
	%from below 
	%in such a way that the solution to $\eqref{eq:a*}$, that is $a_*=a_*(x)$ satisfies
	\begin{equation}\label{eq:addCondi}
		\limi{x} -x\phi''(a_*)a^2_*=\infty, \quad \limsupi{x}\frac{-\ln\lbrb{a_*}}{x}<\infty, \quad \mbox{ and } \quad \forall \delta>0 \limi{x}e^{-\delta x}x\phi''(a_*)a^2_*=0.
	\end{equation}
	Then, for any $k\geq0,l\geq0$, we have, as $x \to \infty$, 
	%asymptotically
	\begin{equation}\label{asymp2}
		\begin{split}
			& \frac{\partial^k\partial ^l }{\partial x^k\partial t^l}\fP{x,t}=\frac{(-1)^k}{\sqrt{2\pi}}e^{a_*t-x\Phi(a_*)}\frac{\Phi^\dagger(a_*)\Phi^{k}(a_*)}{a^{1-l}_*\sqrt{-\Phi''(a_*)x}}\lbrb{1+\bo{\frac{\sqrt{\ln\lbrb{a_*\sqrt{-\Phi''(a_*)x}}} }{a_*\sqrt{-\Phi''(a_*)x}}}}.
			%, \ \mbox{ as }x \to \infty.
		\end{split}
	\end{equation} 
\end{thm}
%\begin{thm}\label{thm:main2}
%	Let $\Phi$ be the Laplace exponent of some potentially killed subordinator. Assume that  conditions $\eqref{def:condiA}$ and $\eqref{def:condiB'}$ hold true. If $L=\infty$ in \eqref{def:condiA} then $f_\phi(x,t)$ defined in \eqref{def:f} above is infinitely differentiable in both variables on $\mathbb{D}$ and, if $L<\infty$, then, for any fixed $k,l\geq 0$, $\frac{\partial^k}{\partial x^k} \frac{\partial^l}{\partial t^l}f(t,x)$  exists for points in $\mathbb{D}$, for which $x$  is large enough. If $t=t(x)$ is such that $t(x)/x\in\lbrb{\mathfrak{b},\Phi'(0^+)}$ and  $\limi{x}t(x)/x=\Phi'(0^+)$ from below in such a way that the solution to $\eqref{eq:a*}$, that is $a_*=a_*(x)$ satisfies
%	\begin{equation}\label{eq:addCondi}
%		\limi{x} -x\phi''(a_*)a^2_*=\infty,\,\limsupi{x}\frac{\ln\lbrb{\frac{1}{a_*}}}{x}<\infty,\,\forall \delta>0 \limi{x}e^{-\delta x}x\phi''(a_*)a^2_*=0,
%	\end{equation}
%	then, for any $k\geq0,l\geq0$, we have asymptotically
%	\begin{equation}\label{asymp2}
%		\begin{split}
%			& \frac{\partial^k\partial ^l }{\partial x^k\partial t^l}\fP{x,t}\stackrel{\infty}{=}\frac{(-1)^k}{\sqrt{2\pi}}e^{a_*t-x\Phi(a_*)}\frac{\Phi^\dagger(a_*)\Phi^{k}(a_*)}{a^{1-l}_*\sqrt{-\Phi''(a_*)x}}\lbrb{1+\bo{\frac{\sqrt{\ln\lbrb{a_*\sqrt{-\Phi''(a_*)x}}} }{a_*\sqrt{-\Phi''(a_*)x}}}}.
%		\end{split}
%	\end{equation} 
%\end{thm}
%See Section \ref{sec:proofs} for the proof of Theorem \ref{thm:main2}.
\begin{rmk}\label{rem:main2}
Conditon \eqref{eq:addCondi} is not so restrictive, as will be observed in Section \ref{subsubsec:EX}.
We also mention that if $\phi''(0^+)<\infty$ then  \eqref{def:condiB'} holds true and the last two conditions in \eqref{eq:addCondi} follow from the first, which actually becomes $\limi{x}xa^2_*=\infty$.
%, which in turn implies \eqref{def:condiB'}.
\end{rmk}
Next,  in the spirit of Theorem 1 in \cite{DonRiv}, we use the results above to obtain information for the densities of the subordinators themselves. Recall that $\Pbb{\sigma(x)\in dt}=g_\Phi(x,t)dt$ provided $g_\Phi$ exists (see \eqref{def:gphi}). Furthermore, set $G_\Phi(x,t)=\Pbb{\sigma(x) \le t}$. The proof of the next result uses the relation
\[\IntOI e^{-z t}G_\Phi(x,t)dt=\frac{e^{-x\Phi(z)}}{z}={\frac{1}{\Phi^\dagger\lbrb{z}}}\int_{0}^\infty f_\Phi(x,t)e^{-z t}dt\]
and is therefore identical to the ones of Theorem \ref{thm:mainL}, Theorem \ref{thm:main1} and Theorem \ref{thm:main2} including the existence of $g_\phi$, that is the derivative of $G_\phi$ in $t$, and its derivatives. We only state the following result.
\begin{thm}\label{thm:mainS}
	Let $\Phi$ be the Laplace exponent of some potentially killed subordinator and assume that condition \eqref{def:condiA} holds. Then, for any $n \ge 0$, there exists $x_0(n,L) \ge 0$ such that, for any $k,l\ge 0$ with $k+l \le n$ and $(x,t) \in \mathbb{D}$ with $x>x_0(n,L)$, $\displaystyle \frac{\partial^k}{\partial x^k}\frac{\partial^l}{\partial t^l}G_\Phi(x,t)$ is well-defined. If $L=\infty$ in \eqref{def:condiA}, then $x_0(n,\infty)=0$. In particular, the density $g_\Phi$ is well-defined for $x > x_0(1,L)$ and $\displaystyle \frac{\partial^k}{\partial x^k}\frac{\partial^l}{\partial t^l}g_\Phi(x,t)$ is well-defined for $x>x_0(k+l+1,L)$. Furthermore,the following three statements hold true.
	\begin{itemize}
		\item[$(i)$]  Under the conditions and the notation of Theorem~\ref{thm:mainL}  set $t=t(x)$ with $t(x)/x \in (\mathfrak{b},\Phi'(0+))$ and $t(x)/x \to \mathfrak{b}$, as $x \to \infty$. Then, as $x\to\infty$,
		\begin{equation}\label{asympS1}
			\begin{split}
				\sup_{x\mathfrak{b}<t \le t(x)}\abs{ \Phi^\dagger(c)\frac{\frac{\partial^k\partial ^l }{\partial x^k\partial t^l}G_{\Phi}\lbrb{x,t}}{\frac{\partial^k\partial ^l }{\partial x^k\partial t^l} f_{\Phi}\lbrb{x,t}}-1}&=\bo{\sup_{c\geq a_*(x)}\frac{\sqrt{\ln\lbrb{c\sqrt{-\Phi''(c)x}}} }{c\sqrt{-\Phi''(c)x}}}=\bo{\sqrt{\frac{\ln(x)}{x\ln a_*}}},
			\end{split}
		\end{equation}
		where
		\begin{equation*}
			\mathbb{D}^\prime=\{(t,x): \ x\mathfrak{b} < t \le t(x) < x\phi'(0+)\}, \ c:=c(t,x)=(\Phi')^{-1}(t/x), \ \mbox{ for }(t,x) \in \mathbb{D}^\prime
		\end{equation*}
		and $a_*:=a_*(x)=c(t(x),x)$.
		\item[$(ii)$]  Under the conditions of Theorem~\ref{thm:main1}  for all $[t_1,t_2] \subset (\mathfrak{b},\Phi'(0+))$ it holds, as $x\to\infty$,
		\begin{equation}\label{asympS}
			\begin{split}
				\sup_{xt_1 \le t \le xt_2}\abs{ \Phi^\dagger(c)\frac{\frac{\partial^k\partial ^l }{\partial x^k\partial t^l}G_{\Phi}\lbrb{x,t}}{\frac{\partial^k\partial ^l }{\partial x^k\partial t^l} f_{\Phi}\lbrb{x,t}}-1}
				=\bo{\sqrt{\frac{\ln(x)}{x}}},
			\end{split}
		\end{equation}
		where
		\begin{equation*}
			\mathbb{D}^\prime=\{(t,x): \ xt_1 \le t \le xt_2 \}, \ \mbox{ and } \ c:=c(t,x)=(\Phi')^{-1}(t/x) \ \mbox{ for }(t,x) \in \mathbb{D}^\prime.
		\end{equation*}
		\item[$(iii)$]Under the conditions of Theorem~\ref{thm:main2}  set $t=t(x)$ with $t(x)/x \in (\mathfrak{b},\Phi'(0+))$ and $t(x)/x \to \Phi'(0+)$ as $x \to \infty$. Then, as $x\to\infty$,
		\begin{equation}\label{asympS2}
			\frac{\partial^k\partial ^l }{\partial x^k\partial t^l}G_{\Phi}\lbrb{x,t}=\frac{1}{\phi^\dagger(a_*)}\frac{\partial^k\partial ^l }{\partial x^k\partial t^l}f_{\Phi}\lbrb{x,t}\left(1+\bo{\sqrt{\frac{\ln(x)}{x\ln(a_*(x))}}}\right),
		\end{equation}
		where $a_*:=a_*(x)=(\Phi')^{-1}(t(x)/x)$.
	\end{itemize}
	
	% if $t/x=t(x)/x \downarrow \mathfrak{b}$, then, as $x\to\infty$,
	%\begin{equation}\label{asympS1}
	%	\begin{split}
	%\sup_{(t,x) \in \mathbb{D}^\prime}\abs{ \Phi^\dagger(c)\frac{\frac{\partial^k\partial ^l }{\partial x^k\partial t^l}G_{\Phi}\lbrb{x,t}}{\frac{\partial^k\partial ^l }{\partial x^k\partial t^l} f_{\Phi}\lbrb{x,t}}-1}&=\bo{\sup_{c\geq a_*(x)}\frac{\sqrt{\ln\lbrb{c\sqrt{-\Phi''(c)x}}} }{c\sqrt{-\Phi''(c)x}}}=\bo{\sqrt{\frac{\ln(x)}{x\ln a_*(x)}}}.
	%	\end{split}
	%\end{equation}
	%When $t$ ranges in $[xt_1,xt_2]$ with $[t_1,t_2]\subseteq (\mathfrak{b},\phi'(0^+))$ then under the conditions and notation of Theorem \ref{thm:main1} we have, as $x\to\infty$, 
	%\begin{equation}\label{asympS}
	%	\begin{split}
	%		\sup_{(t,x) \in \mathbb{D}^\prime}\abs{ \Phi^\dagger(c)\frac{\frac{\partial^k\partial ^l }{\partial x^k\partial t^l}G_{\Phi}\lbrb{x,t}}{\frac{\partial^k\partial ^l }{\partial x^k\partial t^l} f_{\Phi}\lbrb{x,t}}-1}
	%		=\bo{\sqrt{\frac{\ln(x)}{x}}}.
	%	\end{split}
	%\end{equation}
%Finally, if $\limi{x}t/x=\Phi'(0^+)$ and  under the assumptions of Theorem \ref{thm:main2} we have  the expression \eqref{asympS1}. 
\end{thm}
In particular, under \eqref{def:condiA}, we have $f^{\rm k}_\phi(x,t)=qG_\phi(x,t)$ and $f^{\rm c}_{\Phi}(x,t)=\mathfrak{b}g_\phi(x,t)$, see \eqref{def:fk} and \eqref{def:fc}. Hence, the conclusions of Theorem \ref{thm:mainS} can be transferred to $f^{\rm c}_{\Phi}$ and $f^{\rm k}_\Phi$.
\begin{coro}\label{cor:mainS}
	Under the conditions of Theorem \ref{thm:mainS}, the asymptotics given in  \eqref{asympS1},\eqref{asympS} and \eqref{asympS2} hold for $q^{-1}f^{\rm k}_{\Phi}$ and $\mathfrak{b}^{-1}f^{\rm c}_{\Phi}$, provided $\mathfrak{b},q>0$.
\end{coro}

\subsubsection{Discussion and comparison to existing results}\label{subsubsec:EX}
First we consider general asymptotic results in the literature for which we need the following lemmae. Recall the definition of $\Delta(x)$ in \eqref{def:D}.
\begin{lem}\label{lem:phi''}
	Let $\phi$ be a Bernstein function. Then, for any $x>0$, it holds that
	\begin{equation}\label{eq:phi''}
		\begin{split}
			& e^{-1}\Delta(x)\leq-\phi''(x)\leq \Delta(x)+\frac{e^{-1}}{x^2}\bar{\mu}_\phi\lbrb{\frac{1}{x}}.
		\end{split}
	\end{equation}
\end{lem}
Introduce the condition
\begin{equation}\label{condi:DR}
	\begin{split}
		&\liminfi{y}\frac{y^2\Delta(y)}{\bar{\nu}_\Phi(y^{-1})}=\liminfo{y}\frac{\Delta(y^{-1})}{y^2\bar{\nu}_\Phi(y)}>0, 
	\end{split}
\end{equation}
which implies $\Phi(\infty)=\infty$ and resembles the well-known positive increase condition, see \cite[eq (6)]{MinSav23+}.
\begin{lem}\label{lem:condi}
	Let $\phi$ be a Bernstein function with $\phi(\infty)=\infty$. If
	\begin{equation}\label{def:condiB''}
		\liminfi{x}\frac{\phi''(2x)}{\phi''(x)}>0, \tag{$\mathbb{A}_2^*$}
	\end{equation} 
then \eqref{def:condiB} holds true. Condition \eqref{condi:DR} implies the validity of \eqref{def:condiB''}, \eqref{def:condiB}, \eqref{def:condiA} with $L=\infty$ and the existence of a $\beta>0$ small enough such that
% for some $\beta>0$ small enough that
\begin{equation}\label{eq:dBound}
\liminfi{x}\frac{\Delta\lbrb{x}}{x^{-2+\beta}}>0,\,\,\, \liminfi{x}\frac{-\phi''(x)}{x^{-2+\beta}}>0  \text{ and } \Delta(x)\asymp -\phi''(x), \mbox{ as $x\to\infty$}.
\end{equation}
\end{lem}
The lemmae above are proved in Section \ref{subsec:aux}. 

Now we are ready to compare the conditions and the results of Theorem \ref{thm:mainS} to the ones in the literature. The conditions of course match those in the theorems concerning densities and their derivatives of inverse subordinators but those seem to have not been studied in such detail prior to our work. All results below relate to subordinators. 

One of the most general results in the literature is \cite[Theorem 3.3]{GrLTr21} which contains the following non-uniform version of \eqref{asympS1}, \eqref{asympS} and \eqref{asympS2} of Theorem \ref{thm:mainS}
\begin{equation*}
g_{\phi}(x,t)=\frac{1}{\sqrt{-2\pi x \phi''(a_*)}}e^{a_*t-x\Phi(a_*)}\lbrb{1+\so{1}},
\end{equation*}
where $a_*=(\phi')^{-1}(t/x)$ and $t,x$ are admissible as in our claims. The main condition of \cite[Theorem 3.3]{GrLTr21} implies via the notion of almost increasing functions, see \cite[p.6]{GrLTr21} and \cite[Lemma 2.8]{GrLTr21} that, for some $\alpha>0$, $-x^{2-\alpha}\phi''(x)$ is almost increasing and all claims of \eqref{def:condiB''} and \eqref{eq:dBound} hold. Hence, our \eqref{def:condiB}, \eqref{def:condiA} with $L=\infty$ are satisfied and thus \eqref{asympS1} and \eqref{asympS} require less restrictive conditions and yield uniform estimates with  speed of convergence. For \eqref{asympS2} we cannot derive \eqref{def:condiB'} from the condition in \cite[Theorem 3.3]{GrLTr21}. Also \cite[Corollaries 3.5 and 3.7]{GrLTr21} offer uniform asymptotic equivalences for $g_\phi$.

Next, we discuss the results in \cite[Theorem 3.2 (iii)]{DonRiv} which correspond to the claims of \eqref{asympS1}, \eqref{asympS} and \eqref{asympS2} of Theorem \ref{thm:mainS}. We note that the asymptotic results in \cite[Theorem 3.2 (iii)]{DonRiv} are uniform but lack speed of convergence. The main condition in the setting of \eqref{asympS1}, i.e. $t(x)/x\downarrow \mathfrak{b}$, is \eqref{condi:DR}, which from Lemma \ref{lem:condi} implies our \eqref{def:condiB}, \eqref{def:condiA} with $L=\infty$. On the other hand, choosing $\nu_{\Phi}(dy)=y^{-1}\ln\lbrb{y^{-1}}\ind{y\in\lbrb{0,1}}dy$, we see that \eqref{condi:DR} is not satisfied whereas  \eqref{def:condiA} and \eqref{def:condiB} hold since at infinity $x\Phi'''(x)\asymp x^{-2}\ln(x)\asymp -\Phi''(x)$ and $x^2\Delta(x)\asymp\ln\lbrb{x}$, as $x \to \infty$. In the regime \eqref{asympS} \cite[Condition H]{DonRiv} is implied by our \eqref{def:condiA} but we offer explicit asymptotic speed of convergence and results for all derivatives. For \eqref{asympS2} the conditions in \cite{DonRiv} may not be matched with ours as they concern \eqref{condi:DR} at zero. Finally, we highlight that \cite{DonRiv} employ as a main tool the Escher transform which although very powerful may not be directly used for the derivatives.

For complete Bernstein functions with $\mathfrak{b}=q=0$ with \LL measure that admits density $m_\phi$ with asymptotic  at zero of the type $m_\phi(y)=c_0y^{-\alpha_0}+c_1y^{-\alpha_1}+\cdots$ and $\limi{y}e^{-\sigma y}m_{\phi}(y)=0, \sigma>0$,  \cite[Theorem 3.6 (ii)]{Fah_10} offers, for fixed $t$ and $x\to\infty$, asymptotic representation as \eqref{asympS1} with an explicit speed of convergence. The assumptions are much more restrictive and only fixed space is considered.

Next we present some results representing the state-of-the-art for two-sided bounds for $g_\phi$ and $f_\phi$. Under the conditions of \cite[Theorem 3.3]{GrLTr21} the authors present explicit upper bounds for $g_\phi$, see \cite[Theorem 4.7]{GrLTr21},
whereas with some further restrictions they obtain clear lower bounds, see \cite[Theorem 4.11]{GrLTr21}. Neat two-sided bounds for $g_\phi, f_\phi$ are deducted in \cite[Theorem 4.4]{CKKW20} by seemingly very simple, clever approach but under more restrictive assumptions. Two-sided bounds for $g_\phi$ are also obtained in \cite[Theorem 1.3]{ChoKim21} and their form is precisely as the asymptotic term for $k=l=0$ in Theorem \ref{thm:mainS}. Again the restrictions are not mild. It is worth noting that other than \cite{GrLTr21} the other papers demand the existence of \LL density and impose some conditions on it. Bounds are also presented in \cite{KK13}.
%\ga{Let us also underline that \eqref{def:condiA} and \eqref{def:condiB} are verified if $\Phi$ is a Bernstein function that is regularly varying at infinity with index $\alpha>0$. Indeed, on the one hand \eqref{def:condiB} follows by a direct application of the monotone density theorem \cite[Theorem 1.7.2]{bingham}. On the other hand, \eqref{def:condiA} follows with $L=\infty$ once we notice that $\mu_\Phi(dy)=m(y)dy$ where $m$ is regularly varying at $0$ with index $-\alpha-1$ (see \cite[Proposition 5.23]{pottheory}). Actually, \eqref{def:condiA} is verified even if $\Phi$ is a Bernstein function such that $\Phi(x)\asymp x^{\alpha}\mathfrak{l}(x)$ as $x \to \infty$ for some slowly varying function $\mathfrak{l}$, as a consequence of \cite[Theorem 13.2.10]{ksv12}.}

\subsection{A power series representation of densities and their derivatives and their behaviour at zero}\label{subsec:series}
Here we discuss our results concerning the power-series representation of the density of the inverse subordinator and its derivatives. These will be obtained by assuming that the Laplace exponent $\Phi$ of the involved potentially killed subordinator satisfies the following assumptions:
	\begin{align}
	 \begin{split} & \mbox{There exists $\theta \in (0,\pi)$ such that $\Phi$ admits a holomorphic extension} \\ & \qquad \qquad \qquad \mbox{on $\C\left(\pi-\frac{\theta}{2}\right)$ which is continuous on $\overline{\C\left(\pi-\frac{\theta}{2}\right)}$} \end{split} \tag{$\mathbb{B}_1$} \label{eq:extensionA3}\\
 	&	\lim_{z \to +\infty}\frac{\Phi(z)}{z}=\mathfrak{b} \qquad 	\mbox {uniformly in } \overline{\C\left(\pi-\frac{\theta}{2}\right)}. \tag{$\mathbb{B}_2$} \label{eq:uniformlimcond} 
% 	\\
%	&	\left|\Phi^\dagger\left(\rho e^{i\left(\pi-\frac{\theta}{2}\right)}\right)\right| \le C\Phi^\dagger(\rho), \tag{$\mathbb{A}_5$} 	\label{eq:modcont}
%	 \\
%		&	\int_1^{+\infty} s^{-1} \bar{\nu}_{\Phi}(s) ds \, < \, +\infty. 	\label{eq:integrlog} \tag{$\mathbb{A}_6$}
	\end{align}
We remark that \eqref{eq:extensionA3} and \eqref{eq:uniformlimcond} are very general, since they are satisfied by a wide class of Bernstein functions (see the discussion in Section \ref{discussionassumptions} below). The proofs of the results of this section are provided in Section \ref{subsec:power}. Upon the validity of \eqref{eq:extensionA3} and \eqref{eq:uniformlimcond} we have the following regularity result.
\begin{thm}\label{thm:smoothfmu}
	Let $\Phi$ be the Laplace exponent of a potentially killed subordinator satisfying assumptions \eqref{eq:extensionA3} and \eqref{eq:uniformlimcond}. Then, for any $n \ge 1$, $\bar{\mu}_\phi^{\ast n}$ belongs to $C^\infty(0,+\infty)$ and $f_\Phi \in C^\infty(\mathbb{D})$.
\end{thm}
%Theorem \ref{thm:smoothfmu} is proved in Section \ref{subsec:power}. 
Furthermore, under the same conditions, the following  power series representation holds.
\begin{thm}\label{thm:seriespi}
Let $\Phi$ be the Laplace exponent of a potentially killed subordinator satisfying assumptions \eqref{eq:extensionA3} and \eqref{eq:uniformlimcond}. Then, for any $k,l \ge 0$,
\begin{equation}\label{eq:seriesder}
	\frac{\partial^k \partial^l}{\partial x^k \partial t^l}f_\Phi(x, t) \, = \, \sum_{j=0}^\infty \frac{x^{j}}{j!} \mathcal{I}_{j,k,l}(t), \qquad (x,t) \in \mathbb{D},
\end{equation}
%Then, on $\mathbb{D}$, we have that
%\begin{equation}\label{eq:series}
%	f_\Phi(x, t) \, = \, \sum_{j=0}^\infty \frac{x^j}{j!} \mathcal{I}_{j,0,0}(t),
%\end{equation}
where 
\begin{equation} \label{coeff}
	\mathcal{I}_{j,k,l}(t):=(-1)^{k+j} \sum_{k_1+k_2+k_3=k+j} \frac{(k+j)!}{k_1!k_2!k_3!}q^{k_1}\mathfrak{b}^{k_2} \frac{d^{l+k_2+k_3}}{dt^{l+k_2+k_3}} \bar{\nu}_\Phi^{\ast(k_3+1)}(t)
\end{equation}
and the series is absolutely convergent.
% The derivatives appearing in the coefficients \eqref{coeff} exist for any $t>0$, and $l,k_2, k_3 \geq 0$. Furthermore, for any $k, l \ge 0$ we have, on $\mathbb{D}$,
%\begin{equation}\label{eq:seriesder}
%	\frac{\partial^k \partial^l}{\partial x^k \partial t^l}f_\Phi(x, t) \, = \, \sum_{j=0}^\infty \frac{x^{j}}{j!} \mathcal{I}_{j,k,l}(t),
%\end{equation}
%where the series is absolutely convergent and $\frac{\partial^k \partial^l}{\partial x^k \partial t^l}f_\Phi(x,t)$ is continuous in $\mathbb{D}$.
\end{thm}
%See Section \ref{subsec:power} for the proof of Theorem \ref{thm:seriespi}.
\begin{rmk}\label{rmk:seriesexp}
	Note that the series \eqref{eq:seriesder} which contains the tail of the L\'evy measure, its convolutions and their derivatives is similar to the series expansion for the potential density of a subordinator with drift, see \cite{DoeSav_10}.
\end{rmk}
\begin{rmk}
	Observe that $f_\Phi(x,t)=0$ for any $x<0$ and $t>0$, hence the previous theorem shows that, for fixed $t>0$, $f_\Phi(x,t)$ coincides, for $0<x<\mathfrak{b}/t$, with an entire function, i.e., the right hand side of \eqref{eq:seriesder} for $k=l=0$.
\end{rmk}
The series representation \eqref{eq:seriesder} yields information about the behaviour at $0$ of $f_\Phi$ and its derivatives.
\begin{thm}
	\label{behavatzero}
	Under the assumptions of Theorem \ref{thm:seriespi}, for any $k,l \ge 0$ and any $[t_1,t_2] \subset (0,+\infty)$ we have
	\begin{equation}\label{eq:zero1}
		\sup_{t \in [t_1,t_2]}\left|\frac{\partial^k \partial^l}{\partial x^k \partial t^l}f_\phi(x,t)-\mathcal{P}_{n,k,l}(x,t)\right| \le \frac{x^{n+1}}{(n+1)!}\mathcal{R}_n(x;t_1,t_2),
	\end{equation}
	where
	\begin{equation}\label{eq:zero2}
		\mathcal{P}_{n,k,l}(x,t)=\sum_{j=0}^{n}\frac{x^j}{j!}\mathcal{I}_{j,k,l}(t),
	\end{equation}
	and $\sup_{x \in [0,x_1]}\mathcal{R}_n(x;t_1,t_2)<\infty$  for all $x_1 \in \left(0,\frac{t_1}{\mathfrak{b}}\right)$.
	%\begin{equation*}
	%	\lim_{x \to 0}\frac{\partial^k \partial^l}{\partial x^k \partial t^l} f_\phi(x,t)=\mathcal{I}_{0,k,l}(t)
	%\end{equation*}
	%and
	%\begin{equation*}
	%	\lim_{x \to 0}x^{-1}\left|\frac{\partial^k \partial^l}{\partial x^k \partial t^l} f_\phi(x,t) - \mathcal{I}_{0,k,l}(t)  \right| = \mathcal{I}_{1,k,l}(t),
	%\end{equation*}
	%where both limits are uniform with respect to $t \in [a,b]$ for any $[a,b]\subset (0,+\infty)$.
\end{thm}
%\begin{thm}
%	\label{behavatzero}
%	Under the assumptions of Theorem \ref{thm:seriespi}, for any $k,l \ge 0$ we have
%	\begin{equation*}
%		\lim_{x \to 0}\frac{\partial^k \partial^l}{\partial x^k \partial t^l} f_\phi(x,t)=\mathcal{I}_{0,k,l}(t)
%	\end{equation*}
%	and
%	\begin{equation*}
%		\lim_{x \to 0}x^{-1}\left|\frac{\partial^k \partial^l}{\partial x^k \partial t^l} f_\phi(x,t) - \mathcal{I}_{0,k,l}(t)  \right| = \mathcal{I}_{1,k,l}(t),
%	\end{equation*}
%	where both limits are uniform with respect to $t \in [a,b]$ for any $[a,b]\subset (0,+\infty)$.
%\end{thm}
%Theorem \ref{behavatzero} is proved in Section \ref{subsec:power}.
Though conditions \eqref{eq:extensionA3} and \eqref{eq:uniformlimcond} depend on suitable $\theta \in (0,\pi)$, the series representation \eqref{eq:seriesder} is independent of $\theta$. Indeed, the result follows once one recognize some special integral representations for both the function $f_\Phi$ and the convolution powers $\bar{\mu}_\Phi^{\ast n}$. Note that the series provides an explicit representation of $f_\Phi$ whenever the convolution powers $\bar{\mu}_\Phi^{\ast n}$ can be evaluated. On the other hand, there could be some cases in which the integral formulation of $\bar{\mu}_\Phi^{\ast n}$ can be used to provide such an evaluation. This is the case, for instance, when we can extend $\Phi$ on the whole complex half-plane $\overline{\C(0,\pi)}=\{z \in \C: \ \Im(z) \ge 0\}$. Let us underline that such a property is not necessarily verified by all complete Bernstein functions, as for instance $\phi(z)=\log(1+z)$ cannot satisfy this. 
\begin{prop}\label{prop:extcont}
		Let $\Phi$ be a complete Bernstein function and assume that $\Phi$ can be extended with continuity on $\overline{\C(0,\pi)}$. Denote by $\Phi_{+}$ such extension. Assume further that 
		\begin{equation*}
			\lim_{z \to +\infty}\frac{\Phi^\dagger_+(z)}{z}=0 \mbox{ uniformly in }\overline{\C(0,\pi)}.
		\end{equation*}
		Then, for any $\varepsilon,t>0$, $r\geq 0$ and $n \geq 1$, denoting by $\gamma_\varepsilon$ the parametrized curve $\gamma_\varepsilon: z=\varepsilon e^{i\xi}$ for $\xi \in [-\pi,\pi]$, we have that
		\begin{equation}\label{eq:integralmucont}
			\frac{d^r}{dt^r}\overline{\mu}_\Phi^{n \star}(t)=\frac{1}{\pi}\int_{\varepsilon}^{+\infty}\Im\left[(-1)^{r+n+1}\frac{(\Phi^\dagger_+(-\rho))^n}{\rho^{n-r}}e^{-t\rho}\right]d\rho+\frac{1}{2\pi i}\int_{\gamma_\varepsilon} e^{tz}\frac{(\Phi^\dagger(z))^n}{z^{n-r}}dz.
		\end{equation}
\end{prop}
%See Section \ref{subsec:power} for the proof of Proposition \ref{prop:extcont}.

Similarly to what we did for the asymptotic behaviour at infinity, we can use the relation
\begin{equation}
	\int_0^{+\infty} e^{- z t} G_\phi(x,t) dt \, = \, \frac{e^{-x\Phi(z)}}{z}
\end{equation}
to obtain, with the same method, results on $G_\phi(x,t)$ (see \eqref{def:gphi}). This is stated in the following theorem.
% whose proof is in Section \ref{subsec:power}.
%The following theorem, therefore, is given with no proof.
\begin{thm}
\label{thm:seriessub}
Under \eqref{eq:extensionA3} and \eqref{eq:uniformlimcond}, $G_\phi \in C^\infty(\mathbb{D})$ and for any $k \geq 1$, $l \ge 0$ and $(x,t) \in \mathbb{D}$
\begin{align}
	G_\phi(x,t) \, = \, e^{-qx}+\sum_{j=1}^{+\infty} \frac{x^j}{j!} \mathfrak{I}_{j,0,0}(t),
	\label{eq:seriesdistr}
\end{align}
\begin{align}
	\frac{\partial^l}{\partial t^l} g_\phi(x,t) \, = \, \sum_{j=1}^{+\infty} \frac{x^j}{j!} \mathfrak{I}_{j,0,l+1}(t),
	\label{eq:seriessub1}
\end{align}
and
\begin{align}
	\frac{\partial^l}{\partial t^l} \frac{\partial^k}{\partial x^k} g_\phi(x,t) \, = \, \sum_{j=0}^{+\infty} \frac{x^j}{j!} \mathfrak{I}_{j,k,l+1}(t),
	\label{eq:seriessub2}
\end{align}
where
\begin{equation}\label{eq:fractureI}
	\mathfrak{I}_{j,k,l}(t)=(-1)^{k+j+1}\sum_{k_1+k_2+k_3=k+j-1}\frac{(k+j)!}{k_1!k_2!(k_3+1)!}q^{k_1}\mathfrak{b}^{k_2}\dersup{}{t}{k_2+k_3+l}\bar{\mu}_\phi^{\ast(k_3+1)}(t).
\end{equation}
In particular, all the series are absolutely convergent.
\end{thm}
As in the previous section, under \eqref{eq:extensionA3} and \eqref{eq:uniformlimcond} we know that $f^{\rm k}_\Phi(x,t)=qG_\Phi(x,t)$ and  $f^{\rm c}_{\Phi}(x,t)=\mathfrak{b}g_\phi(x,t)$, see \eqref{def:fk} and\eqref{def:fc}. Hence, the results of Theorem \ref{thm:seriessub} can be transferred to $f^{\rm k}_\Phi$ and $f^{\rm c}_{\Phi}$.
\begin{coro}
\label{cor:seriescreep}
	Under the conditions of Theorem \ref{thm:seriessub}, \eqref{eq:seriesdistr} holds for $f^{\rm k}_\Phi$ up to  a multiplicative factor $q$. Furthermore, \eqref{eq:seriessub1} and \eqref{eq:seriessub2} hold for $f^{\rm c}_\Phi$ up to a multiplicative factor $\mathfrak{b}$.
\end{coro}

\subsubsection{Discussion and comparison to existing results}
\label{discussionassumptions}
Our Theorem \ref{thm:seriessub} is strongly related with some known results on the small-time polynomial expansion of the distribution of L\'evy processes. More precisely, in \cite[Theorem 5.1]{FH09}, the authors proved that for any (non-killed) L\'evy process $Y(t)$ and all $n \ge 0$ the following polynomial expansion holds:
\begin{equation}\label{eq:tailLevyproc}
	\P(Y(x) > t)=\sum_{j=1}^{n}d_j(t)\frac{x^j}{j!}+\frac{x^{n+1}}{(n+1)!}\mathcal{R}_n(x,t), \, 0<t<t_0,
\end{equation}
where $d_j(t)$, $j=1,\dots,n$ are some $t$-dependent coefficients and $\mathcal{R}_n(x,t)$ is bounded for $0<x<x_0$. Their result is true provided that:
\begin{itemize}
	\item[$(i)$] The L\'evy measure $\nu_Y$ of $Y$ admits a density (that we still denote here by $\nu_Y$);
	\item[$(ii)$] for any $\delta>0$ and any $k=0,\dots,2n+1$ it holds $\sup_{|t|>\delta}\left|\frac{d^k }{dt^k}\nu_Y(t)\right|<\infty$;
	\item[$(iii)$] For a fixed $\delta>0$ and for all $k=0,\dots,2n+1$ it holds $\sup_{0 < x < x_0}\sup_{|t|>\delta}\left|\frac{\partial^k }{\partial t^k}p_x(t)\right|<\infty$, where $p_x(t)dt=\P(Y(x) \in dt)$.
\end{itemize}
The theorem also provides an explicit formulation of the remainder term $\mathcal{R}_n(x,t)$. Furthermore, in \cite[Section 6]{FH09}, the authors discuss some sufficient conditions for $(iii)$ to be satisfied. In particular, the statement of \cite[Theorem 5.1]{FH09} holds for stable and tempered stable processes, as observed in \cite[Remark $6.4$, Example $6.5$ and Proposition $6.7$]{FH09}. The results in \cite{FH09}, when applied to subordinators, are less powerfull than ours in the following sense. On the one hand we do not have any restriction on $x$ in Theorem \ref{thm:seriessub}, once one observes that if $(x,t) \notin \mathbb{D}$ then $G_\phi(x,t)$ is constant. On the other hand conditions (ii) and (iii) are tipically hard to be verified, as pointed out also in \cite[Section 6]{FH09}. 
Clearly, Theorem \ref{thm:seriessub} also provides a polynomial approximation given by, for $q=0$,
\begin{equation*}
	1-G_\phi(x,t)=\sum_{j=1}^{n}\frac{x^j}{j!}(-\mathfrak{I}_{j,0,0}(t))+\frac{x^{n+1}}{(n+1)!}\cR_{n}(x,t),
\end{equation*}
where
\begin{equation*}
	\cR_n(x,t)=\sum_{j=n+1}^{+\infty}\frac{(n+1)!}{j!}x^{j-n-1}(-\mathfrak{I}_{j,0,0}(t)).
\end{equation*}
This gives an alternative representation for the remainder $\mathcal{R}_n(x,t)$ in \eqref{eq:tailLevyproc}, provided we are under the assumptions of \cite[Theorem 5.1]{FH09}, together with \eqref{eq:extensionA3} and \eqref{eq:uniformlimcond}.
While these polynomial approximations have been generalized to several other processes (see, for instance, \cite{FH12,FO16}), we are not aware about results similar to Theorems \ref{thm:seriespi} and \ref{behavatzero} for inverse subordinators except that in specific cases. It is worth noticing that the latter provides a (locally uniform) polynomial approximation for \textit{small space} of the density of an inverse subordinator if $q=\mathfrak{b}=0$ and can be combined with Corollary \ref{cor:seriescreep} to find polynomial approximations for \textit{small space} in the general case. Due to the \textit{exchange} of the roles of time and space when passing from a subordinator to its inverse, these results are in line with the ones proved in \cite{FH09}.
Let us underline, in particular, that assumptions \eqref{eq:extensionA3} and \eqref{eq:uniformlimcond} cover a wide class of Bernstein functions, and then of subordinators. Indeed, if $\phi$ is a complete Bernstein function, then Item \eqref{it:analy1} of Lemma \ref{lem:CBern} guarantees that assumption \eqref{eq:extensionA3} is satisfied, while \eqref{eq:uniformlimcond} follows from Item \eqref{it:anglim} of the same lemma. However, we can find some Bernstein functions that are not complete but still satisfy \eqref{eq:extensionA3} and \eqref{eq:uniformlimcond} for some $\theta$. Indeed, if we consider a Bernstein function $\phi$ that is not complete, we know by Proposition \ref{prop:powchar} that there exists $\alpha \in (0,1)$ such that $\phi_\alpha$ is not a complete Bernstein function (but it is still a Bernstein function by Item \eqref{it:compos} of Lemma \ref{lem:Bern}). Now, if we consider $\theta \in (0,\pi)$ so that $\left(\pi-\frac{\theta}{2}\right)\alpha<\frac{\pi}{2}$, that exists since $\alpha<1$, then $z \in \C\left(\pi-\frac{\theta}{2}\right) \mapsto z^\alpha \in \C\left(\left(\pi-\frac{\theta}{2}\right)\alpha\right) \subset \mathbb{H}_0$ is holomorphic. Furthermore, $\phi$ is holomorphic on $\mathbb{H}_0$ and thus the composition $\phi_\alpha$ is holomorphic on $\C\left(\pi-\frac{\theta}{2}\right)$. The continuity of $\phi_\alpha$ over $\overline{\C\left(\pi-\frac{\theta}{2}\right)}$ follows similarly, thus obtaining \eqref{eq:extensionA3}. The uniform limit condition \eqref{eq:uniformlimcond} follows from Item \eqref{it:asymp} of Lemma \ref{lem:Bern} and the fact that $z^\alpha \in \overline{\mathbb{H}_0}$ whenever $z \in \overline{\C\left(\pi-\frac{\theta}{2}\right)}$, since
	\begin{equation*}
		\frac{\phi_\alpha(z)}{z}=\frac{\phi(z^\alpha)}{z^\alpha}z^{\alpha-1} \to 0.
	\end{equation*}
Thus, if we consider, e.g., the Bernstein function $\phi$ whose L\'evy measure is given by \eqref{eq:examplenocomp}, then there exists a $\beta \in (0,1)$ such that $\phi_\beta$ satisfies \eqref{eq:extensionA3} and \eqref{eq:uniformlimcond} and it is not a complete Bernstein function.
Let us stress that there are functions satisfying \eqref{def:condiA} and \eqref{def:condiB} but not \eqref{eq:extensionA3} and \eqref{eq:uniformlimcond} and vice versa. For instance, we have already shown that if $\mu_\Phi(dy)=y^{-1}\log(y^{-1})\mathbb{I}_{\{y \in (0,1)\}}$, then $\phi$ satisfies both \eqref{def:condiA} and \eqref{def:condiB}. Furthermore, we have $\Phi(z)=\frac{z}{2}\int_0^1 e^{-zy}\log^2(y)dy$, which can be clearly extended to the whole complex plane, thus it verifies \eqref{eq:extensionA3}. However, let us consider any $\lambda>0$ and the sequence $z_n=-4\lambda\pi n+i 4\pi n$, so that, for $n$ big enough, since $\cos(4\pi n y)>\frac{1}{2}$ for any $y \in \left(\frac{1}{2}-\frac{1}{12n}, \frac{1}{2}+\frac{1}{12n}\right)$,
\begin{equation*}
	\Im\left(\frac{\Phi(z_n)}{z_n}\right)\ge\frac{1}{2}\int_{\frac{1}{2}-\frac{1}{12n}}^{\frac{1}{2}+\frac{1}{12n}} e^{4\lambda \pi n y}\cos(4\pi n y)\log^2(y)dy \ge \frac{e^{\lambda \pi n}}{24n}\log^2\left(\frac{3}{4}\right) \to \infty.
\end{equation*}
Since $\lambda>0$ is arbitrary, this implies that \eqref{eq:uniformlimcond} cannot be verified for any $\theta \in (0,\pi)$. On the the hand, if we consider $\phi(z) = 1-e^{-z}$, then we know that $\phi_\alpha(z):=\phi\lb z^\alpha\rb$ satisfies \eqref{eq:extensionA3} and \eqref{eq:uniformlimcond} for some $\theta \in (0,\pi)$. However, it does not satisfies \eqref{def:condiA} since it is a bounded Bernstein function. Indeed, in Proposition \ref{prop:D0}, we will show that \eqref{def:condiA} implies that $\Phi$ is unbounded. 

	\section{Proofs of results in Subsection \ref{subsec:L}}
\label{sec:proofs}
Here we prove the results contained in Subsection \ref{subsec:L}
\subsection{Proof of Theorem \ref{thm:regularityfphi1}}
In order to prove Theorem \ref{thm:regularityfphi1} we need some preliminary results. We first show a general condition under which the density $f_\Phi$ is smooth in $\mathbb{D}$ for $x$ large enough and at the same time we provide an integral representation of $f_\Phi$ and its derivatives by means of Laplace inversion. Recall the definition of $\phi^\dagger$ in \eqref{def:Pdag}.
\begin{prop}\label{prop:LTthetapi}
	Let $\Phi$ be the Laplace exponent of a potentially killed subordinator. Assume that for $n \ge 0$ there exist $a>0$ and $x_0\geq 0$ such that, for any $x>x_0$,
	\begin{equation}\label{eq:Real1}
		\begin{split}
			\int_{-\infty}^\infty |b|^n e^{-x\Re{\Phi(\ab)}}db<\infty.
		\end{split}
	\end{equation}
	%		for any $x>x_0$.
	%		Then, for $x>x_0$ and $t>0$,
	%		\begin{equation}\label{eq:P1_0}
		%			\begin{split}
			%				&\int_{-\infty}^\infty \frac{\Phi^{\dagger}(\ab)}{\ab}e^{-x\Phi(\ab)+t\lbrb{\ab}}db
			%			\end{split}
		%		\end{equation} 
	%		is absolutely convergent and, for $x_0<x<t/\mathfrak{b}$,
	%		\begin{equation}
		%		f_\Phi(x,t) \, = \, \int_{-\infty}^{+\infty} \frac{\Phi^\dagger(a+ib)}{a+ib} e^{t(a+ib)-x\Phi(a+ib)} db.
		%		\label{intrepr}
		%		\end{equation}
	%		Furthermore, if for any $x>x_0$,
	%		\begin{equation}\label{eq:Real2}
		%		\int_{-\infty}^{+\infty} |b|^n e^{-x \Re \Phi (a+ib)} db < \infty
		%		\end{equation}
	Then, for any $k,l \ge 0$ with  $k+l \leq n$, $\displaystyle \frac{\partial^k}{\partial x^k}\frac{\partial^l}{\partial t^l}f_\Phi(x,t)$ is well defined for $(x,t) \in \mathbb{D}$ with $x>x_0$ and
	% $x>x_0$ and $t>0$,
	%		\begin{equation}\label{eq:P1}
		%			\begin{split}
			%				&\int_{-\infty}^\infty \frac{\Phi^{\dagger}(\ab)\Phi^{k}(\ab)}{\lbrb{\ab}^{1-l}}e^{-x\Phi(\ab)+t\lbrb{\ab}}db
			%			\end{split}
		%		\end{equation} 
	%is absolutely convergent and, for any 
	\begin{equation}
		\frac{\partial^l}{\partial t^l} \frac{\partial^k}{\partial x^k} f_\Phi(x,t) \, = \, (-1)^k \int_{-\infty}^{+\infty} \frac{\Phi^\dagger (a+ib) (\Phi(a+ib))^k }{(a+ib)^{1-l}} e^{t(a+ib)-x\Phi(a+ib)} db,
		\label{intreprder}
	\end{equation}
	where the integral is absolutely convergent.
\end{prop}
\begin{proof}
	By using \eqref{eq:LT1} we compute the inverse Laplace transform as in \cite[Item a), Theorem 4.2.21]{abhn}. Precisely, we have, for all fixed $x>0$ and for almost all $t>0$, up to a subsequence,
	\begin{equation}
		f_\Phi(x,t) \, = \, \text{C-}\lim _{r \to +\infty} \frac{1}{2\pi i}\int_{a-ir}^{a+ir} e^{z t } 	\,  \frac{\Phi^\dagger(z)}{z} e^{-x\Phi(z)} \, dz,
		\label{ceslimitdag}
	\end{equation}
	where $x_0>0$ is arbitrary and we denote by C-$\lim$ the Cesaro limit, i.e.,
	\begin{equation*}
		\text{C-}\lim _{r \to +\infty} \int_{a-ir}^{a+ir} e^{z t } 	\,  \frac{\Phi^\dagger(z)}{z} e^{-x\Phi(z)} \, dz \, := \, \lim_{R \to +\infty} \frac{1}{R} \int_0^R \int_{a-ir}^{a+ir} e^{z t } 	\,  \frac{\Phi^\dagger(z)}{z} e^{-x\Phi(z)} \, dz \, dr.
	\end{equation*}
	In particular, it holds by \cite[Theorem 4.1.2]{abhn} that
	\begin{equation}
		\text{C-}\lim _{r \to +\infty} \int_{a-ir}^{a+ir} e^{z t } 	\,  \frac{\Phi^\dagger(z)}{z} e^{-x\Phi(z)} \, dz \, = \, \lim _{r \to +\infty} \int_{a-ir}^{a+ir} e^{zt } 	\,  \frac{\Phi^\dagger(z)}{z} e^{-x\Phi(z)} \, dz
		\label{cesequalpointdag}
	\end{equation}
	provided the limit on the right-hand side exists. We set about to prove the latter by noting that
	\begin{equation*}
		\int_{a-ir}^{a+ir} e^{z t } 	\,  \frac{\Phi^\dagger(z)}{z} e^{-x\Phi(z)} \, dz \, = \, i\int_{-r}^r e^{(a+ib)t} \frac{\Phi^\dagger (a+ib)}{a+ib} e^{-x\Phi(a+ib)} db.
	\end{equation*}
	However, by Item \eqref{it:asymp} of Lemma \ref{lem:Bern}, we know that
	%\begin{equation*}
	%\lim_{b \to \pm\infty} \frac{|\Phi^\dagger (a+ib)|}{|a+ib|} \, = %\, 0
	%\end{equation*}
	%and thus 
	the function $|\Phi^\dagger (a+ib)|/|a+ib|$ is continuous and bounded and then by employing \eqref{eq:Real1} for $n=0$ we get
	%. Then \eqref{intrepr} follows from
	\begin{equation*}
		\int_{-\infty}^{+\infty}\frac{|\Phi^\dagger(a+ib)|}{|a+ib|}\left|e^{t(a+ib)-x\Phi(a+ib)}\right|db 
		\le Ce^{ta}\int_{-\infty}^{+\infty}e^{-x\Re \Phi(a+ib)}db<\infty.
	\end{equation*}
	Using \eqref{cesequalpointdag} into \eqref{ceslimitdag} we get \eqref{intreprder} for $k,l=0$ and $(x,t) \in \mathbb{D}$ with $x>x_0$.
	Next, for any $k+l \le n$ and any $(x,t) \in \mathbb{D}$ with $x>x_0$, we prove that $\displaystyle \frac{\partial^k}{\partial x^k}\frac{\partial^l}{\partial t^l}f_\Phi(x,t)$ is well defined and given by \eqref{intreprder}. Let $t \in [t_1, t_2]$, $x \in [x_1, x_2]$. Without loss of generality we choose $x_1>x_0$ and $x_2 < t_1/\mathfrak{b}$. Then we have
	\begin{equation*}
		\begin{split}
			\left| \frac{\Phi^\dagger (a+ib) (\Phi(a+ib))^k }{(a+ib)^{1-l}} e^{t(a+ib)-x\Phi(a+ib)} \right|
			\leq \, \left| \frac{\Phi^\dagger (a+ib) (\Phi(a+ib))^k }{(a+ib)^{1-l}} \right| e^{at_2-x_1 \Re \Phi(a+ib)}
		\end{split}
	\end{equation*}
	since $\Re \Phi (a+ib)\geq 0$ by Item \eqref{it:sign} of Lemma \ref{lem:Bern}. Furthermore, recall by Item \eqref{it:asymp} of Lemma \ref{lem:Bern} that
	\begin{equation*}
		\lim_{b \to \pm \infty}\left| \frac{\Phi (a+ib)}{a+ib} \right| = \mathfrak{b}.
	\end{equation*}
	and that $\abs{\phi(z)}=\mathfrak{b}|z|\lbrb{1+\so{1}}$ uniformly on $\overline{\mathbb{H}_0}$. Hence
	\begin{equation*}
		\left| \frac{\Phi^\dagger (a+ib) (\Phi(a+ib))^k }{(a+ib)^{1-l}}\right| \leq C|a+ib|^{k+l} \leq C \lb a^n + {|b|^{n}}  \rb,
	\end{equation*}
	for some $C>0$. Hence, \eqref{intreprder} follows by noting that
	\begin{equation*}
		\begin{split}
			&\int_{-\infty}^{+\infty} \left| \frac{\Phi^\dagger (a+ib) (\Phi(a+ib))^k }{(a+ib)^{1-l}} \right| e^{at_2-x_1 \Re \Phi(a+ib)} \, db \leq C \int_{-\infty}^{+\infty}  (a^n+|b|^n) e^{at_2-x_1 \Re \Phi(a+ib)} \, db < \infty.
		\end{split}
	\end{equation*}
	Since $k,l \ge 0$ with $k+l \le n$ are arbitrary, the last inequality implies that we can differentiate under the integral $k$ times in $x$ and $l$ times in $t$ for $(x,t) \in \mathbb{D}$ with $x>x_0$, starting from \eqref{intreprder} for $k=l=0$.
\end{proof}
To prove Theorem \ref{thm:regularityfphi1}, it is clear that we have to show \eqref{eq:Real1} for some $x_0=x_0(n,L)$ when assumption \eqref{def:condiA} holds. To do this, we need a preliminary result.
\begin{prop}\label{prop:D0}
	Let $\Phi$ be the Laplace exponent of a potentially killed subordinator  and  \eqref{def:condiA} holds for some $L>0$. Then, $\phi(\infty)=\infty$ and we have, for any $M \in (0,L)$,
	\begin{equation}\label{eq:P}
		\begin{split}
			\liminfi{x}\frac{-x^2\Phi''(x)}{\ln(x)}>Me^{-1},\qquad \liminfi{x}\frac{x\lbrb{\Phi'(x)-\mathfrak{b}}}{\ln(x)}>Me^{-1}.
		\end{split}
	\end{equation} 
\end{prop}
\begin{proof}
	We note from \eqref{eq:phi''} that
	$
			-x^2\Phi''(x)\geq  x^2e^{-1}\Delta(x)
		$
	and the first claim of \eqref{eq:P} is valid. The second follows by integration of the first inequality and $\Phi(\infty)=\infty$ is a consequence of the integration of the second expression in \eqref{eq:P}.
\end{proof}
Now we are ready to show the smoothness of $f_\Phi$ under Assumption \eqref{def:condiA}.
%\begin{prop}\label{prop:D1}
%	Let $\Phi$ be the Laplace exponent of a potentially killed subordinator and assume that \eqref{def:condiA} holds. Then, for any $n\geq 0$ there exists $x_0(n,L)>0$ such that for any $a,t>0$ and any $x>x_0(n,L)$ condition \eqref{eq:Real2} holds. If $L=\infty$ then $x_0(n,\infty)=0$.
%\end{prop}
\begin{proof}[Proof of Theorem \ref{thm:regularityfphi1}]
	First observe that integrability in \eqref{eq:Real1} needs to be established only in  neighbourhood of infinity. We have from the inequality $1-\cos(y)\geq cy^2, y\in\lbbrbb{0,1}, c>0,$ that
	\[\Re\lbrb{\Phi(\ab)}-\Phi(a)=\IntOI \lbrb{1-\cos(by)}e^{-ay}\nu_\Phi(dy)\geq cb^2e^{-\frac{a}b} \int_{0}^{\frac1b}y^2\nu_\Phi(dy)=cb^2e^{-\frac{a}b}\Delta(b).\]
	Clearly, from Assumption \eqref{def:condiA}, for any $M \in (0,L)$, we have, for all $|b|>|b_0|>1$, that
	\begin{equation*}
		\begin{split}
			\int_{|b|>|b_0|} \!\! |b|^{n}e^{-x\Re\lbrb{\Phi(\ab)}}db\leq  e^{-x\Phi(a)}\!\!\!\int_{|b|>|b_0|}\!\!|b|^{n} e^{-xce^{-\frac{a}{|b_0|}}b^2\Delta(b)}db \leq e^{-x\Phi(a)}\!\!\!\int_{|b|>|b_0|} \!\! |b|^{n} e^{-xce^{-\frac{a}{|b_0|}}M\ln|b|}db.
		\end{split}
	\end{equation*}
	The latter is finite for $xce^{-\frac{a}{|b_0|}}M>n+1$. Since $|b_0|$ and $M \in (0,L)$ are arbitrary we have integrability of \eqref{eq:Real1} for $x>\frac{n+1}{cL}=:x_0(n,L)$ with $1/\infty=0$. Theorem \ref{thm:regularityfphi1} then follows by Proposition \ref{prop:LTthetapi}.
\end{proof}

\subsection{Proof of Theorem \ref{thm:mainL}}
We start with a preliminary result.
\begin{prop}\label{prop:D}
	Let $\Phi$ be the Laplace exponent of a potentially killed subordinator  and  \eqref{def:condiA} holds for some $L>0$. Let $t(x)$ be such that $t(x)/x \downarrow \mathfrak{b}, t(x)/x<\phi'(0^+)$ and $a_*(x)=(\phi')^{-1}\lbrb{\frac{t(x)}{x}}\in\lbrb{0,\infty}$. Then $\lim_{x \to \infty}a_*(x)=\infty$ and for any fixed $M>0$ it holds for all $x$ large enough
	\begin{equation}\label{eq:ineq}
		\begin{split}
			&a_*(x)>\frac{Me^{-1}}{\frac{t(x)}{x}-\mathfrak{b}}.
		\end{split}
	\end{equation}
\end{prop}
\begin{proof}
	Note that, since $t(x)/x \downarrow \mathfrak{b}$ and $\Phi'$ is decreasing with $\lim_{z \to \infty}\Phi'(z)=\mathfrak{b}$, then $\lim_{x \to \infty}a_*(x)=\infty$. Furthermore, for any $0<C<L$, using the second inequality of \eqref{eq:P} in Proposition \ref{prop:D0}, we get
		\[Ce^{-1}\leq\liminfi{x}\frac{a_*(x)\lbrb{\Phi'(a_*(x))-\mathfrak{b}}}{\ln a_*(x)}=\liminfi{x}\frac{a_*(x)\lbrb{\frac{t(x)}{x}-\mathfrak{b}}}{\ln a_*(x)}.\]
		This shows that for all $x$ and therefore $a_*(x)$ large enough
		\[\frac{a_*(x)}{\ln a_*(x)}>\frac{C}{2}\frac{e^{-1}}{\frac{t(x)}{x}-\mathfrak{b}}.\]
		Since $\limi{x}\ln a_*(x)=\infty$ this concludes the proof.
\end{proof}
Now, we are ready to prove the main theorem concerning the asymptotic behaviour.
\begin{proof}[Proof of Theorem \ref{thm:mainL}]
	Fix $k,l\geq 0$ and assume that $x>x_{0}(k+l,L)$, see Theorem \ref{thm:regularityfphi1}. Then, for any $a>0$, by \eqref{intreprder1} it holds, for any $t/x>\mathfrak{b}$, that
	\begin{equation}\label{def:derf}
		\frac{\partial^k\partial ^l }{\partial x^k\partial t^l}\fP{x,t}=\frac{(-1)^k}{2\pi}\int_{-\infty}^\infty \frac{\Phi^{\dagger}(\ab)\Phi^{k}(\ab)}{\lbrb{\ab}^{1-l}}e^{-x\Phi(\ab)+t\lbrb{\ab}}db=:I(x,t).
	\end{equation}
	%where we have simply differentiated the inversion of the Laplace transform in \eqref{eq:LT1} which is possible thanks to the absolute integrability of the expression in \eqref{intreprder1}.
	Let $t(x)$ be as in the statement of the theorem and $a_*(x)=(\phi')^{-1}\lbrb{\frac{t(x)}{x}}$, which is well-defined
	%be the solution to
	%\begin{equation*}
	%	\Phi'(a_*(x))=\frac{t(x)}{x}\in\lbrb{\mathfrak{b},\Phi'(0^+)},
	%\end{equation*}
	%which is unique 
	because $\Phi'(x)$ is decreasing with $\limi{x}\Phi'(x)=\mathfrak{b}$, see Item \eqref{it:phi'} of Lemma \ref{lem:Bern}. Since by assumption $t(x)/x\downarrow\mathfrak{b}$, we get from Proposition \ref{prop:D} that $\limi{x}a_*(x)\to\infty$.  Recall that in this theorem we have set
	$\mathbb{D}'=\curly{(t,x):x\mathfrak{b}<t\leq t(x)<x\phi'(0^+)}$, see \eqref{def:region}, and
	\begin{equation}\label{eq:c}
		c:=c(t,x)=(\phi')^{-1}\lbrb{\frac{t}{x}}, \ \mbox{ for }(t,x) \in \mathbb{D}'.
	\end{equation}
	From now on we work with $(t,x)\in \mathbb{D}'$ such that  $x>x_0(k+l,L)$.
	Since $(\phi')^{-1}$ is decreasing, for fixed $x$, $c(\cdot,x)$ is decreasing in $t$ on $\mathbb{D}'$ with $c(t(x),x)=a_*(x)$. For 
	$(t,x)\in\mathbb{D}'$ use \eqref{def:derf} with $a=c$ to get
	\begin{equation}\label{eq:f1}
		I(x,t)=\frac{(-1)^k}{2\pi}e^{ct-x\Phi(c)}\IntII\frac{\Phi^{\dagger}(c+ib)\Phi^{k}(c+ib)}{\lbrb{c+ib}^{1-l}}e^{ibt -x\lbrb{\Phi(c+ib)-\Phi(c}}db,
	\end{equation}
	where $c$ minimizes  $a \in (0,+\infty)\mapsto (at-x\Phi(a)) \in \R$.
	% Next, choose $g(c,x):=\sqrt{2\ln\lbrb{c\sqrt{-\phi''(c)x}}}$ but keep $g$ wherever it is more convenient and set
	Next, set 
	%\begin{equation}\label{eq:varesp}
	%	\begin{split}
		%		&\varepsilon:=\varepsilon(c,x):=\frac{g(c,x)}{ c\sqrt{-\Phi''(c)x}}.
		%	\end{split}
	%\end{equation} 
	\begin{equation}\label{eq:varesp}
		\begin{split}
			g(c,x):=\sqrt{2\ln\lbrb{c\sqrt{-\phi''(c)x}}} \qquad \mbox{ and } \qquad \varepsilon:=\varepsilon(c,x):=\frac{g(c,x)}{ c\sqrt{-\Phi''(c)x}}.
		\end{split}
	\end{equation} 
	We split the region of integration by setting  $\Ic_{\varepsilon}:=[-c\varepsilon,c\varepsilon]$. Put 
	\begin{equation}\label{def:Ie}
		I_\varepsilon(x,t):=\frac{(-1)^k}{2\pi}e^{ct-x\Phi(c)}\int_{\Ic_{\varepsilon}}	J(c,b)e^{ibt -x\lbrb{\Phi(c+ib)-\Phi(c)}}db,
	\end{equation}
	where
	\begin{equation}\label{def:Je1}
		\begin{split}
			J(c,b)&:=\frac{\Phi^{\dagger}(c+ib)\Phi^{k}(c+ib)}{\lbrb{c+ib}^{1-l}}.
		\end{split}
	\end{equation}
Using Taylor's formula and the definition of $c$ in \eqref{eq:c} we get
	\begin{equation}\label{def:Tay}
		\begin{split}
			&x(\Phi(c+ib)-\Phi(c))=ibt-x\frac{b^2}{2}\Phi''(c)-U(c,b),
		\end{split}
	\end{equation}
	where
	\begin{equation*}
		\begin{split}
			U(c,b):=&ibt -x\lbrb{\Phi(c+ib)-\Phi(c)}-x\frac{b^2}{2}\Phi''(c)=ix\int_0^b\int_0^v\int_0^w \phi'''(c+i\rho)d\rho dw dv.
		\end{split}
	\end{equation*}
	Then, recalling \eqref{eq:varesp}, we have  for the integral in \eqref{def:Ie}
	\begin{equation}\label{def:I'}
		\begin{split}
			&\int_{\Ic_{\varepsilon}}J(c,b)e^{ibt -x\lbrb{\Phi(c+ib)-\Phi(c)}}db=\int_{\Ic_{\varepsilon}}J(c,b)e^{\frac{b^2}{2}x\Phi''(c)\lbrb{1+2\frac{U(c,b)}{xb^2\Phi''(c)}}}db\\&=\!\frac{1}{\sqrt{-\Phi''(c)x}}\int_{-g(c,x)}^{g(c,x)}J\lbrb{c,\frac{u}{\sqrt{-\Phi''(c)x}}}\!e^{-\frac{u^2}{2}\lbrb{1-2\frac{U\lbrb{c,u/\sqrt{-\Phi''(c)x}}}{u^2}}}\!\!du,
		\end{split}	
	\end{equation}
	where we have used that $\Phi''(y)<0$, for $y>0$, see Item \eqref{it:phi'} of Lemma \ref{lem:Bern}. Here, we need \eqref{def:condiB}, i.e. $\limsupi{y}y\Phi'''(y)/(-\Phi''(y))=K<\infty$, which together with $\abs{\Phi'''(z)}\leq \Phi'''(\Re(z))$, for $\Re(z)>0$, see Item \eqref{it:real} in Lemma \ref{lem:Bern}, and the definition of $U(c,b)$ yields  that for all $x$ large enough 
	\begin{equation}\label{eq:estU}
		\begin{split}
			\bar{V}(c,x)&:=\sup_{ |u|\leq g(c,x)}V(c,x,u):=\sup_{ |u|\leq g(c,x)}\frac{2}{u^2}\abs{U\lbrb{c,\frac{u}{\sqrt{-\Phi''(c)x}}}}\\
			&\leq  \sup_{ |u|\leq g(c,x)}\frac{x|u|}{3x^{\frac32}\lbrb{-\Phi''(c)}^{\frac32}}\Phi'''(c)  \leq\frac{2K}{3}\frac{g(c,x) }{c\sqrt{-\Phi''(c)x}}=\frac{2K}{3}\varepsilon(c,x).
		\end{split}
	\end{equation}
	Next, setting 
	\begin{equation}\label{def:tJ}
		\begin{split}
			&\tilde{J}(c,u)=\frac{J(c,u)}{J(c,0)}-1,
		\end{split}
	\end{equation}
	we get from \eqref{def:I'} that
	\begin{equation}\label{def:I'1}
		\begin{split}
			&\int_{\Ic_{\varepsilon}}J(c,b)e^{ibt -x\lbrb{\Phi(c+ib)-\Phi(c)}}db	=\frac{J(c,0)}{\sqrt{-\Phi''(c)x}}\int_{-g(c,x)}^{g(c,x)}e^{-\frac{u^2}{2}\lbrb{1+V(c,x,u)}}du\\
			&+\frac{J(c,0)}{\sqrt{-\Phi''(c)x}}\int_{-g(c,x)}^{g(c,x)}\tilde{J}\lbrb{c,\frac{u}{\sqrt{-\Phi''(c)x}}}e^{-\frac{u^2}{2}\lbrb{1+V(c,x,u)}}du=:H_1(c,x)+H_2(c,x).
		\end{split}
	\end{equation}
	Clearly, from \eqref{def:Je1} and \eqref{eq:estU},
	\begin{equation}\label{eq:H1}
		\begin{split}
			&\abs{H_1(c,x)-\sqrt{2\pi}\frac{J(c,0)}{\sqrt{-\Phi''(c)x}}}=\frac{J(c,0)}{\sqrt{-\Phi''(c)}x}\abs{\int_{-g(c,x)}^{g(c,x)}e^{-\frac{u^2}{2}\lbrb{1+V(c,x,u)}}du-\int_{-\infty}^{\infty}e^{-\frac{u^2}{2}}du}\\
			&\leq \sqrt{2\pi}\frac{J(c,0)}{\sqrt{-\Phi''(c)x}}\lbrb{e^{\bar{V}(c,x)}-1+\frac{1}{\sqrt{2\pi}}\int_{|u|>g(c,x)}e^{-\frac{u^2}{2}}du}.
		\end{split}
	\end{equation}
	From the definition of $J$, see \eqref{def:Je1}, we easily get that
	\begin{equation}\label{def:tJ1}
		\begin{split}
			\abs{\frac{d}{du}\log J\lbrb{c,u}}	&\leq \abs{\frac{(\Phi^{\dagger})'(c+iu)}{\Phi^{\dagger}(c+iu)}}+k\abs{\frac{\Phi'(c+iu)}{\Phi(c+iu)}}+\abs{l-1}\abs{\frac{1}{c+iu}}\leq \frac{k+1+\abs{l-1}}{c},
		\end{split}
	\end{equation}
	where we have used that for any Bernstein function $\phi$ and $c>0,u\in\Rb$,
	\[\abs{\frac{\phi'(c+iu)}{\phi(c+iu)}}\leq \frac1c,\]
	which in turn follows from the chain of inequalities
	\[\abs{\phi'(c+iu)}\leq \phi'(c)\leq \frac{\phi(c)}{c}\leq \frac{\Re\phi(c+iu)}{c}\leq \abs{\frac{\phi(c+iu)}{c}}\]
	that come from subsequent application of Items \ref{it:real}, \ref{it:ineq} and \ref{it:sign} of Lemma \ref{lem:Bern}. From \eqref{def:tJ1} we get that
	\begin{equation}\label{aim}
		\begin{split}
			&\abs{\ln\abs{1+\tilde{J}(c,u)}}\leq \abs{\log\lbrb{1+\tilde{J}(c,u)}}=\abs{\log \lbrb{1+\lbrb{\frac{J\lbrb{c,u}}{J(c,0)}-1}}}\leq (k+1+\abs{l-1})\frac{|u|}{c}.
		\end{split}
	\end{equation}
	From \eqref{aim} with \eqref{eq:varesp},
	\begin{equation}\label{aim2}
		\begin{split}
			\sup_{|u|\leq \frac{g(c,x)}{\sqrt{-\Phi''(c)x}}}\abs{\ln\abs{1+\tilde{J}(c,u)}}\leq (k+1+\abs{l-1})\varepsilon(c,x).
		\end{split}
	\end{equation}
	However, $c\geq a_*(x)$, $\limi{x}a_*(x)=\infty$ and $\limi{x}x\sqrt{-\Phi''(x)}=\infty$, see \eqref{eq:P}, imply that 
	\begin{equation}\label{eq:bare}
		\limi{x}\bar{\varepsilon}(x):=\limi{x}\sup_{c\geq a_*(x)}\varepsilon(c,x)=0.
	\end{equation}
	Thus, for all $x$ large enough, \eqref{aim2} yields that, for some $C=C(k,l)$,
	\begin{equation}\label{eq:aim1}
		\sup_{|u|\leq \frac{g(c,x)}{\sqrt{-\Phi''(c)x}}}\abs{\tilde{J}(c,u)}\leq C\varepsilon(c,x).
	\end{equation}
	Hence, from \eqref{eq:estU}, \eqref{def:tJ} and \eqref{def:I'1} for all $x$ large enough
	\begin{equation}\label{estimateU}
		\begin{split}
			H_2(c,x)&\leq e^{\bar{V}(c,x)}\frac{J(c,0)}{\sqrt{-\Phi''(c)x}}\int_{-g(c,x)}^{g(c,x)}\abs{\tilde{J}\lbrb{c,\frac{u}{\sqrt{-\Phi''(c)x}}}}e^{-\frac{u^2}{2}}du\\
			&\leq \sqrt{2\pi}Ce^{\bar{V}(c,x)}\frac{J(c,0)}{\sqrt{-\Phi''(c)x}}\varepsilon(c,x).
		\end{split}
	\end{equation}
	Combining \eqref{eq:estU}, \eqref{def:I'1}, \eqref{eq:H1} and \eqref{estimateU} we obtain for all large $x$
	\begin{equation}\label{eq:H2}
		\begin{split}
			&\abs{H_1(c,x)-\sqrt{2\pi}\frac{J(c,0)}{\sqrt{-\Phi''(c)x}}+H_2(c,x)}\\
			&\leq \sqrt{2\pi}\frac{J(c,0)}{\sqrt{-\Phi''(c)x}}\lbrb{Ce^{\bar{V}(c,x)}\varepsilon(c,x)+e^{\bar{V}(c,x)}-1+\frac{1}{\sqrt{2\pi}}\int_{|u|>g(c,x)}e^{-\frac{u^2}{2}}du}\\
			&\leq \sqrt{2\pi}\frac{J(c,0)}{\sqrt{-\Phi''(c)x}}\lbrb{Ce^{\frac{2K}{3}\varepsilon(c,x)}\varepsilon(c,x)+e^{\frac{2K}{3}\varepsilon(c,x)}-1+\frac{1}{\sqrt{2\pi}}\int_{|u|>g(c,x)}e^{-\frac{u^2}{2}}du}.
		\end{split}
	\end{equation}
	Applying \eqref{eq:H2} and \eqref{eq:bare}
	in \eqref{def:I'1} we get with some $C'=C(k,l,K)>0$ that for all $x$ large enough
	\begin{equation}\label{eq:H3}
		\begin{split}
			&\sup_{c\geq a_*(x)}\abs{\frac{1}{\sqrt{2\pi}}\frac{\sqrt{-\Phi''(c)x}}{J(c,0)}\int_{\Ic_{\varepsilon}}J(c,b)e^{ibt -x\lbrb{\Phi(c+ib)-\Phi(c)}}db-1}\\
			&\leq C'\bar{\varepsilon}(x)+\frac{1}{\sqrt{2\pi}}\sup_{c\geq a_*(x)}\int_{|u|>g(c,x)}e^{-\frac{u^2}{2}}du.
		\end{split}
	\end{equation}
	Plugging this in \eqref{def:Ie} we get for all $x$ large enough
	\begin{equation}\label{def:asympIe1}
		\begin{split}		
			&\sup_{\mathfrak{b}x<t\leq t(x)}\abs{\frac{(-1)^k}{\sqrt{2\pi}}\frac{\sqrt{-\Phi''(c)x}}{J(c,0)}e^{-ct+x\phi(c)}I_\varepsilon(x,t)-1}\leq C'\bar{\varepsilon}(x)+\frac{1}{\sqrt{2\pi}}\sup_{c\geq a_*(x)}\int_{|u|>g(c,x)}e^{-\frac{u^2}{2}}du.
		\end{split}
	\end{equation}
	Using \eqref{def:Je1} we proceed to investigate
	\begin{equation}\label{def:Je}
		\begin{split}
			J_\varepsilon(x,t)&:=\frac{(-1)^k}{2\pi}e^{ct-x\Phi(c)}\int_{\Ic^c_{\varepsilon}}\frac{\Phi^\dagger(c+ib)\Phi^{k}(c+ib)}{\lbrb{c+ib}^{1-l}}e^{ibt -x\lbrb{\Phi(c+ib)-\Phi(c)}}db\\
			&=\frac{(-1)^k}{2\pi}e^{ct-x\Phi(c)}\int_{\varepsilon c\leq |b|\leq dc}\!\!\!\!\!\!\!\!J(c,b)e^{ibt -x\lbrb{\Phi(c+ib)-\Phi(c)}}db\\
			&+\frac{(-1)^k}{2\pi}e^{ct-x\Phi(c)}\int_{|b|\geq dc}J(c,b)e^{ibt -x\lbrb{\Phi(c+ib)-\Phi(c)}}db\\
			&=:\frac{(-1)^k}{2\pi}e^{ct-x\Phi(c)}\lbrb{J_1(c,x)+J_2(c,x)},
		\end{split}
	\end{equation}
	where $d=K^{-1}$. First we estimate  $J_1(c,x)$. We use the Taylor expansion \eqref{def:Tay} to the exponent in  $J_1(c,x)$, $\abs{\Im\Phi'''(z)}\leq \abs{\Phi'''(z)}\leq \Phi'''\lbrb{\Re(z)}$, $\limsupi{y}y\Phi'''(y)/(-\Phi''(y))=K<\infty$ to get after a change of variables $b\to cb$, for all $x$ and therefore $c\geq a_*(x)$ large enough
	\begin{equation*}
		\begin{split}
			\abs{J_1(c,x)}&\leq 2c\int_{\varepsilon}^d\abs{J(c,cb)}e^{\frac{xb^2c^2}{2}\Phi''(c)+\frac{xb^3c^3}{6}\Phi'''(c)}db\leq 2c\int_{\varepsilon}^d\abs{J(c,cb)}e^{\frac{xb^2c^2}{2}\Phi''(c)\lbrb{1-\frac{2Kd}{3}}}db\\
			&= 2c^l\int_{\varepsilon}^d\frac{\abs{\Phi^\dagger(c(1+ib))\Phi^{k}\lbrb{c(1+ib)}}}{\abs{1+ib}^{1-l}}e^{\frac{xb^2c^2}{6}\Phi''(c)}db.
		\end{split}
	\end{equation*}
	To carry on further we note from \eqref{eq:DeltaR} that 
	$
	\abs{\frac{\Phi\lbrb{c\lbrb{1+ib}}}{\Phi(c)}}\leq 3\max\curly{1,b^2}.
	$
	Then, we have with the form of $\varepsilon$, see \eqref{eq:varesp}, and $d=K^{-1}$ that 
	\begin{equation}\label{eq:J1}
		\begin{split}
			&\abs{\frac{\sqrt{-\Phi''(c)x}}{J(c,0)}J_1(c,x)}\leq 3^{k+2}c\sqrt{-\Phi''(c)x}\int_{\varepsilon}^d \max\curly{1,b^{2k+2+l}}e^{\frac{xb^2c^2}{6}\Phi''(c)}db\\
			&\leq 3^{k+2}\max\curly{1,d^{2k+2+l}}c\sqrt{-\Phi''(c)x}\int_{\varepsilon}^de^{\frac{xb^2c^2}{6}\Phi''(c)}db\\
			&\leq 3^{k+\frac52}\max\curly{1,K^{-2k-2-l}}\int_{3^{-\frac12}g(c,x)}^\infty e^{-\frac{u^2}{2}}du.
		\end{split}
	\end{equation}
	For $J_2(c,x)$ we change variables $b\to cb$ and use again $
	\abs{\frac{\Phi\lbrb{c\lbrb{1+ib}}}{\Phi(c)}}\leq 3\max\curly{1,b^2}
	$  to get
	\begin{equation*}
		\begin{split}
			&\abs{\frac{\sqrt{-\Phi''(c)x}}{J(c,0)}J_2(c,x)}\leq 3^{k+1}c\sqrt{-\Phi''(c)x}\int_{|b|\geq d}^\infty \frac{\max\curly{1,b^{2k+2}}}{\abs{1+ib}^{-l+1}}e^{ -x\lbrb{\Re\lbrb{\Phi(c\lbrb{1+ib})}-\Phi(c)}}db\\
			&\leq C_1c\sqrt{-\Phi''(c)x}\int_{d}^\infty b^{2k+1+l}e^{ -x\lbrb{\Re\lbrb{\Phi(c\lbrb{1+ib})}-\Phi(c)}}db\\
			&=C_1d^{2k+2+l}c\sqrt{-\Phi''(c)x}\int_{1}^\infty b^{2k+1+l}e^{ -x\lbrb{\Re\lbrb{\Phi(c\lbrb{1+ibd})}-\Phi(c)}}db,
		\end{split}
	\end{equation*}
	where $C_1>0$ is some constant.   
	Then, with some absolute constant $c_0>0$, we have that 
	\begin{equation}\label{eq:Delta}
		\begin{split}
			& \Re\lbrb{\Phi(c\lbrb{1+ibd})}-\Phi(c)=\IntOI \lbrb{1-\cos\lbrb{bdcy}}e^{-cy}\mu(dy)\\
			&\geq c_0^2e^{-\frac{1}{bd}}b^2d^2c^2\int_{0}^{\frac{1}{bdc}}y^{2}\mu(dy)=c_0^2e^{-\frac{1}{bd}}b^2d^2c^2\Delta(bdc).
		\end{split}
	\end{equation}
	Since $\liminfi{y}y^2\Delta(y)/\ln(y)= L>0$ we choose $M<L$. Set $M'=Mc^2_0e^{-\frac{1}{d}}$. On $b\geq 1$ we get that for $x$ and $c\geq a_*(x)$ large enough such that $M'x>2k+2+l$
	\begin{equation}\label{eq:J_2}
		\begin{split}
			&\abs{\frac{\sqrt{-\Phi''(c)x}}{J(c,0)}J_2(c,x)}\leq C_1d^{2k+2+l}c\sqrt{-\Phi''(c)x}\int_{1}^\infty b^{2k+1+l}e^{ -M'x \ln\lbrb{bdc}}db\\
			&=\frac{C_1d^{2k+2+l}}{M'x-2k-2-l}\frac{c\sqrt{-\Phi''(c)x}}{(cd)^{M'x}}\leq C_2\frac{K^{M'x-2k-2-l}}{(M'x-2k-2-l)}\sqrt{x} c^{\frac12-M'x},
		\end{split}
	\end{equation}
	where $C_2>0$ is some constant, we have fixed $d=K^{-1}$ and we have used $-y^{2}\Phi''(y)\leq 2\Phi(y)=\bo{y}$, see Item \eqref{it:ineq} of Lemma \ref{lem:Bern}. Collecting \eqref{def:asympIe1} and employing \eqref{eq:J1} and \eqref{eq:J_2} in \eqref{def:Je}, we get that, since $c\geq a_*(x)$, \eqref{eq:f1} has the following  form for all $x$ large enough
	\begin{equation*}
		\begin{split}
			&	\sup_{\mathfrak{b}x<t\leq t(x)}\abs{(-1)^k\sqrt{2\pi}\frac{\sqrt{-\Phi''(c)x}}{J(c,0)}e^{-ct+x\phi(c)}I(x,t)-1}\\
			&\leq  C'\bar{\varepsilon}(x)+\frac{1}{\sqrt{2\pi}}\sup_{c\geq a_*(x)}\int_{|u|>g(c,x)}e^{-\frac{u^2}{2}}du\\
			&+3^{k+\frac52}\max\curly{1,K^{-2k-2-l}}\sup_{c\geq a_*(x)}\int_{3^{-\frac12}g(c,x)}^\infty e^{-\frac{u^2}{2}}du+C_2\frac{K^{M'x-2k-2-l}}{(M'x-2k-2-l)}\sqrt{x} (a_*)^{\frac12-M'x}(x).
		\end{split}
	\end{equation*}
	Next, recall that $g(c,x)=\sqrt{2\ln\lbrb{c\sqrt{-\Phi''(c)x}}}$ which converges to infinity thanks to Proposition \ref{prop:D} and the form of $\varepsilon(c,x),\bar{\varepsilon}(x)$, see \eqref{eq:varesp} and \eqref{eq:bare}. We then yield asymptotically
	\begin{equation*}
		\begin{split}
			&	\sup_{\mathfrak{b}x<t\leq t(x)}\abs{(-1)^k\sqrt{2\pi}\frac{\sqrt{-\Phi''(c)x}}{J(c,0)}e^{-ct+x\phi(c)}I(x,t)-1}\\
			&=\bo{\sup_{c\geq a_*(x)}\frac{\sqrt{\ln\lbrb{c\sqrt{-\Phi''(c)x}}}}{ c\sqrt{-\Phi''(c)x}}+\sup_{c\geq a_*(x)}\int_{\sqrt{\frac{2}3\ln\lbrb{c\sqrt{-\Phi''(c)x}}}}^\infty e^{-\frac{u^2}{2}}du+ \frac{K^{M'x}(a_*(x))^{\frac12-M'x}(x)}{\sqrt{x}}}.
		\end{split}
	\end{equation*}
	However, since $\limi{x}-x^2\phi''(x)=\infty$, see \eqref{eq:P}, and $\int_x^{\infty} e^{-u^2/2}du=\bo{x^{-1}e^{-x^2/2}}$ we further get
	\begin{equation*}
		\begin{split}
			&		\sup_{\mathfrak{b}x<t\leq t(x)}\abs{(-1)^k\sqrt{2\pi}\frac{\sqrt{-\Phi''(c)x}}{J(c,0)}e^{-ct+x\phi(c)}I(x,t)-1}\\
			&=\bo{\sup_{c\geq a_*(x)}\frac{\sqrt{\ln\lbrb{c\sqrt{-\Phi''(c)x}}}}{ c\sqrt{-\Phi''(c)x}}+ \frac{K^{M'x}(a_*(x))^{\frac12-M'x}(x)}{\sqrt{x}}}.
		\end{split}
	\end{equation*}
	Since $-y^2\Phi''(y)\leq 2\Phi(y)=\bo{y},$ see Item \eqref{it:ineq} of Lemma \ref{lem:Bern}, we check  that the first expression in the speed of convergence cannot be faster than $(xa_*(x))^{-\frac12}$. Therefore, we conclude that
	\begin{equation*}
		\begin{split}
			&	\sup_{\mathfrak{b}x<t\leq t(x)}\abs{(-1)^k\sqrt{2\pi}\frac{\sqrt{-\Phi''(c)x}}{J(c,0)}e^{-ct+x\phi(c)}I(x,t)-1}=\bo{\sup_{c\geq a_*(x)}\frac{\sqrt{\ln\lbrb{c\sqrt{-\Phi''(c)x}}}}{ c\sqrt{-\Phi''(c)x}}}.
		\end{split}
	\end{equation*}
	which establishes \eqref{asymp}. Finally, from the first  relation of \eqref{eq:P} of Proposition \ref{prop:D0} we have for all $x$ large enough that $c\sqrt{-\phi''(c)}\geq Me^{-1}\ln(c)\geq Me^{-1}\ln a_*(x)$, $M<L$,  which yields
	\[\inf_{c\geq a_*(x)}c\sqrt{-\phi''(c)}\geq Me^{-1}\ln(a_*(x)).\]
	Employing this and again  $-y^2\Phi''(y)\leq 2\Phi(y)=\bo{y},$ we arrive at 
	\begin{equation*}
		\begin{split}
			&		\sup_{\mathfrak{b}x<t\leq t(x)}\abs{(-1)^k\sqrt{2\pi}\frac{\sqrt{-\Phi''(c)x}}{J(c,0)}e^{-ct+x\phi(c)}I(x,t)-1}=\bo{\sqrt{\frac{\ln(x)}{x\ln a_*(x)}}}.
		\end{split}
	\end{equation*}
	Substituting in the latter the expression for $J(c,0)$, see \eqref{def:Je1}, concludes the proof of \eqref{asymp}. Relation $1/a_*=\so{t(x)/x-\mathfrak{b}}$ follows from Proposition \ref{prop:D}.
\end{proof}

\subsection{Proofs of Theorems \ref{thm:main1} and \ref{thm:main2}}
Once the proof of Theorem \ref{thm:mainL} has been established, similar arguments will lead to Theorems \ref{thm:main1} and \ref{thm:main2}. For this reason, we will be economical with the next proofs, as we will refer to the previous arguments while highlighting the necessary changes and adaptations.
%We will be economical with the next proof as it follows the main steps of the proof of Theorem \ref{thm:mainL}.
We proceed with the proof of the next main theorem.
\begin{proof}[Proof of Theorem \ref{thm:main1}]
	We follow closely the proof of Theorem \ref{thm:mainL} using in particular $J(c,b)$ defined in \eqref{def:Je1}. By assumption $(t,x)\in\mathbb{D}^\prime=\{(t,x): \ xt_1 \leq t \leq xt_2\}$ and since $(\phi')^{-1}$ is decreasing then \[c:=c(t,x)=(\phi')^{-1}(t/x)\in \lbbrbb{(\phi')^{-1}(t_2), (\phi')^{-1}(t_1)}=:\mathbb{V}.\] From Theorem \ref{thm:regularityfphi1} we can write for all $x$ large enough and $t/x\in\lbbrbb{t_1,t_2}$
	\begin{equation*}
		\begin{split}
			&\frac{\partial^k\partial ^l }{\partial x^k\partial t^l}\fP{x,t}=\frac{(-1)^k}{2\pi}\int_{-\infty}^\infty J(c,b)e^{-x\Phi(c+ib)+t\lbrb{c+ib}}db=:I(x,t).
		\end{split}
	\end{equation*}
	We use $g(c,x):=\sqrt{2\ln\lbrb{c\sqrt{-\phi''(c)x}}}$ as defined in \eqref{eq:varesp}. Since $c\in \mathbb{V}$ we can repeat the arguments leading up to \eqref{def:I'} with the same definition of $\varepsilon(c):=\varepsilon(c,x)=\frac{g(c,x)}{ c\sqrt{-\Phi''(c)x}}$ and the estimate 
	\begin{equation}\label{eq:estU_1}
		\begin{split}
			\bar{V}(c,x)&:=\sup_{ |u|\leq g(c,x)}V(c,x,u):=\sup_{ |u|\leq g(c,x)}\frac{2}{u^2}\abs{U\lbrb{c,\frac{u}{\sqrt{-\Phi''(c)x}}}}\leq \frac{2C}3\varepsilon(c,x),
		\end{split}
	\end{equation}
	where $C=C(t_1,t_2)$ since in the upper bound prior to \eqref{eq:varesp} we can employ
	\begin{align*}
		\phi'''(c)\leq \phi'''((\phi')^{-1}(t_2)), \quad  c\leq (\phi')^{-1}(t_1), \quad \mbox{ and } \quad  -\phi''(c)\geq (\phi')^{-1}(t_1),
	\end{align*}
	hence we do not need Assumption \ref{def:condiB}.
	%here we estimate $\phi'''(c)\leq \phi'''((\phi')^{-1}(t_2)), c\leq (\phi')^{-1}(t_1), -\phi''(c)\geq (\phi')^{-1}(t_1)$ in the upper bound prior to \eqref{eq:varesp}.
	%This is why we do not need Assumption \ref{def:condiB}. 
	Since \eqref{eq:bare} is valid with the modification
	\begin{equation}\label{eq:bare_1}
		\limi{x}\bar{\varepsilon}(x):=\limi{x}\sup_{c\in\mathbb{V}}\varepsilon(c,x)=0
	\end{equation}
	we similarly arrive at 
	\begin{equation}\label{def:asympIe1_1}
		\begin{split}		
			&\sup_{xt_1 \le t \le xt_2}\abs{(-1)^k\sqrt{2\pi}\frac{\sqrt{-\Phi''(c)x}}{J(c,0)}e^{-ct+x\phi(c)}I_\varepsilon(x,t)-1}\leq C'\bar{\varepsilon}(x)+\frac{1}{\sqrt{2\pi}}\sup_{c\in\mathbb{V}}\int_{|u|>g(c,x)}e^{-\frac{u^2}{2}}du.
		\end{split}
	\end{equation}	
	The same remainder as \eqref{def:Je}, i.e.
	\begin{equation}\label{def:Je_1}
		\begin{split}
			J_\varepsilon(x,t)&:=\frac{(-1)^k}{2\pi}e^{ct-x\Phi(c)}\int_{\Ic^c_{\varepsilon}}J(c,b)e^{ibt -x\lbrb{\Phi(c+ib)-\Phi(c)}}db\\
			&=\frac{(-1)^k}{2\pi}e^{ct-x\Phi(c)}\int_{\varepsilon(c)c\leq |b|\leq dc}J(c,b)e^{ibt -x\lbrb{\Phi(c+ib)-\Phi(c)}}db\\
			&+\frac{(-1)^k}{2\pi}e^{ct-x\Phi(c)}\int_{|b|\geq dc}J(c,b)e^{ibt -x\lbrb{\Phi(c+ib)-\Phi(c)}}db\\
			&=:\frac{(-1)^k}{2\pi}e^{ct-x\Phi(c)}\lbrb{J_1(c,x)+J_2(c,x)},
		\end{split}
	\end{equation}
	is studied similarly with $d=C^{-1}$, see \eqref{eq:estU_1}. First, following the same computations one gets
	\begin{equation}\label{eq:J1_1}
		\begin{split}
			&\abs{\frac{\sqrt{-\Phi''(c)x}}{J(c,0)}J_1(c,x)}\leq 3^{k+\frac52}\max\curly{1,C^{-2k-2-l}}\int_{3^{-\frac12}g(c,x)}^\infty e^{-\frac{u^2}{2}}du.
		\end{split}
	\end{equation}
	For the second term we get
	\begin{equation*}
		\begin{split}
			&\abs{\frac{\sqrt{-\Phi''(c)x}}{J(c,0)}J_2(c,x)}\leq C_1C^{-2k-2-l}c\sqrt{-\Phi''(c)x}\int_{1}^\infty b^{2k+1+l}e^{ -x\lbrb{\Re\lbrb{\Phi(c\lbrb{1+ibd})}-\Phi(c)}}db,
		\end{split}
	\end{equation*}
	with  as in \eqref{eq:Delta} 
	\begin{equation}\label{eq:Delta_1}
		\begin{split}
			& \Re\lbrb{\Phi(c\lbrb{1+ibd})}-\Phi(c)\geq c_0^2e^{-\frac{1}{bd}}b^2d^2c^2\Delta(bdc)\geq Ab^2.
		\end{split}
	\end{equation}
	for some  $A>0$ because $bdc\geq (\phi')^{-1}(t_2)d>0$. As in \eqref{eq:J_2} for $x$ large enough
	\begin{equation*}
		\begin{split}
			&\abs{\frac{\sqrt{-\Phi''(c)x}}{J(c,0)}J_2(c,x)}\leq C_1C^{-2k-2-l}c\sqrt{-\Phi''(c)x}e^{-\frac{A}2 x}\leq C_1C^{-2k-2-l}\sqrt{2cx}e^{-\frac{A}2 x} .
		\end{split}
	\end{equation*}
	Collecting the estimates above and noting that, as $c$ ranges in the bounded set $\mathbb{V}$, $\bar{\varepsilon}(x)\asymp \sqrt{\ln(x)}x^{-\frac12}$, see \eqref{eq:varesp} and \eqref{eq:bare}, we deduce
	\begin{equation*}
		\begin{split}
			&	\sup_{xt_1 \le t \le xt_2}\abs{(-1)^k\sqrt{2\pi}\frac{\sqrt{-\Phi''(c)x}}{J(c,0)}e^{-ct+x\phi(c)}I(x,t)-1}=\bo{\sqrt{\frac{\ln(x)}{x}}}.
		\end{split}
	\end{equation*}
	This concludes the proof of the theorem.
\end{proof}
%We will be economical with the next proof as it follows the main steps of the proof of Theorem \ref{thm:mainL}.
Finally, we only need to provide the proof of Theorem \ref{thm:main2}.
\begin{proof}[Proof of Theorem \ref{thm:main2}] Since $t/x\to \phi'(0^+)$ then from \eqref{eq:a*} we have that $\limi{x}a_*=0$, where $a_*:=a_*(x)$.
	Then, using \eqref{intreprder1} for $x>x_0(k+l,L)$ we write,  for any $t/x>\mathfrak{b}$,
	\begin{equation}\label{def:derf1}
		\begin{split}
			\frac{\partial^k\partial ^l }{\partial x^k\partial t^l}\fP{x,t}&=\frac{(-1)^k}{2\pi}e^{a_*t-x\Phi(a_*)}\lbrb{\int_{\Ic_{\varepsilon}}+\int_{\Ic^c_{\varepsilon}}}J(a_*,b)e^{ibt -x\lbrb{\Phi(a_*+ib)-\Phi(a_*)}}db,
		\end{split}
	\end{equation}
	where $\Ic_{\varepsilon}=\lbbrbb{-\varepsilon(a_*)a_*,\varepsilon(a_*)a_*}$, $J$ defined as in \eqref{def:Je1} and $\varepsilon(a_*)$ as in \eqref{eq:varesp}. The latter makes sense due to the first assumption in \eqref{eq:addCondi}, i.e. $\limi{x} -x\phi''(a_*)a^2_*=\infty$.
	
	We appeal to the proof of Theorem \ref{thm:mainL} with the minor difference that we do not work with a region for $t$ and respectively with $c\geq a_*$ but we use $t=t(x)$ and $a_*=a_*(x)$.
	Absolutely the same estimates hinging on \eqref{def:condiA}, \eqref{def:condiB'} and \eqref{eq:varesp}  and employing  $g(a_*,x)=\sqrt{2\ln\lbrb{a_*\sqrt{-\Phi''(a_*)x}}}$, which drifts to infinity thanks to the first assumption in \eqref{eq:addCondi}, lead  as in the proof of Theorem \ref{thm:mainL} to
	\begin{equation}\label{I'2'}
		\begin{split}
			&\int_{\Ic_{\varepsilon}}J(a_*,b)e^{ibt -x\lbrb{\Phi(a_*+ib)-\Phi(a_*)}}db=\sqrt{2\pi}\frac{J(a_*,x)}{\sqrt{-\Phi''(a_*)x}}\lbrb{1+\bo{\frac{g(a_*,x) }{a_*\sqrt{-\Phi''(a_*)x}}}}.
		\end{split}
	\end{equation}
	Next, we again decompose for some $d>0$ the remaining term on $\mathcal{I}^c_\varepsilon$ to yield
	\begin{equation}\label{def:Je2}
		\begin{split}
			J_\varepsilon(x,t)&:=\frac{(-1)^k}{2\pi}e^{a_*t-x\Phi(a_*)}\int_{\varepsilon(a_*)a_*\leq |b|\leq da_*}J(a_*,b)e^{ibt -x\lbrb{\Phi(a_*+ib)-\Phi(a_*)}}db\\
			&+\frac{(-1)^k}{2\pi}e^{a_*t-x\Phi(a_*)}\int_{|b|\geq da_*}J(a_*,b)e^{ibt -x\lbrb{\Phi(a_*+ib)-\Phi(a_*)}}db\\
			&=:\frac{(-1)^k}{2\pi}e^{a_*t-x\Phi(a_*)}\lbrb{J_1(a_*,x)+J_2(a_*,x)},
		\end{split}
	\end{equation}
	and get the same way $J_1(a_*,x)$ estimated as in \eqref{eq:J1}. Also, in the same fashion
	\begin{equation*}
		\begin{split}
			&\abs{\frac{\sqrt{-\Phi''(a_*)x}}{J(a_*,0)}J_2(a_*,x)}=C'd^{2k+2+l}a_*\sqrt{-\Phi''(a_*)x}\int_{1}^\infty b^{2k+1+l}e^{ -x\lbrb{\Re\lbrb{\Phi(a_*\lbrb{1+ibd})}-\Phi(a_*)}}db,
		\end{split}
	\end{equation*}
	where $C'>0$ is a constant. Here, however, there is a difference in that $a_*\to 0$, hence we  estimate the exponent in two different ways. Pick $A>1$. Then on $1\leq b\leq A/(da_*)$ we have as in \eqref{eq:Delta}
	\begin{equation*}
		\begin{split}
			& \Re\lbrb{\Phi(a_*\lbrb{1+ibd})}-\Phi(a_*)=\IntOI \lbrb{1-\cos\lbrb{bda_*y}}e^{-a_*y}\mu(dy)\\
			&\geq c_0^2e^{-\frac{1}{bd}}b^2d^2a^2_*\int_{0}^{\frac{1}{bda_*}}y^{2}\mu(dy)\geq c_0^2e^{-\frac{1}{bd}}b^2d^2a^2_*\Delta(A).
		\end{split}
	\end{equation*}
	Since the right-hand side of the latter is bounded from below by a constant then, for some $a>0, C>0$,
	\[\int_{1}^{\frac{A}{da_*}} b^{2k+1+l}e^{ -x\lbrb{\Re\lbrb{\Phi(a_*\lbrb{1+ibd})}-\Phi(a_*)}}db\leq Ce^{-ax}.\]
	Next, for any $M<L$, see \eqref{def:condiA}, choose $A$ such that, for $b>A/(da_*)$, as in \eqref{eq:Delta},
	\begin{equation*}
		\begin{split}
			& \Re\lbrb{\Phi(a_*\lbrb{1+ibd})}-\Phi(a_*)\geq c_0^2b^2d^2a^2_*e^{-\frac{1}{bd}}\Delta(bda_*)\geq MA^2c^2_0e^{-\frac{a_*}{A}}\ln\lbrb{bda_*}.
		\end{split}
	\end{equation*}
	Therefore, for all $x$ large enough and hence $a_*$ small enough we get with some $M'$ as large as we wish
	\begin{equation*}
		\begin{split}
			&\int_{\frac{A}{da_*}}^\infty b^{2k+1+l}e^{ -x\lbrb{\Re\lbrb{\Phi(a_*\lbrb{1+ibd})}-\Phi(a_*)}}db\leq \int_{\frac{A}{da_*}}^\infty b^{2k+1+l}\lbrb{bda_*}^{-M'x}db\\
			&=a^{-2k-l-2}_{*}d^{-M'x}\int_{\frac{A}{d}}^\infty u^{2k+1+l}u^{-M'x}du=a^{-2k-l-2}_{*}d^{-2k-2-l}A^{-M'x}\frac{1}{M'x-2k-l-2}.
		\end{split}
	\end{equation*}
	Plugging this and the estimate above we arrive thanks to the second and third requirement of  \eqref{eq:addCondi}, with some $C_0>0$ and $a'>0$, at the bound which settles the claim
	\begin{equation*}
		\begin{split}
			\abs{\frac{\sqrt{-\Phi''(a_*)x}}{J(a_*,0)}J_2(a_*,x)}&\leq C_0 a_*\sqrt{-\Phi''(a_*)x}\lbrb{e^{-ax}+\frac{2}{M' x}e^{-M'x\ln(A)+(2k+l+2)\ln\lbrb{\frac{1}{a_*}}}}\\
			&\leq 2C_0 a_*\sqrt{-\Phi''(a_*)x}e^{-a'x}\leq 2C_0 e^{-a''x}.
		\end{split}
	\end{equation*} 
\end{proof}

\subsection{Proofs of Lemmae \ref{lem:phi''} and \ref{lem:condi}}
\label{subsec:aux}
We now prove the lemmae that we used in Subsection \ref{subsec:L} to compare our results with the existing ones. 
Recall that $\Delta(x):=\int_0^{\frac1x}y^2\nu_\Phi(dy),\,x>0,$ see \eqref{def:D}.
\begin{proof}[Proof of Lemma \ref{lem:phi''}]
	Since $-\phi''(x)=\IntOI e^{-xy}y^2\mu_{\Phi}(dy)$ we get that
	$-\phi''(x)\geq e^{-1}\int_{0}^{\frac{1}{x}}y^2\mu_{\Phi}(dy)$
	and the lower bound in \eqref{eq:phi''} follows. Then, using $ve^{-v}\leq e^{-1}$, for $v\geq1,$ upper bound follows from
	\begin{equation*}
		\begin{split}
			-\phi''(x)	&\leq \int_{0}^{\frac{1}{x}}y^2\mu_{\Phi}(dy)+\frac{1}{x^2}\int_{\frac{1}{x}}^{\infty}e^{-xy}x^2y^2\mu_{\Phi}(dy)
			\leq \Delta(x)+\frac{e^{-1}}{x^2} \bar{\mu}_\phi\lbrb{\frac{1}{x}}.
		\end{split}
	\end{equation*}
\end{proof}

\begin{proof}[Proof of Lemma \ref{lem:condi}]
	From Item \ref{it:phi'} of Lemma \ref{lem:Bern}  $\phi'''$ is positive and non-increasing and therefore
	\[-\phi''(x)\geq \int_x^{2x}\phi'''(y)dy\geq x\phi'''(2x).\]
	Hence \eqref{def:condiB''} implies \eqref{def:condiB}. Let us show that \eqref{condi:DR} triggers \eqref{def:condiB''}. From \eqref{eq:phi''} and \eqref{condi:DR} we get at infinity
	$-\phi''(x)\asymp \Delta(x)$. Also
	from  \cite[eq. (3.3)]{DonRiv} we have that
	\begin{equation}\label{eq:12}
		\begin{split}
			&\Delta\lbrb{\frac{1}{x}}\leq 2\int_0^x w\bar{\mu}_\phi(w)dw=\Delta\lbrb{\frac{1}{x}}+x^2 \bar{\mu}_\phi\lbrb{x}
		\end{split}
	\end{equation}
	and we arrive from \eqref{eq:phi''} that for small $x$  we have
	\begin{equation}\label{eq:relation}
		\begin{split}
			&\int_0^x w\bar{\mu}_\phi(w)dw\asymp-\phi''(x^{-1})\asymp \Delta(x^{-1}).
		\end{split}
	\end{equation}
	However, from {the assumption} \textbf{[SC']} at \cite[Page $8$]{DonRiv} we have, for any $\lambda>1$, some $c\geq 1$ and $\alpha\in\lbrbb{0,2}$,
	\begin{equation*}
		\begin{split}
			&1\leq \liminfo{x}\frac{\int_0^{\lambda x} w\bar{\mu}_\phi(w)dw}{\int_0^x w\bar{\mu}_\phi(w)dw}\leq \limsupo{x}\frac{\int_0^{\lambda x} w\bar{\mu}_\phi(w)dw}{\int_0^x w\bar{\mu}_\phi(w)dw}\leq c\lambda^{2-\alpha}.
		\end{split}
	\end{equation*}
	Setting $\lambda=2$ and using \eqref{eq:relation} with $x\to 1/2y$ we deduce that at infinity
	\[-\phi''(y)\asymp -\phi''(2y),\]
	i.e. \eqref{def:condiB''} holds and hence \eqref{def:condiB} follows.
	From \cite[Definition, eq(2.0.7), p.65]{BGT87} and \eqref{eq:relation} we get that $-\phi''(x)$ is O-regularly varying at infinity, see \cite[Corollary 2.0.5]{BGT87}. Then from \cite[2.1.9 of Theorem 2.1.8]{BGT87}  the lower Matuszewska index  of $-\phi''(x)$ is larger or equal to $\alpha-2>-2$ and the upper index is not greater than $0$. Therefore, $-\phi''(x)$ is of bounded decrease as in \cite[p.71]{BGT87} and from \cite[Proposition 2.2.1]{BGT87} we get that $-\phi''(x)\geq x^{-2+\alpha-\epsilon}$ for all $0<\epsilon$  and $x$ large enough. Therefore, from \eqref{eq:relation} we obtain that
	\[\Delta\lbrb{x}\geq Cx^{-2+\alpha-\epsilon}\]
	for some $C>0$ and all $x$ large enough and \eqref{def:condiA} holds with $L=\infty$. Also the stronger \eqref{eq:dBound} holds.
\end{proof}

\section{Proofs of results in Subsection \ref{subsec:series}}
\label{subsec:power}
Here, we prove the main results contained in Subsection \ref{subsec:series}  based on Assumptions \eqref{eq:extensionA3} and \eqref{eq:uniformlimcond}.
\subsection{Proof of Theorem \ref{thm:smoothfmu}}
To prove Theorem \ref{thm:smoothfmu}, we will make use of the following integral representation, which is a consequence of the Laplace inversion formula of Proposition \ref{prop:LTthetapi}. Recall the definition of $\mathbb\D$ in \eqref{def:Reg}.
\begin{prop}\label{prop:intreptheta}
	Let $\Phi$ be the Laplace exponent of a potentially killed subordinator satisfying Assumptions \eqref{eq:extensionA3} and \eqref{eq:uniformlimcond} for some $\theta \in (0,\pi)$. Fix any $\varepsilon>0$ and let $\gamma_{\varepsilon,\theta}$ be the circle arc in $\C$ parametrized as $\gamma_{\varepsilon,\theta}:z=\varepsilon e^{i\xi}$ for $\xi \in \left[\frac{\theta}{2}-\pi,\pi-\frac{\theta}{2}\right]$.	
	%	define the set 
	%	\begin{equation*}
		% \gamma_{\varepsilon,\theta} := \ll z \in \mathbb{C}: z=\varepsilon e^{i\xi}, \xi \in \left[ \frac{\theta}{2}-\pi, \pi- \frac{\theta}{2} \right] \rr.
		% \end{equation*}	
	Then, on $\mathbb{D}$,
	\begin{equation}\label{statement1}
		\begin{split}
			f_\Phi (x,t) \, = \,& \frac{1}{\pi} \int_\varepsilon^{+\infty}\Im\left(\frac{\Phi^\dagger\lb \rho e^{i\left(\pi-\frac{\theta}{2}\right)}\rb}{\rho } e^{-x \Phi \lb \rho e^{i\left(\pi-\frac{\theta}{2}\right)}  \rb+ t \rho e^{i\left(\pi-\frac{\theta}{2}\right)} }\right) \, d\rho  +\frac{1}{2\pi i}\int_{\gamma_{\varepsilon,\theta}} \frac{\Phi^\dagger(z)}{z}e^{-x\Phi(z)+tz}dz.
		\end{split}
	\end{equation}
	%If, furthermore,
	%	\begin{equation}\label{eq:intcondtheta}
		%		\int_0^1\frac{\left|\Im \Phi^\dagger \left(\rho e^{i\left(\pi-\frac{\theta}{2}\right)}\right)\right|}{\rho}d\rho<+\infty
		%	\end{equation}
	%	then, on $\mathbb{D}$,
	%	\begin{equation}
		%		\int_0^{+\infty} \left|\Im \l \frac{\Phi^\dagger\left(\rho e^{i\left(\pi-\frac{\theta}{2}\right)}\right)}{\rho} e^{it\rho e^{i\l \pi -\frac{\theta}{2} \r}-x\Phi\l \rho e^{i \l \pi-\frac{\theta}{2} \r} \r} \r\right| \, d\rho <+\infty
		%	\end{equation}
	%	and 
	%	\begin{equation}\label{eq:intreptheta}
		%		f_\Phi(x,t) \, = \, 	\frac{1}{\pi}\int_0^{+\infty}\Im \l \frac{\Phi^\dagger\left(\rho e^{i\left(\pi-\frac{\theta}{2}\right)}\right)}{\rho} e^{t\rho e^{i \l \pi-\frac{\theta}{2} \r}-x\Phi\l \rho e^{i \l \pi-\frac{\theta}{2} \r} \r} \r \, d\rho.
		%	\end{equation}
\end{prop}
\begin{proof}
	Again, by \eqref{ceslimitdag} and \eqref{cesequalpointdag} for fixed $a>0$ and for almost any $(x,t) \in \mathbb{D}$
	\begin{equation*}
		f_\Phi(x,t) \, = \, \lim _{b \to +\infty} \frac{1}{2\pi i}\int_{a-ib}^{a+ib} e^{z t } 	\,  \frac{\Phi^\dagger(z)}{z} e^{-x\Phi(z)} \, dz,
	\end{equation*}
	provided that the limit on the right-hand side exists. Here, we compute the limit by using Cauchy's Theorem. 
	For $R \ge a$, let $b(R)=\sqrt{R^2-a^2}$. Set $A(R):= a-ib(R)$, $B(R):=a+ib(R)$ and observe that $|A(R)|=|B(R)|=R$. In particular, note that
	\begin{equation}\label{eq:420}
		\lim _{b \to +\infty} \int_{a-ib}^{a+ib} e^{z t } 	\,  \frac{\Phi^\dagger(z)}{z} e^{-x\Phi(z)} \, dz \, = \, \lim_{R \to +\infty} \int_{A(R)}^{B(R)} e^{z t } 	\,  \frac{\Phi^\dagger(z)}{z} e^{-x\Phi(z)} \, dz.
	\end{equation}
	Now define $C(R):=Re^{i\lb \pi - \frac{\theta}{2} \rb}$ and $F(R):=Re^{i\lb \pi + \frac{\theta}{2} \rb}$. Furthermore, let $\varepsilon >0$ and define $D(\varepsilon):=\varepsilon e^{i\lb \pi - \frac{\theta}{2} \rb}$ and $E(\varepsilon):=\varepsilon e^{i\lb \pi + \frac{\theta}{2} \rb}$. We let $\Gamma_R^+$ be the anticlockwise oriented circular arc joining $B(R)$ to $C(R)$, while we define $\Gamma_R^-$ the anticlockwise oriented circular arc joining $F(R)$ to $A(R)$. Denote also $\ell_1$ to be the oriented segment connecting $A(R)$ to $B(R)$, $\ell_2$ the oriented segment connecting $C(R)$ to $D(\varepsilon)$ and $\ell_3$ the oriented segment connecting $E(\varepsilon)$ to $F(R)$. Finally, let $-\gamma_{\varepsilon,\theta}$ be the clockwise oriented circular arc joining $D(\varepsilon)$ to $E(\varepsilon)$.
	Let $\partial \mathfrak{D}$ be the closed contour obtained by connecting, in this order, $\ell_1$, $\Gamma_R^+$, $\ell_2$, $-\gamma_{\varepsilon,\theta}$, $\ell_3$, $\Gamma_R^-$ (see Figure \ref{fig1}).	Such a contour is the boundary of an open set $\mathfrak{D}$ of the complex plane in which by assumption
		\begin{equation*}
			\mathfrak{D} \ni z  \mapsto F(z; x,t) \, : = \,  \frac{\Phi^\dagger(z)}{z}e^{z t-x\Phi(z)} \in \mathbb{C}, \qquad x>0, t>0,
		\end{equation*}
		is holomorphic and continuous at the boundary. Hence, we can apply Cauchy's Theorem to get
	% \mladen{the sign in front $+ \, \int_{\gamma_{\varepsilon,\theta}} F(z; x, t) \, dz $ has to be checked}
	\begin{equation*}
		\int_{\partial \mathfrak{D}} F(z; x, t) \, dz \, = \, 0.
	\end{equation*}
	This implies that
	\begin{align}\label{decomp}
		&\int_{A(R)}^{B(R)} e^{z t } 	\,  \frac{\Phi^\dagger(z)}{z} e^{-x\Phi(z)} \, dz \\ = \,& - \int_{\Gamma_R^+}F(z; x, t) \, dz \, - \, \int_{\ell_2} F(z; x, t) \, dz \, + \, \int_{\gamma_{\varepsilon,\theta}} F(z; x, t) \, dz \, - \, \int_{\ell_3} F(z; x, t) \, dz   - \int_{\Gamma_R^-} F(z; x, t) \, dz .	
	\end{align}
	\begin{minipage}{0.5\linewidth}
		\begin{tikzpicture}[scale=0.4]
			\draw [decoration={markings,mark=at position 1 with
				{\arrow[scale=3,>=stealth]{>}}},postaction={decorate}] (0,-8.5) -- (0,8.5);
			\draw [decoration={markings,mark=at position 1 with
				{\arrow[scale=3,>=stealth]{>}}},postaction={decorate}] (-7,0) -- (7,0);
			\draw[black, dashed] (-2,-2)--(0,0);
			\draw[black, dashed] (-2,2)--(0,0);
			\draw[black] (5,-7) -- (5,-4.88);
			\draw[black] (5,4.88) -- (5,7);
			\draw[black, dashed] (5,-7) -- (5,-8);
			\draw[black, dashed] (5,7) -- (5,8);
			\draw[black, dashed] (0,0)--(4.88,4.70);
			\begin{scope}[very thick,decoration={
					markings,
					mark=at position 0.5 with {\arrow{>}}}
				] 
				\centerarc[red,very thick,postaction={decorate}](0,0)(44:90:7);
				\centerarc[red,very thick,postaction={decorate}](0,0)(90:135:7);
				\draw[red,very thick, postaction={decorate}] (5,-4.88) -- (5,4.88);
				\draw[red,very thick,postaction={decorate}] (-4.88,4.88) -- (-2,2);
				\centerarc[red,very thick,postaction={decorate}](0,0)(135:-135:2.828);
				\draw[red,very thick,postaction={decorate}] (-2,-2)--(-4.88,-4.88);
				\centerarc[red,very thick,postaction={decorate}](0,0)(-135:-90:7);
				\centerarc[red,very thick,postaction={decorate}](0,0)(-90:-44:7);
			\end{scope}
			\fill[red] (5,4.88) circle (0.2);
			\fill[red] (-4.95,4.95) circle (0.2);
			\fill[red] (-2,2) circle (0.2);
			\fill[red] (-2,-2) circle (0.2);
			\fill[red] (-4.95,-4.95) circle (0.2);
			\fill[red] (5,-4.88) circle (0.2);
			\node at (6.3,5.40) {\large $B(R)$};
			\node at (6.3,-5.40) {\large $A(R)$};
			\node at (-6.3,-5.25) {\large $F(R)$};
			\node at (-6.3,5.25) {\large $C(R)$};
			\node at (-3,1.5) {\large $D(\varepsilon)$};
			\node at (-3,-1.5) {\large $E(\varepsilon)$};
			\node at (-0.85,-1.5) {\large $\varepsilon$};
			\centerarc[dashed](0,0)(135:225:0.5);
			\node at (-1,0.3) {\large $\theta$};
			\node at (2.8,3.5) {\large $R$};
			%		\node at (5.6,0.3) {\large $x_0$};
			%		\node at (4.5,2) {\large $r$};
			\fill[red] (0,7) circle (0.2);
			\fill[red] (0,-7) circle (0.2);
			%		\fill[black] (5,0) circle (0.1);
			\node at (2,7.5) {\large $M_+(R)$};
			\node at (2,-7.7) {\large $M_-(R)$};
			\node at (-1,-5.7) {\large $\Gamma_R^-$};
			\node at (-1,5.7) {\large $\Gamma_R^+$};
			\node at (-2.7,3.8)  {\large $\ell_2$};
			\node at (-2.7,-3.8)  {\large $\ell_3$};
			\node at (2,0.3)  {\large $\gamma_{\varepsilon,\theta}$};
			\node at (4.5,0.6)  {\large $\ell_1$};
		\end{tikzpicture}
		\captionof{figure}{\label{fig1}Sketch of the keyhole-type contour.}
	\end{minipage}
	\begin{minipage}{0.5\linewidth}
		\begin{tikzpicture}[scale=0.4]
			\draw [decoration={markings,mark=at position 1 with
				{\arrow[scale=3,>=stealth]{>}}},postaction={decorate}] (0,-8.5) -- (0,8.5);
			\draw [decoration={markings,mark=at position 1 with
				{\arrow[scale=3,>=stealth]{>}}},postaction={decorate}] (-7,0) -- (7,0);
			\draw[black, dashed] (2,-2)--(0,0);
			%\draw[black, dashed] (-2,2)--(0,0);
			\draw[black] (5,-7) -- (5,-4.88);
			\draw[black] (5,4.88) -- (5,7);
			\draw[black, dashed] (5,-7) -- (5,-8);
			\draw[black, dashed] (5,7) -- (5,8);
			\draw[black, dashed] (0,0)--(4.88,4.70);
			\begin{scope}[very thick,decoration={
					markings,
					mark=at position 0.5 with {\arrow{>}}}
				] 
				\centerarc[red,very thick,postaction={decorate}](0,0)(44:90:7);
				\centerarc[red,very thick,postaction={decorate}](0,0)(90:180:7);
				\centerarc[red,very thick,postaction={decorate}](0,0)(180:90:2.828);
				\centerarc[red,very thick,postaction={decorate}](0,0)(90:0:2.828);
				\centerarc[red,very thick,postaction={decorate}](0,0)(0:-90:2.828);
				\centerarc[red,very thick,postaction={decorate}](0,0)(-90:-180:2.828);
				\draw[red,very thick, postaction={decorate}] (5,-4.88) -- (5,4.88);
				\centerarc[red,very thick,postaction={decorate}](0,0)(-180:-90:7);
				\centerarc[red,very thick,postaction={decorate}](0,0)(-90:-44:7);
			\end{scope}
			\begin{scope}[very thick,decoration={
					markings,
					mark=at position 0.3 with {\arrow{>[left]}}}
				]
				\draw[red,very thick,postaction={decorate}] (-7,0) -- (-2.828,0);
				\draw[red,very thick,postaction={decorate}] (-2.828,0)--(-7,0);
			\end{scope}
			\fill[red] (5,4.88) circle (0.2);
			\fill[red] (-7,0) circle (0.2);
			\fill[red] (-2.828,0) circle (0.2);
			\fill[red] (5,-4.88) circle (0.2);
			\node at (6.3,5.40) {\large $B(R)$};
			\node at (6.3,-5.40) {\large $A(R)$};
			%\node at (-6.3,-5.25) {\large $F(R)$};
			\node at (-8.5,1.2) {\large $C(R)$};
			\node at (-8.5,-1.2) {\large $F(R)$};
			\node at (-1.5,0.8) {\large $D(\varepsilon)$};
			\node at (-1.5,-0.8) {\large $E(\varepsilon)$};
			%\node at (-3,-1.5) {\large $E(\varepsilon)$};
			\node at (0.85,-1.5) {\large $\varepsilon$};
			%\centerarc[dashed](0,0)(135:225:0.5);
			%\node at (-1,0.3) {\large $\theta$};
			\node at (2.8,3.5) {\large $R$};
			%		\node at (5.6,0.3) {\large $x_0$};
			%		\node at (4.5,2) {\large $r$};
			\fill[red] (0,7) circle (0.2);
			\fill[red] (0,-7) circle (0.2);
			%		\fill[black] (5,0) circle (0.1);
			\node at (2,7.5) {\large $M_+(R)$};
			\node at (2,-7.7) {\large $M_-(R)$};
			\node at (-1,-5.7) {\large $\Gamma_R^-$};
			\node at (-1,5.7) {\large $\Gamma_R^+$};
			\node at (-4,1)  {\large $\ell_2$};
			\node at (-6,-1)  {\large $\ell_3$};
			\node at (2,0.3)  {\large $\gamma_{\varepsilon}$};
			\node at (4.5,0.6)  {\large $\ell_1$};
		\end{tikzpicture}
		\captionof{figure}{\label{fig2}Sketch of the keyhole contour.}	
	\end{minipage}\\
	
	Now we deal with various terms separately.
	To deal with the first integral, let $M^+(R)= Ri$ and split the curve $\Gamma_R^+$ into $\Gamma_R^1$ connecting $B(R)$ to $M^+(R)$ and $\Gamma_R^2$ connecting $M^+(R)$ to $C(R)$ so that
	\begin{equation}	\label{29}
		\int_{\Gamma_R^+} F(z; x, t) \, dz  \, = \, \int_{\Gamma_R^1} F(z; x, t) \, dz  + \int_{\Gamma_R^2} F(z; x, t) \, dz.
	\end{equation}
	We begin with $\Gamma_R^1$. Note that, by the Estimation Lemma \cite[Theorem 5.24]{howie}
	\begin{equation}\label{estlemmagen}
		\left| \int_{\Gamma_R^1} F(z; x, t) \, dz \right| \, \leq \, \text{length}(\Gamma_R^1) \, \max_{z \in \Gamma_R^1} \left| F(z; x, t) \right|,
	\end{equation}
	where, with an abuse of notation, we denote by $\Gamma_R^1$ also the image of the parametrized oriented curve. To evaluate the maximum appearing in \eqref{estlemmagen} we parametrize $\Gamma_R^1$ as follows
	\begin{equation*}
		\Gamma_R^1 \, = \, \ll z \in \mathbb{C}: z = R e^{i\xi}, \xi \in \left[ \xi_{B}(R), \frac{\pi}{2} \right] \rr,
	\end{equation*}
	where $\xi_B(R)= \arctan \lb \frac{b(R)}{a} \rb$.
	Then, for $z = Re^{i\xi} \in \Gamma_R^1$, it holds
	\begin{equation*}
		\left| F(z; x,t) \right| \, =  \frac{\left| \Phi^\dagger(Re^{i\xi}) \right|}{R} e^{Rt\cos \xi -x\Re \Phi (Re^{i\xi})}.
	\end{equation*}
	Without loss of generality we can assume $\xi_B(R)> \frac{\pi}{4}$. First, observe that $\Re \lb Re^{i\xi}\rb= R \cos \xi \geq 0$, for any $\xi \in \left[ \xi_B(R), \frac{\pi}{2} \right]$, and thus, by Item \eqref{it:sign} of Lemma \ref{lem:Bern}, we have that $\Re \Phi \lb Re^{i\xi} \rb \geq 0$. Furthermore, it holds that $R \cos \xi \leq R \cos \xi_{B}(R) = a$, for any $\xi \in \left[ \xi_B(R), \frac{\pi}{2} \right]$. Hence, we get
	\begin{equation*}
		\max_{z \in \Gamma_R^1} \left| F(z; x,t) \right| \, \leq \, e^{ta} \max_{\xi \in \left[ \frac{\pi}{4}, \frac{\pi}{2} \right]} \left| \frac{\Phi^{\dagger} \lb Re^{i\xi} \rb}{Re^{i\xi}} \right|
	\end{equation*}
	and thus by Item \eqref{it:asymp} of Lemma \ref{lem:Bern} and the definition of $\phi^\dagger$, see \eqref{def:Pdag}, it holds
	\begin{equation}\label{maxtozerogen}
		\lim_{R \to + \infty} \max_{z \in \Gamma_R^1} \left| F(z; x,t) \right| = 0.
	\end{equation}
	%	Note now that
	%	\begin{equation*}
		%		\begin{split}
			%			\text{length}(\Gamma_R^1) \, = \,& R \l \frac{\pi}{2} - \xi_B(R) \r = R \l \frac{\pi}{2}-\arctan \frac{r(R)}{a} \r 
			%			= \,  \frac{R a}{r(R)} \frac{r(R)}{a} \l \frac{\pi}{2}-\arctan \frac{r(R)}{a} \r.
			%		\end{split}	
		%	\end{equation*}
	%	Since $r(R)= \sqrt{R^2-a^2}$ we obtain
	Furthermore, it is not difficult to check that
	\begin{equation}\label{lengthtox0gen}
		\lim_{R \to +\infty}\text{length}(\Gamma_R^1) \, = \, a.
	\end{equation}
	By combining \eqref{maxtozerogen} and \eqref{lengthtox0gen} with \eqref{estlemmagen} we have that
	\begin{equation}\label{firsttozerogen11}
		\lim_{R \to +\infty} \int_{\Gamma_R^1} F(z; x,t) \, dz \, = \, 0.
	\end{equation}	
	Now, we deal with the second term in \eqref{29}, i.e. the one on $\Gamma_R^2$.
	%\begin{equation*}
	%	\int_{\Gamma_R^2}F(z; x,t) \, dz \, = \, \int_{\Gamma_R^2} e^{z t} \frac{\Phi^\dagger(z)}{z} e^{-x\Phi(z)} dz.
	%\end{equation*}
	Recalling that $(x,t) \in \mathbb{D}$, let $\delta=t-\mathfrak{b}x>0$ and choose $p > 1$ such that $\frac{t}{p}-\mathfrak{b}x>\frac{\delta}{2}$, that exists since $\lim_{p \to 1}\frac{t}{p}-\mathfrak{b}x=\delta$. Let also $p^\prime >1$ be the conjugate exponent of $p>1$, i.e. $\frac{1}{p}+\frac{1}{p^\prime}=1$. Then we have
	%	Let $z = i\zeta$ and set $\widetilde{\Gamma}_R^2:= \ll z \in \mathbb{C}: z=Re^{i\zeta}, \zeta \in \left[ 0, \frac{\pi-\theta}{2} \right] \rr$ to get
	%	\begin{equation*}
		%		\int_{\Gamma_R^2}F(z; x,t) \, dz \, = \,i \int_{\widetilde{\Gamma}_R^2} e^{i\zeta t} \frac{\Phi^\dagger(i\zeta)}{i\zeta} e^{-x \Phi(i\zeta)} \, d\zeta.
		%	\end{equation*} 
	%	Let $\delta=t-\mathfrak{b}x>0$ and choose $p > 1$ such that $\frac{t}{p}-\mathfrak{b}x>\frac{\delta}{2}$. This is always possible since we are on $\mathbb{D}$ and since $\frac{t}{p}-\mathfrak{b}x < \delta$ with $\frac{t}{p}-\mathfrak{b}x \to \delta$ as $p\to 1$. Let $q_* \ge 1$ such that $\frac{1}{p}+\frac{1}{q_*}=1$ and $G(\zeta;x,t):= e^{\lambda t/p}(i\zeta)^{-1}\Phi(i\zeta) e^{-x\Phi(i\zeta)}$. We get
	\begin{equation}	\label{220}
		\begin{split}
			\left| \int_{\Gamma_R^2}F(z; x,t) \, dz \right| \,  \leq \, & R \lb \max_{z \in \Gamma_R^2}  \left| \frac{\Phi^\dagger(z)}{z} e^{ \frac{t}{p}z-x\Phi(z)} \right| \rb \lb \int_0^{\frac{\pi-\theta}{2}} e^{-(Rt\sin \xi)/p^\prime} d\xi \rb  \\
			\leq \, & \frac{p^\prime \pi}{t} \lb \max_{z \in \Gamma_R^2}  \left| \frac{\Phi^\dagger(z)}{z} e^{ \frac{t}{p}z-x\Phi(z)} \right| \rb = \, \frac{p^\prime \pi e^{-xq}}{t} \lb \max_{z \in \Gamma_R^2}  \left| \frac{\Phi^\dagger(z)}{z} e^{ \left(\frac{t}{p}-\mathfrak{b}x\right)z-x\Phi^\dagger(z)} \right| \rb, 
		\end{split}	
	\end{equation}
	%	\begin{equation}	\label{220}
		%	\begin{split}
			%		\left| \int_{\Gamma_R^2}F(z; x,t) \, dz \right| \, = \, &\left| \int_{\widetilde{\Gamma}_R^2} e^{i\zeta t} \frac{\Phi^\dagger(i\zeta)}{i\zeta} e^{-x \Phi(i\zeta)} \, d\zeta \right|  \\
			%		= \, & \left| \int_0^{(\pi-\theta)/2}  e^{iRe^{i\xi}t/q_*} G(Re^{i\xi};x,t) Re^{i\xi} d\xi \right|  \\
			%		\leq \, & R \int_0^{(\pi-\theta)/2}  e^{-(Rt \sin \xi )/q_*} \left| G\l Re^{i\xi};x,t \r \right| d\xi  \\
			%		\leq \, & R \l \max_{\xi \in \left[ 0, \frac{\pi-\theta}{2} \right]}  \left| G\l Re^{i\xi};x,t \r \right| \r \l \int_0^\pi e^{-(Rt\sin \xi)/q_*} d\xi \r  \\
			%		\leq \, & \frac{q_*\pi}{t} \l \max_{z \in \Gamma_R^2}  \left| \frac{\Phi^\dagger(z)}{z} e^{ \frac{t}{p}z-x\Phi(z)} \right| \r \\
			%		= \, & \frac{q_*\pi e^{-xq}}{t} \l \max_{z \in \Gamma_R^2}  \left| \frac{\Phi^\dagger(z)}{z} e^{ \left(\frac{t}{p}-\mathfrak{b}x\right)z-x\Phi^\dagger(z)} \right| \r, 
			%	\end{split}	
		%	\end{equation}
	where in the second inequality we have used Jordan's inequality \cite[eq (2), page 262]{brownchurchill}. Now, consider 
	\begin{equation}\label{eq:goftheproof}
	g(z)=e^{ \left(\frac{t}{p}-\mathfrak{b}x\right)z-x\Phi^\dagger(z)}	
	\end{equation}
and observe that, by hypothesis, $g$ is holomorphic on $\C\left(\frac{\pi}{2},\pi-\frac{\theta}{2}\right)$ and continuous on $\overline{\C\left(\frac{\pi}{2},\pi-\frac{\theta}{2}\right)}$. Furthermore, for $z=iR$ we have
	\begin{equation*}
		|g(iR)|=e^{-x\Re \Phi^\dagger(iR)} \le 1,
	\end{equation*}
	since $\Re \Phi^\dagger(iR) \ge 0$ by Item \eqref{it:sign} of Lemma \ref{lem:Bern}. For $z=Re^{i\left(\pi-\frac{\theta}{2}\right)}$ we have instead
	\begin{equation*}
		\left|g\left(Re^{i\left(\pi-\frac{\theta}{2}\right)}\right)\right|=e^{-R\left(\frac{t}{p}-\mathfrak{b}x\right)\cos\left(\frac{\theta}{2}\right)-x\Re \Phi^\dagger\left(Re^{i\left(\pi-\frac{\theta}{2}\right)}\right)}.
	\end{equation*}
	Now let us show that for $R$ big enough, $\left|g\left(Re^{i\left(\pi-\frac{\theta}{2}\right)}\right)\right| \le 1$. Indeed,
	\begin{align*}
		\left|g\left(Re^{i\left(\pi-\frac{\theta}{2}\right)}\right)\right|&=\exp\left(-R\cos\left(\frac{\theta}{2}\right)\left(\frac{t}{p}-\mathfrak{b}x+\frac{x}{\cos\left(\frac{\theta}{2}\right)}\frac{\Re \Phi^\dagger\left(Re^{i\left(\pi-\frac{\theta}{2}\right)}\right)}{R}\right)\right).
	\end{align*}
	Once we observe that 
	\begin{equation*}
		\left|\frac{\Re \Phi^\dagger\left(Re^{i\left(\pi-\frac{\theta}{2}\right)}\right)}{R}\right| \le \frac{\left| \Phi^\dagger\left(Re^{i\left(\pi-\frac{\theta}{2}\right)}\right)\right|}{R}
	\end{equation*}
	and we use \eqref{eq:uniformlimcond} to state that $\lim_{R \to +\infty}\frac{\left| \Phi^\dagger\left(Re^{i\left(\pi-\frac{\theta}{2}\right)}\right)\right|}{R}=0$, we know that there exists a constant $C(t,x,p)$ such that for $R>C(t,x,p)$
	\begin{equation*}
		\frac{t}{p}-\mathfrak{b}x+\frac{x}{\cos\left(\frac{\theta}{2}\right)}\frac{\Re \Phi^\dagger\left(Re^{i\left(\pi-\frac{\theta}{2}\right)}\right)}{R} \ge \frac{\delta}{4}.
	\end{equation*}
	Hence, for $R>C(t,x,p)$,
	\begin{align*}
		\left|g\left(Re^{i\left(\pi-\frac{\theta}{2}\right)}\right)\right|&\le \exp\left(-\frac{\delta R\cos\left(\frac{\theta}{2}\right)}{4}\right) \le 1.
	\end{align*}
	On the other hand, $R \in [0,+\infty) \mapsto \left|g\left(Re^{i\left(\pi-\frac{\theta}{2}\right)}\right)\right| \in \R$ is continuous and then the latter inequality implies that there exists $M=M(x,t,p)$ such that $|g(z)| \le M$ for any $z \in \partial \C\left(\frac{\pi}{2},\pi-\frac{\theta}{2}\right)$.
	Furthermore, observe that, for $z=Re^{i\xi}$ with $\xi \in \left(\frac{\pi}{2},\pi-\frac{\theta}{2}\right)$ we have, by \eqref{eq:uniformlimcond}, $|\Re\Phi(z)|\le |\Phi(z)| \le C|z|$, for $|z| \ge 1$. Hence, for $|z| \ge 1$, transferring through \eqref{def:Pdag} $\phi^\dagger$ to $\phi$ and using that $\Re(z)\leq 0$, we get that
	\begin{equation*}
		|g(z)|=e^{\frac{t}{p} \Re z+{x}q-x\Re(\Phi(z))} \le e^{xq+x|\Re(\Phi(z))|} \le e^{qx} e^{xC|z|}.
	\end{equation*}
	The continuity of the function $z \in \overline{\C\left(\frac{\pi}{2},\pi-\frac{\theta}{2}\right)} \mapsto |g(z)|e^{xC|z|} \in \R$ guarantees that there exists a constant $M_1=M_1(x,t,p,q)>0$ such that 
	\begin{equation*}
		|g(z)| \le M_1e^{xC|z|}, \ \forall z \in \overline{\C\left(\frac{\pi}{2},\pi-\frac{\theta}{2}\right)}.
	\end{equation*}
	Since $\pi-\frac{\theta}{2}-\frac{\pi}{2}<\frac{\pi}{2}$, we can use Phragmen-Lindel\"of Theorem (see \cite[Chapter $4$, Exercise $9$, Item (b)]{stein10}) to obtain $|g(z)| \le M$ for any $z \in \overline{\C\left(\frac{\pi}{2},\pi-\frac{\theta}{2}\right)}$.
	Thus, from \eqref{220}, we have
	\begin{equation*}
		\left| \int_{\Gamma_R^2}F(z; x,t) \, dz \right| \le \frac{M(x,t,p)p'\pi e^{-xq}}{t} \lb \max_{z \in \Gamma_R^2}  \left| \frac{\Phi^\dagger(z)}{z} \right| \rb.
	\end{equation*}
	Taking the limit as $R \to +\infty$ we finally have
	\begin{equation}\label{230}
		\lim_{R \to +\infty} \int_{\Gamma_R^2} F(z;x,t)\, dz \, = \, 0,
	\end{equation}
	that, combined with \eqref{firsttozerogen11}, leads to
	\begin{equation}	\label{circsopra}
		\lim_{R\to +\infty} \int_{\Gamma_R^+} F(z;x,t) \, dz \, = \,0.
	\end{equation}
	In the same spirit, it is possible to see that
	\begin{equation}\label{circsotto}
		\lim_{R \to +\infty} \int_{\Gamma_R^-}F(z;x,t) \, dz \, = \, 0.	
	\end{equation}
	We consider now the integral on $\ell_2$. We have that
	\begin{equation}\label{239}
		\int_{\ell_2} F (z;x,t) dz \,= \, -\int_{\varepsilon}^R \frac{\Phi^\dagger\lb\rho e^{i\lb \pi-\frac{\theta}{2} \rb}\rb}{\rho } e^{-x \Phi \lb \rho e^{i\lb \pi-\frac{\theta}{2} \rb} \rb+t  \rho e^{i \lb \pi-\frac{\theta}{2} \rb}} \, d\rho.
	\end{equation}
	Choose $p,p^\prime > 1$ as above so that, recalling that $\frac{\left|\Phi^\dagger\left(\rho e^{i\left(\pi-\frac{\theta}{2}\right)}\right)\right|}{\rho} \le C$ for $\rho \ge \varepsilon$ by \eqref{eq:uniformlimcond}, we have
	\begin{align}
		&\int_\varepsilon^\infty \left|  \frac{\Phi^\dagger\lb\rho e^{i\lb \pi-\frac{\theta}{2} \rb}\rb}{\rho } e^{-x \Phi \lb \rho e^{i\lb \pi-\frac{\theta}{2} \rb} \rb+t  \rho e^{i \lb \pi-\frac{\theta}{2} \rb}} \right| d\rho \notag \\	= \, & e^{-qx}\int_\varepsilon^\infty \frac{|\Phi^\dagger\lb \rho e^{i \lb \pi-\frac{\theta}{2} \rb} \rb|}{\rho}  \exp \ll {-\rho \cos\left(\frac{\theta}{2}\right) \lb \frac{t}{p}-\mathfrak{b}x +x \frac{\Re \Phi^\dagger \lb \rho e^{i \lb \pi-\frac{\theta}{2} \rb}  \rb}{\rho \cos\left(\frac{\theta}{2}\right)}  \rb } \rr e^{-\rho  \frac{t}{p^\prime} \cos \frac{\theta}{2} }d\rho \notag \\
		\le & CM\int_\varepsilon^\infty e^{-\frac{\rho t}{p^\prime}\cos\left(\frac{\theta}{2}\right)}d\rho <+\infty, \label{richiamatadopo}
	\end{align}
	where we have also used that $\frac{t}{p}-\mathfrak{b}x>\frac{\delta}2>0$.
	Hence, we can take the limit in \eqref{239} to get
	\begin{equation}	\label{243}
		\lim_{R \to +\infty}\int_{\ell_2} F(z; x,t) \, dz \, = \, -\int_\varepsilon^{+\infty}\frac{\Phi^\dagger\lb\rho e^{i\lb \pi-\frac{\theta}{2} \rb}\rb}{\rho } e^{-x \Phi \lb \rho e^{i\lb \pi-\frac{\theta}{2} \rb} \rb+t  \rho e^{i \lb \pi-\frac{\theta}{2} \rb}} \, d\rho \notag \\ =: \, -I_1(\varepsilon).
	\end{equation}
	Analogously, on $\ell_3$, we have that
	\begin{equation}\label{244}
		\lim_{R \to +\infty}\int_{\ell_3} F(z; x,t) \, dz \, = \, \int_\varepsilon^{+\infty}\frac{\Phi^\dagger\lb\rho e^{i\lb \pi+\frac{\theta}{2} \rb}\rb}{\rho } e^{-x \Phi \lb \rho e^{i\lb \pi+\frac{\theta}{2} \rb} \rb+t  \rho e^{i \lb \pi+\frac{\theta}{2} \rb}} \, d\rho\, =: \, I_2(\varepsilon).
	\end{equation}
	Furthermore, by using the fact that $\overline{\Phi(z)} = \Phi(\overline{z})$ by Schwartz reflection principle (see \cite[Theorem $5.6$]{stein10}), we know that $I_2(\varepsilon) = \overline{I_1(\varepsilon)}$. Hence, taking the limit as $R \to +\infty$ in \eqref{decomp} and using \eqref{circsopra}, \eqref{circsotto}, \eqref{243} and \eqref{244}, we get
	\begin{equation*}
		\begin{split}
			\lim_{b \to +\infty}\int_{a-ib}^{a+ib} e^{z t } 	\,  \frac{\Phi^{\dagger}(z)}{z} e^{-x\Phi(z)} \, dz \, & = \,  I_1(\varepsilon)-\overline{I_1(\varepsilon)} + \, \int_{\gamma_{\varepsilon,\theta}} F(z; x, t) \, dz   =2iI_3(\varepsilon)+\, \int_{\gamma_{\varepsilon,\theta}} F(z; x, t) \, dz,
		\end{split}
	\end{equation*}
	where we denote
	\begin{align}
		I_3(\varepsilon):=\int_\varepsilon^{+\infty}\Im\left(\frac{\Phi^\dagger\lb \rho e^{i\left(\pi-\frac{\theta}{2}\right)}\rb}{\rho } e^{-x \Phi \lb \rho e^{i\left(\pi-\frac{\theta}{2}\right)}  \rb+ t \rho e^{i\left(\pi-\frac{\theta}{2}\right)} }\right) \, d\rho.
		\label{beforeepsilon}
	\end{align}
	This proves \eqref{statement1}.
\end{proof}
\begin{rmk}
	Under the hypotheses of Proposition \ref{prop:intreptheta}, if furthermore \begin{equation}\label{eq:intcondtheta}
		\int_0^1\frac{\left|\Im \Phi^\dagger \left(\rho e^{i\left(\pi-\frac{\theta}{2}\right)}\right)\right|}{\rho}d\rho<+\infty
	\end{equation}
	then we can send $\varepsilon \to 0$ in \eqref{statement1}, wherein by the estimation lemma the second term converges to zero, hence getting
	\begin{equation*}
		f_\Phi (x,t) \, = \, \frac{1}{\pi} \int_0^{+\infty}\Im\left(\frac{\Phi^\dagger\lb \rho e^{i\left(\pi-\frac{\theta}{2}\right)}\rb}{\rho } e^{-x \Phi \lb \rho e^{i\left(\pi-\frac{\theta}{2}\right)}  \rb+ t \rho e^{i\left(\pi-\frac{\theta}{2}\right)} }\right) \, d\rho.
	\end{equation*}
\end{rmk}
A similar integral representation holds also for $\bar{\mu}_\phi^{\ast n}$. The proof is similar to the one of Proposition \ref{prop:intreptheta}, where we substitute the term $\frac{\phi^\dagger(z)}{z}e^{-x\phi(z)}$ with $\left(\frac{\phi^\dagger(z)}{z}\right)^n$. For such a reason, we only underline the parts of the proof that are actually different.
\begin{prop}
	\label{lem:convtail}
	With the same notation of Proposition \ref{prop:intreptheta}, under \eqref{eq:extensionA3} and \eqref{eq:uniformlimcond}, for $n \ge \red{1}$, it holds that, for any $t>0$,
	\begin{equation}\label{convcode}
		\bar{\mu}_\Phi^{\ast n} (t) = \frac{1}{\pi} \int_{\varepsilon}^{+\infty} \Im \left[ \lb \frac{\Phi^\dagger \lb \rho e^{i \lb \pi-\frac{\theta}{2} \rb} \rb}{\rho e^{i \lb \pi-\frac{\theta}{2} \rb}} \rb^n e^{i \lb \pi-\frac{\theta}{2} \rb}e^{t\rho e^{i \lb \pi-\frac{\theta}{2} \rb}} \right] d\rho +\frac{1}{2\pi i} \int_{\gamma_{\varepsilon,\theta}} e^{tz} \lb \frac{\phi^\dagger (z)}{z} \rb^n dz.
	\end{equation}
	%	 Furthermore, it is true that the function $(0, +\infty) \mapsto \bar{\nu}_\phi^{\star n} (t)$, has, for any $n=1,2,\cdots$, derivatives of all order $r \geq 0$, such that
	%	\begin{equation}
		%		\frac{d^r}{dt^r} \bar{\mu}_\Phi^{n\star} (t) 
		%		= \frac{1}{\pi} \int_{\varepsilon}^{+\infty} \Im \left[  \frac{\l\Phi^\dagger \l \rho e^{i \l \pi-\frac{\theta}{2} \r} \r\r^n}{\rho^{n-r}e^{i(n-r-1)\l \pi-\frac{\theta}{2} \r}}  e^{t\rho e^{i \l \pi-\frac{\theta}{2} \r}} \right] d\rho -\frac{1}{2\pi i} \int_{\gamma_{\varepsilon,\theta}} e^{tz}  \frac{\l\phi^\dagger (z)\r^n}{z^{n-r}}  dz.
		%		\label{derivatecode}
		%	\end{equation}
\end{prop}
\begin{proof}
	Setting $F_n(z;t)=\left(\frac{\phi^\dagger(z)}{z}\right)^ne^{tz}$ the proof follows as the one in Proposition \ref{prop:intreptheta}. The main differences concern integrals over $\Gamma^2_R$ and $\ell_2$. First, by Jordan's inequality \cite[eq. (2), page 262]{brownchurchill},
	\begin{equation*}
		\left|\int_{\Gamma^2_R}F_n(z;t)dz\right| \le \frac{\pi}{t}\left(\max_{z \in \Gamma_R^2}\left|\frac{\Phi^\dagger(z)}{z}\right|^n\right) \to 0,\,\, \mbox{as $R\to\infty$},
	\end{equation*}
	where the limit holds by assumption \eqref{eq:uniformlimcond}. Concerning the integral over $\ell_2$, observe that
	\begin{equation*}
		\int_{\varepsilon}^{\infty}|F_n\left(\rho e^{i\left(\pi-\frac{\theta}{2};t\right)};t\right)|dz \le \int_{\varepsilon}^{\infty}e^{-\rho t \cos\left(\frac{\theta}{2}\right)}\left|\frac{\Phi^\dagger\left(\rho e^{i\left(\pi-\frac{\theta}{2}\right)}\right)}{\rho}\right|^nd\rho \le C\int_{\varepsilon}^{\infty}e^{-\rho t \cos\left(\frac{\theta}{2}\right)} d\rho<+\infty,
	\end{equation*}
	where we have used the fact that $\left|\frac{\Phi^\dagger\left(\rho e^{i\left(\pi-\frac{\theta}{2}\right)}\right)}{\rho}\right|^n$ is bounded for $\rho \ge \varepsilon$ by assumption \eqref{eq:uniformlimcond}.
\end{proof}
Now we are ready to prove Theorem \ref{thm:smoothfmu} by using the previously obtained integral representations.
\begin{proof}[Proof of Theorem \ref{thm:smoothfmu}]
	Let us first prove that $\bar{\mu}_\phi^{\ast n} \in C^\infty(0,+\infty)$. To do this, fix $\varepsilon>0$, let $l \ge 1$ and $[t_1,t_2]\subset (0,+\infty)$. Let $F_n(z;t)=\left(\frac{\phi^\dagger(z)}{z}\right)^ne^{tz}$. Then $\frac{\partial^l}{\partial t^l}F_n(z,t)=z^l\left(\frac{\phi^\dagger(z)}{z}\right)^ne^{tz}$ is continuous for $(z,t) \in \gamma_{\varepsilon,\theta} \times [t_1,t_2]$, where, with an abuse of notation, $\gamma_{\varepsilon,\theta}$ is the image of the parametrized curve defined in Proposition \ref{prop:intreptheta}. Then we have
	\begin{equation}\label{eq:upboundovereps}
		\left|\frac{\partial^l}{\partial t^l}F_n(z,t)\right| \le \max_{(z,t) \in \gamma_{\varepsilon,\theta} \times [t_1,t_2]}\left|z^l\left(\frac{\phi^\dagger(z)}{z}\right)^ne^{tz}\right|,
	\end{equation}
	where the right-hand side is constant, hence integrable over $\gamma_{\varepsilon,\theta}$. Next, let
	\begin{equation*}
		G_n(\rho,t)=\Im \left[ \lb \frac{\Phi^\dagger \lb \rho e^{i \lb \pi-\frac{\theta}{2} \rb} \rb}{\rho e^{i \lb \pi-\frac{\theta}{2} \rb}} \rb^n e^{i \lb \pi-\frac{\theta}{2} \rb}e^{t\rho e^{i \lb \pi-\frac{\theta}{2} \rb}} \right]
	\end{equation*}
	and observe that
	\begin{equation*}
		\frac{\partial^l}{\partial t^l}G_n(\rho,t)=\Im \left[ \lb \frac{\Phi^\dagger \lb \rho e^{i \lb \pi-\frac{\theta}{2} \rb} \rb}{\rho e^{i \lb \pi-\frac{\theta}{2} \rb}} \rb^n \rho^l e^{i(l+1) \lb \pi-\frac{\theta}{2} \rb}e^{t\rho e^{i \lb \pi-\frac{\theta}{2} \rb}} \right].
	\end{equation*}
	For $\rho \ge \varepsilon$ and $t \in [t_1,t_2]$, we have
	\begin{equation}\label{eq:contrder2}
		\left|\frac{\partial^l}{\partial t^l}G_n(\rho,t)\right| \le C e^{-t_1 \rho \cos \frac{\theta}{2}} \rho^l,
	\end{equation}
	since $\left|\frac{\Phi^\dagger\left(\rho e^{i\left(\pi-\frac{\theta}{2}\right)}\right)}{\rho}\right|^n$ is bounded for $\rho \ge \varepsilon$ by assumption \eqref{eq:uniformlimcond}. Observe that the right-hand side of \eqref{eq:contrder2} is integrable over $(\varepsilon,+\infty)$. Hence, by \eqref{eq:upboundovereps} and \eqref{eq:contrder2} and the fact that $l \ge 1$ is arbitrary, we can differentiate $l$ times inside the integrals in \eqref{convcode}, getting	
	\begin{equation}
		\frac{d^r}{dt^r} \bar{\mu}_\Phi^{\ast n} (t) 
		= \frac{1}{\pi} \int_{\varepsilon}^{+\infty} \Im \left[  \frac{\lb\Phi^\dagger \lb \rho e^{i \lb \pi-\frac{\theta}{2} \rb} \rb\rb^n}{\rho^{n-r}e^{i(n-r-1)\lb \pi-\frac{\theta}{2} \rb}}  e^{t\rho e^{i \lb \pi-\frac{\theta}{2} \rb}} \right] d\rho +\frac{1}{2\pi i} \int_{\gamma_{\varepsilon,\theta}} e^{tz}  \frac{\lb\phi^\dagger (z)\rb^n}{z^{n-r}}  dz.
		\label{derivatecode}
	\end{equation}
	This proves that $\bar{\mu}_\phi^{\ast n} \in C^\infty(0,+\infty)$.
	
	Now let us prove that $f_\Phi \in C^\infty(\mathbb{D})$. To do this, fix any $k,l \ge 0$, $0<t_1<t_2$ and $0<x_1<x_2<t_1/\mathfrak{b}$, recalling that any $(x,t) \in \mathbb{D}$ admits a compact neighbourhood of the form $[x_1,x_2]\times [t_1,t_2]$ specified before. Let $F(z;x,t)=\frac{\phi^\dagger(z)}{z}e^{-x\phi(z)+tz}$ and observe that
	\begin{equation*}
		\frac{\partial^k}{\partial x^k}\frac{\partial^l}{\partial t^l}F(z;x,t)=(-1)^k\frac{z^l\phi^\dagger(z)(\phi(z))^k}{z}e^{-x\phi(z)+tz}.
	\end{equation*} 
	The latter is continuous over $\gamma_{\varepsilon,\theta} \times [x_1,x_2]\times [t_1,t_2]$ and then
	\begin{equation}\label{eq:smoothf1}
		\left|\frac{\partial^k}{\partial x^k}\frac{\partial^l}{\partial t^l}F(z;x,t)\right|\le \max_{(z,x,t) \in \gamma_{\varepsilon,\theta} \times [x_1,x_2]\times [t_1,t_2]}\left|\frac{z^l\phi^\dagger(z)(\phi(z))^k}{z}e^{-x\phi(z)+tz}\right|,
	\end{equation}
	where the right-hand side is a constant and then it is integrable over $\gamma_{\varepsilon,\theta}$. Now set
	\begin{equation*}
		G(\rho;x,t)=\Im\left(\frac{\Phi^\dagger\lb \rho e^{i\left(\pi-\frac{\theta}{2}\right)}\rb}{\rho } e^{-x \Phi \lb \rho e^{i\left(\pi-\frac{\theta}{2}\right)}  \rb+ t \rho e^{i\left(\pi-\frac{\theta}{2}\right)} }\right)
	\end{equation*}
	and observe that
	\begin{equation*}
		\frac{\partial^k}{\partial x^k}\frac{\partial^l}{\partial t^l}G(\rho;x,t)=(-1)^k\Im\left(\frac{\Phi^\dagger\lb \rho e^{i\left(\pi-\frac{\theta}{2}\right)}\rb\left(\Phi\left(\rho e^{i\left(\pi-\frac{\theta}{2}\right)}\right)\right)^k}{\rho^{k+1}}\rho^{l+k} e^{il\left(\pi-\frac{\theta}{2}\right)}e^{-x \Phi \lb \rho e^{i\left(\pi-\frac{\theta}{2}\right)}  \rb+ t \rho e^{i\left(\pi-\frac{\theta}{2}\right)} }\right).
	\end{equation*}
	Recall that $\left|\frac{\Phi^\dagger\lb \rho e^{i\left(\pi-\frac{\theta}{2}\right)}\rb\left(\Phi\left(\rho e^{i\left(\pi-\frac{\theta}{2}\right)}\right)\right)^k}{\rho^{k+1}}\right|$ is bounded as $\rho \ge \varepsilon$, and set $p,p^\prime >1$ as right after \eqref{firsttozerogen11}. Arguing as in the proof of Proposition \ref{prop:intreptheta}, we know that there exists $M>0$ such that $|g(z)| \le M$ for all $z \in \overline{\mathbb{C}\left(\frac{\pi}{2},\pi-\frac{\theta}{2}\right)}$, where $g$ is defined in \eqref{eq:goftheproof} with $t_2$ and $x_1$ in place of $t$ an $x$. Hence, we have
	\begin{equation}\label{smoothfphi2}
		\left|\frac{\partial^k}{\partial x^k}\frac{\partial^l}{\partial t^l}G(\rho;x,t)\right|\le CM e^{-\frac{\rho t_1}{p^\prime}\cos \frac{\theta}{2}},
	\end{equation}
	where the right-hand side is integrable on $(\varepsilon,+\infty)$. Hence, since $k,l \ge 0$ are arbitrary, we can differentiate $k$ times in $x$ and $l$ times with respect to $t$ in \eqref{statement1}, to get that $f_\Phi \in C^\infty(\mathbb{D})$ and
	\begin{align}\label{derivatefphi}
		\begin{split}
			\frac{\partial^k}{\partial x^k}\frac{\partial^l}{\partial t^l}f_\phi(x,t)&= \frac{(-1)^k}{\pi} \int_\varepsilon^{+\infty}\Im\left(\frac{\Phi^\dagger\lb \rho e^{i\left(\pi-\frac{\theta}{2}\right)}\rb\lb\Phi\lb \rho e^{i\left(\pi-\frac{\theta}{2}\right)}\rb\rb^{k}}{\rho^{1-l} e^{-i l \left(\pi-\frac{\theta}{2}\right)}}e^{-x \Phi \lb \rho e^{i\left(\pi-\frac{\theta}{2}\right)}  \rb+ t \rho e^{i\left(\pi-\frac{\theta}{2}\right)} }\right) \, d\rho \\
			&+\frac{(-1)^k}{2\pi i}\int_{\gamma_{\varepsilon,\theta}} \frac{\Phi^\dagger(z)(\Phi(z))^{k}}{z^{1-l}}e^{-x\Phi(z)+tz}dz.
		\end{split} 
	\end{align}
\end{proof}

\subsection{Proof of Theorem \ref{thm:seriespi}}
In order to prove Theorem \ref{thm:seriespi} we employ the integral representations given in Propositions \ref{prop:intreptheta} and \ref{lem:convtail}.
\begin{proof}[Proof of Theorem \ref{thm:seriespi}]
	As in the \eqref{derivatefphi}, we have
	\begin{align}	
		\begin{split}
			\frac{\partial^k}{\partial x^k}\frac{\partial^l}{\partial t^l}f_\phi(x,t)&= \frac{(-1)^k}{\pi} \int_\varepsilon^{+\infty}\Im\left(\frac{\Phi^\dagger\lb \rho e^{i\left(\pi-\frac{\theta}{2}\right)}\rb\lb\Phi\lb \rho e^{i\left(\pi-\frac{\theta}{2}\right)}\rb\rb^{k}}{\rho^{1-l} e^{-i l \left(\pi-\frac{\theta}{2}\right)}}e^{-x \Phi \lb \rho e^{i\left(\pi-\frac{\theta}{2}\right)}  \rb+ t \rho e^{i\left(\pi-\frac{\theta}{2}\right)} }\right) \, d\rho \\
			&+\frac{(-1)^k}{2\pi i}\int_{\gamma_{\varepsilon,\theta}} \frac{\Phi^\dagger(z)(\Phi(z))^{k}}{z^{1-l}}e^{-x\Phi(z)+tz}dz.
		\end{split}
	\end{align}
	Writing $e^{-x\Phi(z)}$ as a power series and assuming we can exchange the series with the integral, we have
	\begin{align*}	
		&\frac{\partial^k}{\partial x^k}\frac{\partial^l}{\partial t^l}f_\phi(x,t)= \sum_{j=0}^{+\infty}(-1)^{k+j}\frac{x^j}{j!}\\
		&\times \left[\frac{1}{\pi} \int_\varepsilon^{+\infty}\Im\left(\frac{\Phi^\dagger\lb \rho e^{i\left(\pi-\frac{\theta}{2}\right)}\rb\lb\Phi\lb \rho e^{i\left(\pi-\frac{\theta}{2}\right)}\rb\rb^{k+j}}{\rho^{1-l} e^{-i l \left(\pi-\frac{\theta}{2}\right)}}e^{t \rho e^{i\left(\pi-\frac{\theta}{2}\right)} }\right) \, d\rho+\frac{1}{2\pi i}\int_{\gamma_{\varepsilon,\theta}} \frac{\Phi^\dagger(z)(\Phi(z))^{k+j}}{z^{1-l}}e^{tz}dz\right]\\
		&=\sum_{j=0}^{+\infty}\sum_{k_1+k_2+k_3=k+j} \frac{(k+j)!}{k_1!k_2!k_3!}(-1)^{k+j}\frac{x^j}{j!}q^{k_1}\mathfrak{b}^{k_2}\\
		&\times \left[\frac{1}{\pi} \int_\varepsilon^{+\infty}\Im\left(\frac{\left(\Phi^\dagger\lb \rho e^{i\left(\pi-\frac{\theta}{2}\right)}\rb\right)^{k_3+1}}{\rho^{1-l-k_2} e^{-i (l+k_2) \left(\pi-\frac{\theta}{2}\right)}}e^{t \rho e^{i\left(\pi-\frac{\theta}{2}\right)} }\right) \, d\rho+\frac{1}{2\pi i}\int_{\gamma_{\varepsilon,\theta}} \frac{(\Phi^\dagger(z))^{k_3+1}}{z^{1-l-k_2}}e^{tz}dz\right]\\
		&=\sum_{j=0}^{+\infty}\sum_{k_1+k_2+k_3=k+j} \frac{(k+j)!}{k_1!k_2!k_3!}(-1)^{k+j}\frac{x^j}{j!}q^{k_1}\mathfrak{b}^{k_2}\\
		&\times \left[\frac{1}{\pi} \int_\varepsilon^{+\infty}\Im\left(\frac{\left(\Phi^\dagger\lb \rho e^{i\left(\pi-\frac{\theta}{2}\right)}\rb\right)^{k_3+1}}{\rho^{k_3+1-(k_2+k_3+l)} e^{i (k_3-(l+k_2+k_3)) \left(\pi-\frac{\theta}{2}\right)}}e^{t \rho e^{i\left(\pi-\frac{\theta}{2}\right)} }\right) \, d\rho+\frac{1}{2\pi i}\int_{\gamma_{\varepsilon,\theta}} \frac{(\Phi^\dagger(z))^{k_3+1}}{z^{k_3+1-(k_2+k_3+l)}}e^{tz}dz\right]\\
		&=\sum_{j=0}^{+\infty}\sum_{k_1+k_2+k_3=k+j} \frac{(k+j)!}{k_1!k_2!k_3!}(-1)^{k+j}\frac{x^j}{j!}q^{k_1}\mathfrak{b}^{k_2}\frac{d^{k_2+k_3+l}}{d t^{k_2+k_3+l}}\mu^{\ast (k_3+1)}(t),
	\end{align*} 
	where we used \eqref{derivatecode} in the last equality. Now we only have to prove that we can exchange the series with the integrals. This is clear for the integral over $\gamma_{\varepsilon, \theta}$, thus let us only consider the one over $(\varepsilon,+\infty)$.
	Indeed, we have
	\begin{align}
		& \int_\varepsilon^{+\infty} \sum_{j=0}^{+\infty}\left|   (-1)^{k+j} \frac{ \Phi^\dagger\lb \rho e^{i\lb \pi-\frac{\theta}{2} \rb} \rb \lb \Phi\lb \rho e^{i\lb \pi-\frac{\theta}{2} \rb} \rb \rb^{{k+j}}}{j! \; \rho^{1-l}e^{-il\left(\pi-\frac{\theta}{2}\right)}} \; x^j \; e^{t  \rho e^{i \lb \pi-\frac{\theta}{2} \rb}} \right| d\rho  \notag \\
		= \,& \int_\varepsilon^\infty \frac{\left| \Phi^\dagger\lb \rho e^{i\lb \pi-\frac{\theta}{2} \rb} \rb \right|\left|\Phi\left(\rho e^{i\lb \pi-\frac{\theta}{2} \rb}\right)\right|^k}{\rho^{k+1}}\rho^{l+k} e^{x \left| \Phi \lb \rho e^{i \lb \pi-\frac{\theta}{2} \rb} \rb \right|-t\rho \cos \frac{\theta}{2}} d\rho \notag \\
		\leq \, & e^{xq}\int_\varepsilon^{+\infty} \frac{\left| \Phi^\dagger\lb \rho e^{i\lb \pi-\frac{\theta}{2} \rb} \rb \right|\left|\Phi\left(\rho e^{i\lb \pi-\frac{\theta}{2} \rb}\right)\right|^k}{\rho^{k+1}}\rho^{l+k} e^{x\mathfrak{b}\rho\cos\left(\frac{\theta}{2}\right)-t\rho \cos \frac{\theta}{2}+\Phi^\dagger\lb \rho e^{i \lb \pi-\frac{\theta}{2} \rb} \rb} d\rho.	\label{248}
	\end{align}
	It is clear that we have to check integrability in the right-hand side of \eqref{248} only in a neighbourhood of infinity. To do this, set $\delta=t-\mathfrak{b}x$ and $p,p^\prime>1$ as in the proof of Proposition \ref{prop:intreptheta}. By \eqref{eq:uniformlimcond} we know that there exists $K$ big enough such that $\frac{\left|\phi^\dagger\lb \rho e^{i \lb \pi-\frac{\theta}{2} \rb} \rb \right|\left|\phi\lb \rho e^{i \lb \pi-\frac{\theta}{2} \rb} \rb \right|^k}{\rho^{k+1}}$ is bounded and $\frac{\left|\phi^\dagger\lb \rho e^{i \lb \pi-\frac{\theta}{2} \rb} \rb \right|}{\rho \cos(\theta/2)}<\frac{\delta}{4}$, whenever $\rho >K$. Hence we get
	\begin{align}
		\begin{split}
			&\int_{K}^{+\infty}\frac{\left|\phi^\dagger\lb \rho e^{i \lb \pi-\frac{\theta}{2} \rb} \rb \right|\left|\phi\lb \rho e^{i \lb \pi-\frac{\theta}{2} \rb} \rb \right|^k}{\rho^{k+1}} e^{-\frac{t}{p^\prime}\rho \cos\left(\frac{\theta}{2}\right)}\exp\left(-\rho \cos\left(\frac{\theta}{2}\right)\left(\frac{t}{p}-x\mathfrak{b}-\frac{\left|\Phi^\dagger\lb \rho e^{i \lb \pi-\frac{\theta}{2} \rb} \rb \right|}{\rho \cos\left(\frac{\theta}{2}\right)}\right)\right) d\rho \notag \\
			\leq \,& C \int_K^{+\infty} \rho^{k+l}e^{-\frac{t}{p^\prime}\rho \cos \lb \frac{\theta}{2} \rb}\red{d\rho} < +\infty.
		\end{split}
	\end{align}
	This concludes the proof.
	\end{proof}

\subsection{Proof of Theorem \ref{behavatzero} }
The behaviour at zero provided in Theorem \ref{behavatzero} can be shown as a direct consequence of the series representation given in Theorem \ref{thm:seriespi}.
\begin{proof}[Proof of Theorem \ref{behavatzero}]
	Let $[t_1,t_2] \subset (0,+\infty)$ and observe that for $t \in [t_1,t_2]$ we have by \eqref{derivatecode}
	\begin{equation}\label{eq:estIjkl}
		\begin{split}
			|\mathcal{I}_{j,k,l}(t)| &\le \frac{1}{\pi}\int_\varepsilon^{+\infty}\frac{\left|\Phi^{\dagger}\left(\rho e^{i\left(\pi-\frac{\theta}{2}\right)}\right)\right|\left|\Phi\left(\rho e^{i\left(\pi-\frac{\theta}{2}\right)}\right)\right|^{j+k}}{\rho^{1-l}}e^{-t_1\rho \cos\left(\frac{\theta}{2}\right)}d\rho  \\
			&+ \frac{\varepsilon^{l}}{2\pi} \int_{\frac{\theta}{2}-\pi}^{\pi-\frac{\theta}{2}}|\Phi^\dagger (\varepsilon e^{i\varphi})||\phi(\varepsilon e^{i\varphi})|^{j+k} e^{\varepsilon t_2 |\cos(\varphi)|}d\varphi.
		\end{split}
	\end{equation}
	Hence, for $0<x<\frac{t_1}{\mathfrak{b}}$ and $t \in [t_1,t_2]$, recalling the definition of $\mathcal{P}_{n,k,l}(x,t)$ given in \eqref{eq:zero2}, we have
%$	 \in [t_1,t_2]$ we have 
	\begin{align}
		\left|\frac{\partial^k \partial^l}{\partial x^k \partial t^l}f_\Phi(x,t)-\mathcal{P}_{n,k,l}(x,t)\right|  \le & \sum_{j=n+1}^{+\infty}\frac{x^j}{j!}|\mathcal{I}_{j,k,l}(t)| \notag \\
		%\leq \, & \sum_{j=n+1}^{+\infty} \frac{x^j}{j!} \left[ \frac{1}{\pi}\int_\varepsilon^{+\infty}\frac{\left|\Phi^{\dagger}\left(\rho e^{i\left(\pi-\frac{\theta}{2}\right)}\right)\right|\left|\Phi\left(\rho e^{i\left(\pi-\frac{\theta}{2}\right)}\right)\right|^{j+k}}{\rho^{1-l}}e^{-t_1\rho \cos\left(\frac{\theta}{2}\right)}d\rho \right.\notag \\
		%&+ \left.\frac{\varepsilon^{l}}{2\pi} \int_{\frac{\theta}{2}-\pi}^{\pi-\frac{\theta}{2}}|\Phi^\dagger (\varepsilon e^{i\varphi})||\phi(\varepsilon e^{i\varphi})|^{j+k} e^{\varepsilon t_2 |\cos(\varphi)|}d\varphi \right]\\
		\leq \, & \frac{x^{n+1}}{(n+1)!}\sum_{j=0}^{+\infty} \frac{x^j}{j!} \left[ \frac{1}{\pi}\int_\varepsilon^{+\infty}\frac{\left|\Phi^{\dagger}\left(\rho e^{i\left(\pi-\frac{\theta}{2}\right)}\right)\right|\left|\Phi\left(\rho e^{i\left(\pi-\frac{\theta}{2}\right)}\right)\right|^{j+k+n+1}}{\rho^{1-l}}e^{-t_1\rho \cos\left(\frac{\theta}{2}\right)}d\rho \right.\notag \\
		&+ \left.\frac{\varepsilon^{l}}{2\pi} \int_{\frac{\theta}{2}-\pi}^{\pi-\frac{\theta}{2}}|\Phi^\dagger (\varepsilon e^{i\varphi})||\phi(\varepsilon e^{i\varphi})|^{j+k+n+1} e^{\varepsilon t_2 |\cos(\varphi)|}d\varphi \right]
		%\textcolor{red}{d\rho}.
		\label{laststep}
	\end{align}
	where, in the last step, we used \eqref{eq:estIjkl}. The convergence of the series in \eqref{laststep} can be ascertained as in \eqref{248}. Taking the supremum in $[t_1,t_2]$ in \eqref{laststep}, we get \eqref{eq:zero1}.
	% result then follows by letting $x \to 0$ in \eqref{laststep}.
\end{proof}

\subsection{Proof of Proposition \ref{prop:extcont}}
A slight modification of the arguments used in the proofs of Proposition \ref{lem:convtail} and Theorem \ref{thm:smoothfmu} leads to the desired result.
\begin{proof}[Proof of Proposition \ref{prop:extcont}]
	Since $\phi$ is a complete Bernstein function that can be extended by continuity over $\overline{\C(0,\pi)}$ with extension $\phi_+$, then, by using the relation $\overline{\phi(z)}=\phi(\overline{z})$, it is clear that it can be also extended by continuity over $\overline{\C(-\pi,0)}$. If we denote such an extension $\phi_-$, for any $z \in \overline{\C(0,\pi)}$ we get
	\begin{equation*}
		\overline{\phi_+(z)}=\phi_-(\overline{z}).
	\end{equation*}
	The proof is then carried on exactly as in Propositions \ref{prop:intreptheta} and \ref{lem:convtail}, by setting $\theta=0$ (see Figure \ref{fig2}), where we use $\phi^\dagger$ when we integrate over $\ell_1$, $\Gamma_R^+$, $\Gamma_R^-$ and $-\gamma_\varepsilon$, $\phi^\dagger_+$ over $\ell_2$ and $\phi^\dagger_-$ over $\ell_1$.
\end{proof}

\subsection{Proof of Theorem \ref{thm:seriessub}}
In order to prove Theorem \ref{thm:seriessub} we first need to provide an integral representation for $G_\phi$, $g_\phi$ and its derivatives, analogously to what we did for Theorems \ref{thm:smoothfmu} and \ref{thm:seriespi}.
This is done in the following proposition, whose proof is almost identical to the one of Proposition \ref{prop:intreptheta} and thus is omitted.
\begin{prop}
	Let $\Phi$ be the Laplace exponent of a potentially killed subordinator satisfying assumptions \ref{eq:extensionA3} and \eqref{eq:uniformlimcond} for some $\theta \in (0,\pi)$. Fix any $\varepsilon>0$ and let $\gamma_{\varepsilon,\theta}$ be defined as in Proposition \ref{prop:intreptheta}. Then, on $\mathbb{D}$,
	\begin{equation}\label{statementG}
		\begin{split}
			G_\Phi (x,t) \, = \,& \frac{1}{\pi} \int_\varepsilon^{+\infty}\Im\left(\frac{e^{-x \Phi \lb \rho e^{i\left(\pi-\frac{\theta}{2}\right)}  \rb+ t \rho e^{i\left(\pi-\frac{\theta}{2}\right)} }}{\rho} \right) \, d\rho  +\frac{1}{2\pi i}\int_{\gamma_{\varepsilon,\theta}} \frac{e^{-x\Phi(z)+tz}}{z}dz.
		\end{split}
	\end{equation}
	In particular, $G_\Phi \in C^\infty(\mathbb{D})$, $g_\Phi \in C^\infty(\mathbb{D})$ is well defined and for any $k,l \ge 0$ we have
	\begin{equation}\label{statementG2}
		\begin{split}
			\frac{\partial^k}{\partial x^k}\frac{\partial^l}{\partial t^l}g_\Phi (x,t) \, = \,& \frac{1}{\pi} \int_\varepsilon^{+\infty}\Im\left(\left(\Phi\left(\rho e^{i\left(\pi-\frac{\theta}{2}\right)}\right)\right)^k\rho^le^{il\left(\pi-\frac{\theta}{2}\right)}e^{-x \Phi \lb \rho e^{i\left(\pi-\frac{\theta}{2}\right)}  \rb+ t \rho e^{i\left(\pi-\frac{\theta}{2}\right)} } \right) \, d\rho  \\
			&+\frac{1}{2\pi i}\int_{\gamma_{\varepsilon,\theta}} (\Phi(z))^kz^le^{-x\Phi(z)+tz}dz.
		\end{split}
	\end{equation}
\end{prop}
We employ the latter together with Proposition \ref{lem:convtail} in the following proof.

\begin{proof}[Proof of Theorem \ref{thm:seriessub}]
	Let us first consider $G_\Phi$. Starting from \eqref{statementG}, we have, assuming that we can exchange the order of the series and the integral,
	\begin{align}
		G_\Phi(x,t)
		&=\sum_{j=0}^{+\infty}(-1)^j\frac{x^j}{j!}\left[\frac{1}{\pi}\int_{\varepsilon}^{+\infty}\Im\left(\frac{\left(\Phi\left(\rho e^{i\left(\pi-\frac{\theta}{2}\right)}\right)\right)^j}{\rho}e^{t \rho e^{i\left(\pi-\frac{\theta}{2}\right)}}\right)d\rho+\frac{1}{2\pi i}\int_{\gamma_{\varepsilon,\theta}} \frac{(\Phi(z))^j}{z}e^{tz}dz\right]\notag \\
		&=\frac{1}{\pi}\int_{\varepsilon}^{+\infty}\Im\left(\frac{1}{\rho} {e^{t \rho e^{i\left(\pi-\frac{\theta}{2}\right)}}}\right) \, d\rho+\frac{1}{2\pi i}\int_{\gamma_{\varepsilon,\theta}} \frac{1}{z} {e^{t z}} \, dz\notag\\
		&+\sum_{j=1}^{+\infty}(-1)^j\frac{x^j}{j!}\left[\frac{1}{\pi}\int_{\varepsilon}^{+\infty}\Im\left(\frac{\left(\Phi\left(\rho e^{i\left(\pi-\frac{\theta}{2}\right)}\right)\right)^j}{\rho}e^{t \rho e^{i\left(\pi-\frac{\theta}{2}\right)}}\right)d\rho+\frac{1}{2\pi i}\int_{\gamma_{\varepsilon,\theta}} \frac{(\Phi(z))^j}{z}e^{tz}dz\right]\notag\\
		& ={\frac{1}{\pi}\int_{\varepsilon}^{+\infty}\Im\left(\frac{1}{\rho} e^{t \rho e^{i\left(\pi-\frac{\theta}{2}\right)}}\right) \, d\rho+\frac{1}{2\pi i}\int_{\gamma_{\varepsilon,\theta}} \frac{1}{z} e^{t z} \, dz} \notag\\		
		&+\sum_{j=1}^{+\infty}\sum_{k_1+k_2+k_3=j-1}(-1)^j\frac{x^j}{k_1!k_2!(k_3+1)!}q^{k_1}\mathfrak{b}^{k_2}\notag\\
		&\times \left[\frac{1}{\pi}\int_{\varepsilon}^{+\infty}\Im\left(\frac{\left(\Phi^\dagger\left(\rho e^{i\left(\pi-\frac{\theta}{2}\right)}\right)\right)^{k_3+1}}{\rho^{k_3+1-(k_2+k_3)}e^{i(k_3-(k_2+k_3))}}e^{t \rho e^{i\left(\pi-\frac{\theta}{2}\right)}}\right)d\rho+\frac{1}{2\pi i}\int_{\gamma_{\varepsilon,\theta}} \frac{(\Phi^\dagger(z))^{k_3+1}}{z^{k_3+1-(k_2+k_3)}}e^{tz}dz\right]\notag\\
		& { + \sum_{j=1}^{+\infty}(-1)^j\frac{x^j}{j!}\left[\frac{1}{\pi}\int_{\varepsilon}^{+\infty}\Im\left(\frac{\left(q+\mathfrak{b}\rho e^{i\left(\pi-\frac{\theta}{2}\right)}\right)^j}{\rho}e^{t \rho e^{i\left(\pi-\frac{\theta}{2}\right)}}\right)d\rho+\frac{1}{2\pi i}\int_{\gamma_{\varepsilon,\theta}} \frac{(q+\mathfrak{b}z)^j}{z}e^{tz}dz\right]}
		\label{multinom}
	\end{align}
where in the last step we used the multinomial theorem, for $j \ge 1$, to expand $\left( \phi \left( \rho e^{i \left( \pi-\frac{\theta}{2} \right)} \right) \right)^j=\left( q+\mathfrak{b}\rho e^{i \left( \pi-\frac{\theta}{2} \right)}+\phi^\dagger \left( \rho e^{i \left( \pi-\frac{\theta}{2} \right)} \right) \right)^j$. Incorporating the first summand of \eqref{multinom} into the last summation, we achieve
	\begin{align*}
		G_\Phi(x,t)&=\sum_{j=0}^{+\infty}(-1)^j\frac{x^j}{j!}\left[\frac{1}{\pi}\int_{\varepsilon}^{+\infty}\Im\left(\frac{\left(q+\mathfrak{b}\rho e^{i\left(\pi-\frac{\theta}{2}\right)}\right)^j}{\rho}e^{t \rho e^{i\left(\pi-\frac{\theta}{2}\right)}}\right)d\rho+\frac{1}{2\pi i}\int_{\gamma_{\varepsilon,\theta}} \frac{(q+\mathfrak{b}z)^j}{z}e^{tz}dz\right]\\
		&+\sum_{j=1}^{+\infty}\sum_{k_1+k_2+k_3=j-1}(-1)^j\frac{x^j}{k_1!k_2!(k_3+1)!}q^{k_1}\mathfrak{b}^{k_2}\\
		&\times \left[\frac{1}{\pi}\int_{\varepsilon}^{+\infty}\Im\left(\frac{\left(\Phi^\dagger\left(\rho e^{i\left(\pi-\frac{\theta}{2}\right)}\right)\right)^{k_3+1}}{\rho^{k_3+1-(k_2+k_3)}e^{i(k_3-(k_2+k_3))}}e^{t \rho e^{i\left(\pi-\frac{\theta}{2}\right)}}\right)d\rho+\frac{1}{2\pi i}\int_{\gamma_{\varepsilon,\theta}} \frac{(\Phi^\dagger(z))^{k_3+1}}{z^{k_3+1-(k_2+k_3)}}e^{tz}dz\right]\\
		&=\sum_{j=0}^{+\infty}(-1)^j\frac{x^j}{j!}\left[\frac{1}{\pi}\int_{\varepsilon}^{+\infty}\Im\left(\frac{\left(q+\mathfrak{b}\rho e^{i\left(\pi-\frac{\theta}{2}\right)}\right)^j}{\rho}e^{t \rho e^{i\left(\pi-\frac{\theta}{2}\right)}}\right)d\rho+\frac{1}{2\pi i}\int_{\gamma_{\varepsilon,\theta}} \frac{(q+\mathfrak{b}z)^j}{z}e^{tz}dz\right]\\
		&+\sum_{j=1}^{+\infty}\sum_{k_1+k_2+k_3=j-1}(-1)^j\frac{x^j}{k_1!k_2!(k_3+1)!}q^{k_1}\mathfrak{b}^{k_2}\frac{d^{k_2+k_3}}{dt}\mu^{\ast(k_3+1)}(t)\\
		&{=:S_\varepsilon(x,t;q,\mathfrak{b})+\sum_{j=1}^{+\infty}\sum_{k_1+k_2+k_3=j-1}(-1)^j\frac{x^j}{k_1!k_2!(k_3+1)!}q^{k_1}\mathfrak{b}^{k_2}\frac{d^{k_2+k_3}}{dt}\mu^{\ast(k_3+1)}(t)}.
	\end{align*}
	We now evaluate $S_\varepsilon$. Consider the function $F_j(z;t)=\frac{(q+\mathfrak{b}z)^j}{z}e^{tz}$, that is holomorphic on $\C \setminus \{0\}$, Hermitian and admits a simple pole in $0$ with residue $q^j$. We have that
		\begin{align*}
			&\frac{1}{\pi}\int_{\varepsilon}^{+\infty}\Im\left(\frac{\left(q+\mathfrak{b}\rho e^{i\left(\pi-\frac{\theta}{2}\right)}\right)^j}{\rho}e^{t \rho e^{i\left(\pi-\frac{\theta}{2}\right)}}\right)d\rho+\frac{1}{2\pi i}\int_{\gamma_{\varepsilon,\theta}} \frac{(q+\mathfrak{b}z)^j}{z}e^{tz}dz\\
			&=\frac{1}{2\pi i}\left[\int_{\varepsilon}^{+\infty}F_j\left(\rho e^{i\left(\pi-\frac{\theta}{2}\right)};t\right) e^{i\left(\pi-\frac{\theta}{2}\right)}d\rho-\int_{\varepsilon}^{+\infty}F_j\left(\rho e^{i\left(\pi+\frac{\theta}{2}\right)};t\right) e^{i\left(\pi+\frac{\theta}{2}\right)}d\rho+\int_{\gamma_{\varepsilon,\theta}}F_j(z;t)dz\right].
		\end{align*}
		Taking into account the notation in Figure \ref{fig1}, denote by $-\Gamma_{R,\theta}$ the clockwise oriented circular arc joining $F(R)$ and $C(R)$, $-\ell_2$ the oriented segment joining $D(\varepsilon)$ to $C(R)$ and $-\ell_3$ the oriented segment joining $F(R)$ to $E(\varepsilon)$. Denote by $\mathfrak{D}_{R,\theta}$ the domain whose negatively oriented contour is given by $\gamma_{\varepsilon,\theta}$, $-\ell_2$, $-\Gamma_{R,\theta}$ and $-\ell_3$. Then, since $F_j$ is holomorphic on $\mathfrak{D}_{R,\theta}$, that is a simply connected domain, we have, by Cauchy's theorem
		\begin{equation}\label{eq:DReps}
			\int_{-\partial \mathfrak{D}_{R,\theta}}F_j(z;t)dz=0.
		\end{equation}
		On the other hand, if we denote by $\Gamma_R$ and $\gamma_\varepsilon$ the counterclockwise oriented circles of radius respectively $R$ and $\varepsilon$, we have by Cauchy's residue theorem
		\begin{equation*}
			\frac{1}{2\pi i}\int_{\Gamma_{R}}F_j(z;t)dz=\frac{1}{2\pi i}\int_{\gamma_{\varepsilon}}F_j(z;t)dz=q^j.
		\end{equation*}
		Hence, denoting by $A_{R,\varepsilon}$ the annulus delimited by the images of $\Gamma_R$ and $\gamma_\varepsilon$, we have
		\begin{equation}\label{eq:AReps}
			\int_{-\partial A_{R,\varepsilon}}F(z;t)dz=0=\int_{-\partial \mathfrak{D}_{R,\theta}}F_j(z;t)dz,
		\end{equation}
		where $-\partial A_{R,\varepsilon}$ is the negatively oriented contour of $A_{R,\varepsilon}$. If we use the parametrization of $-\ell_2$ and $\ell_3$ as $z=\rho e^{i\left(\pi\mp \frac{\theta}{2}\right)}$ for $\rho \in (\varepsilon,R)$, \eqref{eq:AReps} implies
		\begin{multline}\label{eq:AReps2}
			\int_{-\Gamma_R}F(z;t)dz+\int_{\gamma_\varepsilon}F(z;t)dz
			=\int_{\varepsilon}^{R}F_j\left(\rho e^{i\left(\pi-\frac{\theta}{2}\right)};t\right) e^{i\left(\pi-\frac{\theta}{2}\right)}d\rho\\-\int_{\varepsilon}^{R}F_j\left(\rho e^{i\left(\pi+\frac{\theta}{2}\right)};t\right) e^{i\left(\pi+\frac{\theta}{2}\right)}d\rho+\int_{\gamma_{\varepsilon,\theta}}F_j(z;t)dz+\int_{-\Gamma_{R,\theta}}F(z;t)dz.
		\end{multline}
		Furthermore, if we denote by $-\Gamma^\dagger_{R,\theta}$ the clockwise oriented circular arc joining $C(R)$ to $F(R)$ it is clear that
		\begin{equation*}
			\int_{-\Gamma_R}F(z;t)dz-\int_{-\Gamma_{R,\theta}}F(z;t)dz=\int_{-\Gamma^{\dagger}_{R,\theta}}F(z;t)dz
		\end{equation*}
		hence \eqref{eq:AReps2} leads to
		\begin{multline}\label{eq:AReps3}
			\int_{-\Gamma^\dagger_{R,\theta}}F(z;t)dz+\int_{\gamma_\varepsilon}F(z;t)dz\\
			=\int_{\varepsilon}^{R}F_j\left(\rho e^{i\left(\pi-\frac{\theta}{2}\right)};t\right) e^{i\left(\pi-\frac{\theta}{2}\right)}d\rho-\int_{\varepsilon}^{R}F_j\left(\rho e^{i\left(\pi+\frac{\theta}{2}\right)};t\right) e^{i\left(\pi+\frac{\theta}{2}\right)}d\rho+\int_{\gamma_{\varepsilon,\theta}}F_j(z;t)dz.
		\end{multline}
		By the estimation lemma we have
		\begin{equation*}
			\left|\int_{-\Gamma^\dagger_{R,\theta}}F(z;t)dz\right| \le \theta R \max_{z \in \Gamma^\dagger_{R,\theta}}|F(z;t)| \le \theta (q+\mathfrak{b}R)^je^{-t\cos\left(\frac{\theta}{2}\right)R}
		\end{equation*}
		hence, taking the limit as $R \to +\infty$ in \eqref{eq:AReps3} we get
		\begin{multline}
			2\pi i q^j=\int_{\gamma_\varepsilon}F(z;t)dz\\
			=\int_{\varepsilon}^{\infty}F_j\left(\rho e^{i\left(\pi-\frac{\theta}{2}\right)};t\right) e^{i\left(\pi-\frac{\theta}{2}\right)}d\rho-\int_{\varepsilon}^{\infty}F_j\left(\rho e^{i\left(\pi+\frac{\theta}{2}\right)};t\right) e^{i\left(\pi+\frac{\theta}{2}\right)}d\rho+\int_{\gamma_{\varepsilon,\theta}}F_j(z;t)dz.
		\end{multline}
		Substituting this value into $S_\varepsilon(x,t;q,\mathfrak{b})$ we get \eqref{eq:seriesdistr}.

	The other two series \eqref{eq:seriessub1} and \eqref{eq:seriessub2} are obtained analogously, once one notices that for any $j \ge 0$
	\begin{equation*}
		\frac{1}{\pi}\int_{\varepsilon}^{+\infty}\Im\left(\left(q+\mathfrak{b}\rho e^{i\left(\pi-\frac{\theta}{2}\right)}\right)^je^{t \rho e^{i\left(\pi-\frac{\theta}{2}\right)}}\right)d\rho+\frac{1}{2\pi i}\int_{\gamma_{\varepsilon,\theta}} (q+\mathfrak{b}z)^je^{tz}dz=\frac{1}{2\pi i}\int_{\gamma_\varepsilon}(q+\mathfrak{b}z)^je^{tz}dz=0,
	\end{equation*}
	since $(q+\mathfrak{b}z)^je^{tz}$ is holomorphic in the disc $\{z \in \C: \ |z|<\varepsilon\}$. Finally, the fact that one can actually exchange the series with the integral is proven exactly as in the proof of Theorem \ref{thm:seriespi}.
\end{proof}

	\section{Examples and Further Discussion}\label{sec:ex}
Here we provide computational examples and further discussion of our results.
\subsection{Examples for the results in Section \ref{subsec:L}}
 We start with examples concerning Section \ref{subsec:L}.

Take a Bernstein function with $q=\mathfrak{b}=0$ and a L\'evy measure with a compactly supported density that is regularly varying at $0$ with index $-1-\alpha$ for some $\alpha \in (0,1)$, as for instance
\begin{equation}\label{eq:examplenocomp}
\nu_{\Phi}(dy)=m(y)\ind{y<\frac12}dy, \quad \, m(y)\sim y^{-\alpha-1}, \mbox{ as }x \to 0.	
\end{equation}
We can opt for general regular variation and $\mathfrak{b},q>0$ but we leave it to the interested reader. 
%\ga{Observe also that the related Bernstein function $\Phi$ is not complete, as clearly $m$ is not completely monotone, since it is compactly supported.}
Then,
\[\Phi'''(x)\sim C_3x^{\alpha-3} \text{ and } -\Phi''(x)\sim C_2x^{\alpha-2}, \mbox{ as }x \to \infty\] 
 for positive constants $C_2,C_3$. Clearly,  \eqref{def:condiB} is satisfied. Also, due to the finite support of $\mu_\phi$ we have that $\abs{\phi''(0^+)}<\infty,\phi'''(0^+)<\infty$ and \eqref{def:condiB'} is valid too. Furthermore, for some $C>0$,
\[x^2\Delta(x)=x^2\int_{0}^{\frac1x}y^2 m(y)dy\sim Cx^\alpha, \ \mbox{ as } \ x \to \infty,\]
and  \eqref{def:condiA} holds with $L=\infty$. It is clear that, actually, assumptions \eqref{def:condiA} and \eqref{def:condiB} hold for any Bernstein function $\Phi$ that is regularly varying at $\infty$ with index $\alpha \in (0,1)$. Going back to the example, let $t(x)/x\downarrow 0$. Using $-\Phi''(x)\sim C_2x^{\alpha-2}$, as $x \to \infty$, we establish, for any $k,l\geq 0$,  as $x\to\infty$,
\begin{equation}\label{asymp1ex}
	\begin{split}
		& \sup_{(t,x) \in \mathbb{D}^\prime}\abs{(-1)^k\frac{\sqrt{2\pi C_2}}{C^{k+1}_0}\frac{\sqrt{x}}{c^{k+l+1-\frac{\alpha}{2}}}e^{-ct+x\Phi(c)}\frac{\partial^k\partial ^l }{\partial x^k\partial t^l}\fP{x,t}-1}=\bo{\sqrt{\frac{\ln\lbrb{x}} {a^{\alpha}_*x}}},
	\end{split}
\end{equation}
where we used Theorem \ref{thm:mainL}  with $\Phi(y)=\phi^{\dagger}(y)\sim C_0 y^{\alpha}$, as $y \to \infty$, and the monotonicity in $c$ of the first asymptotic term in \eqref{asymp}. Finally, employing that $\Phi'(y)\sim C_1y^{\alpha-1}$, as $y \to \infty$, we get from \eqref{eq:a*1} 
\[C_1a^{\alpha-1}_*\sim \frac{t(x)}{x} \quad \mbox{ as }x \to \infty \iff a_* \sim \lbrb{\frac{1}{C_1}\frac{t(x)}{x}}^{\frac{1}{\alpha-1}}, \quad \mbox{ as }x \to \infty,\]
and we can plug in this expression in \eqref{asymp1ex}. The speed of convergence of the first asymptotic term in \eqref{asymp} is hence $(t(x)/x)^{\alpha/(2(1-\alpha))}x^{-\frac12}$ which offers faster decay when $\alpha$ approaches $1$ and  we note that it is faster than the second term of the asymptotic in \eqref{asymp} which in this case  is of order $\bo{x^{-1/2}}$.

We continue the example above with an illustration of Theorem \ref{thm:main1}, for which only \eqref{def:condiA} is required.  Take $t(x)=x$ and then $a_*=a_*(x)=(\phi')^{-1}(1)$. Hence, from \eqref{def:fin} we deduce that the explicit in $x$
\begin{equation*}
	\begin{split}
		&\frac{\partial^k\partial ^l }{\partial x^k\partial t^l}\fP{x,t}=\frac{(-1)^k\Phi^\dagger(a_*)\Phi^{k}(a_*)}{\sqrt{2\pi}a^{1-l}_*\sqrt{-\Phi''(a_*)}}\frac{e^{x\lbrb{a_*-\Phi(a_*)}}}{\sqrt{x}} \lbrb{1+\bo{\sqrt{\frac{\ln(x)}{x}}}}, \quad \mbox{ as }x \to \infty.
	\end{split}
\end{equation*}
Since $x\phi'(x)\leq \phi(x)$, see \eqref{it:ineq} of Lemma \ref{lem:Bern}, we double check that $a_*-\phi\lbrb{a_*}\leq 0$. 

Finally, for this example, consider the case when $\limi{x}t(x)/x=\infty=\phi'(0^+)$. Then as above $a_*\sim C (t(x)/x)^{-\frac{1}{1-\alpha}}$  and $x=\so{t(x)}$, as $x \to \infty$, thus
\[-a^2_*\phi''(a_*)x\sim Ca^{\alpha}_*x\sim C't^{-\frac{\alpha}{1-\alpha}}(x)x^{\frac{1}{1-\alpha}}, \quad \mbox{ as }x \to \infty,\]
which goes to infinity if $t(x)=\so{x^{\frac{1}{\alpha}}}$ and the first requirement of \eqref{eq:addCondi} holds. Also, for any $\delta>0$,
\[\limi{x}e^{-\delta x}t^{-\frac{\alpha}{1-\alpha}}(x)x^{\frac{1}{1-\alpha}}=0,\]
for $t(x)=\so{x^{\frac{1}{\alpha}}}$ and the third condition in \eqref{eq:addCondi} holds. Under the same restrictions $\ln(1/a_*)$ is at most of logarithmic growth and the second imposition of \eqref{eq:addCondi} is valid and Theorem \ref{thm:main2} is applicable.

 Let us discuss the route to the derivation of new fine local estimates in the region of the lower envelope. We use the example above with $\alpha=1/2$. In this case one has to consider 
 \[\Pbb{\sigma(x)\leq c\frac{x^2}{\log_2x}},\]
 where $\log_2x=\log\log x, c>0$, see \cite[Chapter III]{bertoinb}. Set $t(x)=c\frac{x^2}{\log_2x}$. Clearly, $t(x)/x\to \infty$ and from above \eqref{def:condiA},\eqref{def:condiB} and \eqref{def:condiB'} hold. Also, 
 \[\phi'(a_*)=\frac{t(x)}{x} \Rightarrow a_*\sim C\frac{\log^2_2 x}{x^2}, \quad \mbox{ as }x \to \infty.\]
 Furthermore, the first requirement of \eqref{eq:addCondi} is fulfilled since  as $x \to \infty$,
 \[-xa_*\phi''(a_*)\sim Cxa^{-\frac12}_*\sim C\log_2(x)\to \infty. \]
The second and third impositions of \eqref{eq:addCondi} are then obvious and Theorem \ref{thm:mainS} holds true and by virtue of its claims it yields local estimates for the densities and all derivatives of the probabilities above.
 
Assume that  $\abs{\phi''(0^+)}<\infty$ and hence $\phi'(0^+)<\infty$. Let us determine the speed by which $t(x)/x$ may approach $\phi'(0^+)$ so that our results hold. First, we note that since $a_*\to 0$ then, as $x \to \infty$,
\[-a_*\phi''(0^+)\sim \phi'(0)-\phi'(a_*)=\phi'(0)-\frac{t(x)}{x}.\]
From Remark \ref{rem:main2} we have to ensure that  $\limi{x}a_*\sqrt{x}=\infty$. Hence, from the last relation we must have
\[\limi{x}\frac{\sqrt{x}}{x\phi'(0)-t(x)}=0.\]
Thus, $\sqrt{x}=\so{x\phi'(0)-t}$, as $x \to \infty$, or alternatively $t<x\phi'(0)-K\sqrt{x}$, for all $K>0$, and all $x$ large enough. This captures the region below that of the central limit theorem  for the density $g_\phi$, see Theorem \ref{thm:mainS}, and therefore we can approximate with high precision, as $x\to\infty,$ quantities of the type
\[\Pbb{\sigma(x)\in\lbbrbb{a,b}}=\int_a^b g_\phi(x,t)dt.\]
\subsection{Examples for the results in Section \ref{subsec:series}}\label{subsec:exseries}
First we develop some explicit examples for Section \ref{subsec:series}.
%As already stated in the previous section, 
It is clear that the representation \eqref{eq:seriesder} provides some completely explicit result if we evaluate the derivatives of the convolution powers $\bar{\mu}_{\Phi}^{\ast n}$. For instance, this is doable for stable subordinators in which case the tail of the L\'evy measure of $\Phi(z)=z^\alpha$ is given by
	\begin{equation*}
		\bar{\mu}_\alpha(t)=\frac{t^{-\alpha}}{\Gamma(1-\alpha)},
	\end{equation*}
	where we use the subscript $\alpha$ instead of $\Phi$ to underline the dependence on the single parameter.
	% since here $\Phi(z)=z^\alpha$. 
	 For these specific tails, it is well-known that
	\begin{equation*}
		\bar{\mu}_\alpha^{\ast (j+1)}(t)=\frac{t^{j-(j+1)\alpha}}{\Gamma(j+1-(j+1)\alpha)}
	\end{equation*}
	and then
	\begin{equation*}
		\dersup{}{t}{j}\bar{\mu}_\alpha^{\ast  j}(t)=\frac{t^{-j\alpha-1}}{\Gamma(-j\alpha)}=\frac{t^{-j\alpha-1}}{\pi}\sin(\pi j\alpha)\Gamma(1+j\alpha),
	\end{equation*}
	where we have used Euler's reflection formula, provided that $\alpha \not \in \mathbb{Q}$, while for $\alpha \in \mathbb{Q}$ we have to pay attention to the case in which $j\alpha$ is an integer, for which it can be simply proven that $\bar{\mu}_\alpha^{\ast j}(t)$ is a monomial of degree less than $j$ and thus $\dersup{}{t}{j}\bar{\mu}_\alpha^{\ast j}(t)=0$ as expected. For $l=0$ substituting this into \eqref{eq:fractureI} and then into \eqref{eq:seriessub1}, we get that
	\begin{equation}\label{eq:seriesstable}
		g_\alpha(x,t)=\sum_{j=1}^{+\infty}(-1)^{j+1}\frac{x^j}{j!}\frac{t^{-j\alpha-1}}{\pi}\sin(\pi j\alpha)\Gamma(1+j\alpha).
	\end{equation}
	This series expansion is well-known in literature, see e.g. \cite[Equation $(7)$]{PG10}. The same argument can be also adopted to obtain the series representation of $f_\alpha$. Indeed, with the same arguments as before
	\begin{equation*}
		\dersup{}{t}{j}\bar{\mu}_\alpha^{\ast(j+1)}(t)=\frac{t^{-(j+1)\alpha}}{\pi(j+1)\alpha}\sin(\pi\alpha(j+1))\Gamma(\alpha(j+1)+1)
	\end{equation*}
	for any $j \ge 0$ and $\alpha \in (0,1)$. Thus, by using \eqref{eq:seriesder}, we get
	\begin{equation}\label{eq:seriesstable123}
		f_\alpha(x,t)=\sum_{j=0}^{\infty}(-1)^{j}\frac{\Gamma(1+(j+1)\alpha)}{\alpha (j+1)!}\frac{\sin(\pi \alpha(j+1))}{\pi}x^j t^{-\alpha(j+1)}.
	\end{equation}
	Such a series can also be deducted by combining \eqref{eq:seriesstable} with the relation
	\begin{equation*}
		f_\alpha(x,t)=\frac{t}{\alpha}x^{-1-\frac{1}{\alpha}}g_\alpha(1,tx^{-\frac{1}{\alpha}}),
	\end{equation*}
	see \cite{meerstra2} for more details. The series representation \eqref{eq:seriesstable123} has also been obtained by means of a limit argument in \cite[Remark 2.3]{kumar}.
	
	A similar approach can be used to deduce some information on the relativistic (or tempered) stable subordinator, i.e. when $\Phi(z)=(\lambda+z)^\alpha-\lambda^\alpha$. In \cite{kumar} the authors provide both an integral and a series representation for the density $f_{\alpha,\lambda}$ of the inverse tempered stable subordinator. In the proof, they exploit the possibility to extend $\phi$ to the whole complex half-plane $\overline{\C(0,\pi)}$ and then they integrate on a keyhole contour centered in $-\lambda$. This seems to be slightly different from our contour, which is not really a keyhole contour and it is always centered in $0$. However, Proposition \ref{prop:extcont} lets us extend the approach to a full keyhole contour. Furthermore, if we choose $\varepsilon<\lambda$, then for any $j \ge 0$ we get
	\begin{equation*}
		\int_{\gamma_\varepsilon}e^{tz}\frac{(\phi(z))^{j+1}}{z}dz=0,
	\end{equation*}
	since the integrand is holomorphic on the disc $\{z \in \C: \ |z|<\varepsilon\}$. Since $\phi_+(\rho)$ is real for $\rho>-\lambda$, we can rewrite
	\begin{align*}
		\dersup{}{t}{j+1}\bar{\mu}_\Phi^{\ast j}(t)&=\frac{1}{\pi}\int_{\lambda}^{+\infty}\Im\left[\frac{(\Phi_+(-\rho))^{j+1}}{\rho}e^{-t\rho}\right]d\rho=\frac{e^{-t\lambda}}{\pi}\int_{0}^{+\infty}\Im\left[\frac{(\Phi_+(\lambda-\rho))^{j+1}}{\rho+\lambda}e^{-t\rho}\right]d\rho.
	\end{align*}
	This leads to
	\begin{align*}
		f_{\Phi}(x,t)&=\frac{e^{-t\lambda}}{\pi}\sum_{j=0}^{+\infty}\frac{x^j}{j!}(-1)^j\int_{0}^{+\infty}\Im\left[\frac{(\Phi_+(\lambda-\rho))^{j+1}}{\rho+\lambda}e^{-t\rho}\right]d\rho\\
		&=\frac{e^{x\lambda^\alpha-t\lambda}}{\pi}\sum_{j=0}^{+\infty}\frac{x^j}{j!}(-1)^j\int_{0}^{+\infty}\Im\left[\frac{(\Phi_+(\lambda-\rho))^{j+1}}{\rho+\lambda}e^{-x(-\rho)^\alpha-t\rho}\right]d\rho\\
		&=\frac{e^{x\lambda^\alpha-t\lambda}}{\pi}\int_0^{+\infty}\frac{e^{x\rho^\alpha \cos(\alpha \pi)-t\rho}}{\rho +\lambda}[\rho^\alpha \sin(\alpha \pi -x\rho^\alpha \sin(i\alpha \pi))+\lambda^\alpha\cos(\alpha \pi)\sin(x\rho^\alpha \sin (i\alpha \pi))]d\rho,
	\end{align*}
	where the last equality follows by simple algebraic manipulations. It is the integral representation of \cite[Theorem 2.1]{kumar} and thus, arguing as in \cite[Proposition 2.1]{kumar}, we get the series representation
	\begin{equation}\label{seriesinvrel}
		\begin{split}
			f_{\Phi} (x,t) \, = \, \frac{e^{\lambda^\alpha x}}{\pi} \sum_{j=0}^{+\infty}  \frac{(-1)^j x^j}{j!} \lambda^{\alpha(j+1)} &\left[ \Gamma (1+\alpha(j+1)) \Gamma (-\alpha (k+1), \lambda t) \sin ((j+1)\alpha \pi) \right. \\
			& \left. - \Gamma (1+\alpha j) \Gamma (-\alpha j, \lambda t) \sin (j\alpha \pi) \right],
		\end{split}
	\end{equation}
	where $\Gamma(x,y)$ is the upper-incomplete Gamma function.
	
	With the same arguments, we can use Proposition \ref{prop:extcont} to obtain the integral representation of the density of the inverse Gamma subordinator (i.e. the case $\Phi(z)=\log(1+z)$), as in \cite[Proposition 1]{kumar2}.

	In \cite{bur}, the authors studied a special class of Thorin subordinators. Let us consider now an example taken from this paper. Fix $\alpha \in (0,1)$, let us consider the subordinator $\sigma$ whose Laplace exponent is given by $\Phi(z)=\varphi(z)-\varphi(0)-z$, where $\varphi(z)$ is the unique solution of $\varphi(z)-(\varphi(z))^\alpha=z$ for $z \ge 0$. Clearly $\Phi(0)=0$ and then $q=0$. Furthermore, by \cite[Theorem 1]{bur}, we know that $\mathfrak{b}=0$ and, by \cite[Proposition $5$]{bur}, we get that $g_\Phi$ is well-defined. We can rewrite \cite[Equation (21)]{bur} in the following form, for $x,t>0$,
	\begin{equation}\label{eq:seriesgphi}
		g_\Phi(x,t)=\sum_{j=1}^{+\infty}\frac{x^j}{j!}(A_j(t)+B_j(t)+C_j(t)),
	\end{equation}
	where
	\begin{align*}
		A_j(t)&=j\sum_{n=1}^{+\infty}(-1)^{n+1}\frac{\Gamma(1+n\alpha)}{n!}\sin(\pi n \alpha)t^{-n\alpha+n}\\
		B_j(t)&=j(j-1)\sum_{n=1}^{+\infty}(-1)^{n+1}\frac{(n+1)\Gamma(1+n\alpha)}{n!}\sin(\pi n \alpha)t^{-n\alpha+n-1}\\
		C_j(t)&=\begin{cases}
			0, & j=1,2,\\	
			\displaystyle \sum_{k=0}^{j}\frac{j!}{(j-k)!}\sum_{n=k+1}^{+\infty}(-1)^{n+1}\binom{n+1}{k+2}\frac{\Gamma(1+n\alpha)}{n!}\sin(\pi n \alpha)t^{-n\alpha+n-2-k}, & j \ge 3.
		\end{cases}
	\end{align*}
	In particular, comparing \eqref{eq:seriessub1} and \eqref{eq:seriesgphi}, we get, by means of \eqref{eq:fractureI},
	\begin{equation*}
		\frac{d^j}{d t^j}\bar{\mu}_\Phi^{\ast j}(t)=A_j(t)+B_j(t)+C_j(t), \quad t>0, \ j \ge 1.
	\end{equation*}
	For $j=1$, we get again \cite[Equation (22)]{bur}. Moreover, setting
	\begin{align*}
		\overline{A}_j(t)&=j\sum_{n=1}^{+\infty}(-1)^{n}\frac{\Gamma(1+n\alpha)}{n!(n-n\alpha+1)}\sin(\pi n \alpha)t^{-n\alpha+n+1}\\
		\overline{B}_j(t)&=j(j-1)\sum_{n=1}^{+\infty}(-1)^{n}\frac{(n+1)\Gamma(1+n\alpha)}{n!(n-n\alpha)}\sin(\pi n \alpha)t^{-n\alpha+n}\\
		\overline{C}_j(t)&=\begin{cases}
			0, & j=1,2,\\	
			\displaystyle \sum_{k=0}^{j}\frac{j!}{(j-k)!}\sum_{n=k+1}^{+\infty}(-1)^{n}\binom{n+1}{k+2}\frac{\Gamma(1+n\alpha)}{n!(n-n\alpha-1-k)}\sin(\pi n \alpha)t^{-n\alpha+n-1-k}, & j \ge 3,
		\end{cases}
	\end{align*}
	we have
	\begin{equation*}
		\frac{d^j}{d t^j}\bar{\mu}_\Phi^{\ast j+1}(t)=\bar{A}_{j+1}(t)+\bar{B}_{j+1}(t)+\bar{C}_{j+1}(t), \quad t>0, \ j \ge 0,
	\end{equation*}
	and then we achieve the density of the inverse subordinator $L_\Phi$ by means of \eqref{eq:seriesder}.
	
	Finally, we show a specific application of Theorem \ref{behavatzero} to the context of time-nonlocal equations. For example, in \cite{moving} the following assumption is used to prove the main result:
	\begin{itemize}
		\item	For any $0<a \le b$ there exists $\delta_{a,b}>0$ and a function $F_{a,b}:(0,\delta_{a,b}) \to (0,+\infty)$ such that 
		\begin{equation*}
			\int_0^{\delta_{a,b}} x^{-\frac{1}{2}}F_{a,b}(x)dx<\infty
		\end{equation*} 
		and, for any $x \in (0,\delta_{a,b})$ and $t \in [a,b]$,
		\begin{equation*}
			\left|\pd{f_{\Phi}}{x}(x,t)\right| \le F_{a,b}(x).
		\end{equation*}
	\end{itemize}
	It is clear that we can Theorem \ref{behavatzero} for $n=0$ and since the remainder is locally uniform in $t$, the latter condition is verified with $F_{a,b}$ being independent of $x$ if $\phi$ satisfies \eqref{eq:extensionA3} and \eqref{eq:uniformlimcond}. In particular, the latter holds whenever (but not exclusively if) $\phi$ is a complete Bernstein function, as discussed in Section \ref{discussionassumptions}.

	\section*{Acknowledgements}
	The authors are gratefull the anonymous referees whose remarks and suggestions have considerably improved a previous draft of the paper.

	The author Giacomo Ascione has been partially supported by the MIUR-PRIN 2022 project “Anomalous Phenomena on Regular and Irregular Domains: Approximating Complexity for the Applied Sciences”, no. 2022XZSAFN.
	Mladen Savov acknowledges - “This study is financed by the European Union-NextGenerationEU, through the National Recovery and Resilience Plan of the Republic of Bulgaria, project No BG-RRP-2.004-0008".
	The author Bruno Toaldo has been partially supported by the MIUR-PRIN 2022 project “Non-Markovian dynamics and non-local equations”, no. 202277N5H9.
	The authors Giacomo Ascione and Bruno Toaldo are partially supported by Gruppo Nazionale per l’Analisi Matematica, la Probabilit\`{a} e le loro Applicazioni (GNAMPA-INdAM). 
	
	% !TeX encoding = UTF-8
% !TeX spellcheck = en_GB

\vspace{1cm}

\end{document}